\documentclass{amsart}
\usepackage{amssymb,amsmath,amscd,amsthm}
\usepackage{amsfonts,epsfig,latexsym,graphicx,amssymb}
\usepackage{color}
\usepackage{enumerate}
\usepackage{todonotes}

\usepackage{hyperref}

\date{\today}

\newcommand{\Z}{{\mathbb Z}}
\newcommand{\R}{{\mathbb R}}

\newcommand{\N}{{\mathbb N}}

\newcommand{\mK}{\mathcal{K}}
\newcommand{\mA}{\mathcal{A}}
\newcommand{\Cn}{\overline{M}}

\newcommand{\OZ}{\Omega_{\Z}}
\newcommand{\Es}{\mathcal{E}}

\newcommand{\SL}{{\mathrm{SL}}}
\newcommand{\Rot}{{\mathrm{Rot}}}

\newcommand{\bo}{\omega}
\newcommand{\bv}{\bar{v}}

\newcommand{\tf}{\tilde f}

\newcommand{\tx}{\tilde x}

\newcommand{\ki}{j}

\newcommand{\Ind}{\mathbf{1}}
\newcommand{\beps}{\overline{\eps}}

\newcommand{\mjump}{\bar{m}}
\newcommand{\mV}{m'}
\newcommand{\mmax}{\widehat{m}}
\newcommand{\mW}{\mjump}

\newcommand{\vu}{w}

\newcounter{mit}
\newenvironment{mainit}
	{\setcounter{mit}{0}\begin{itemize}}
	{\end{itemize}}
\newcommand{\mitem}{\refstepcounter{mit}\item[\themit]}
\renewcommand{\themit}{\textbf{\Roman{mit}}}

\newcounter{Bit}
\newenvironment{Bitemize}
	{\setcounter{Bit}{0}\begin{itemize}}
	{\end{itemize}}
\newcommand{\bitem}{\refstepcounter{Bit}\item[\theBit]}
\renewcommand{\theBit}{(B\arabic{Bit})}

\newtheorem{theorem}{Theorem}[section]
\newtheorem{lemma}[theorem]{Lemma}
\newtheorem{prop}[theorem]{Proposition}
\newtheorem{coro}[theorem]{Corollary}

\theoremstyle{definition}
\newtheorem{remark}[theorem]{Remark}

\newtheorem{example}[theorem]{Example}
\newtheorem{defi}[theorem]{Definition}

\newcommand{\Sc}{{\mathbb S}^1}

\newcommand{\Lx}{L}
\newcommand{\La}{L_p}
\newcommand{\ep}{\varepsilon'}

\def\dist{{\rm dist}}
\def\supp{\mathop{\rm supp}}

\newcommand{\bi}{{\bf i}}

\def\N{{\mathbb N}}

\newcommand{\E}{{\mathbb E}\,}
\newcommand{\Prob}{{\mathbb P}\,}
\def\P{\Prob}

\def\be{\begin{equation}}
\def\ee{\end{equation}}

\newcommand{\eps}{{\varepsilon}}

\newcommand{\const}{{\rm const}}

\newcommand{\mgr}{\mu}
\newcommand{\msp}{\nu}

\begin{document}

\title[Non-stationary Anderson Localization% and Parametric Furstenberg Theorem
]{Non-stationary Anderson Localization}% and Parametric Furstenberg Theorem}% \\ on random matrix products}

%\title[Non-stationary version of Furstenberg Theorem]{Non-stationary version of Furstenberg Theorem\\ on Random Products of $\SL(m, \mathbb{R})$ matrices}

\author[A.\ Gorodetski]{Anton Gorodetski}

\address{Department of Mathematics, University of California, Irvine, CA~92697, USA}

\email{asgor@uci.edu}

\thanks{A.\ G.\ was supported in part by NSF grant DMS--2247966.}    %DMS--1855541.} % NSF grant DMS--1301515.}

\author[V. Kleptsyn]{Victor Kleptsyn}

\address{CNRS, Institute of Mathematical Research of Rennes, IRMAR, UMR 6625 du CNRS}

\email{victor.kleptsyn@univ-rennes1.fr}

\thanks{V.K. was supported in part by ANR Gromeov (ANR-19-CE40-0007) and by Centre Henri Lebesgue (ANR-11-LABX-0020-01)}

\begin{abstract}
We consider discrete Schr\"odinger operators on $\ell^2(\mathbb{Z})$ with bounded random but not necessarily identically distributed values of the potential.
%The distribution at a given site is not assumed to be absolutely continuous (or to contain an absolutely continuous component).
We prove spectral localization (with exponentially decaying eigenfunctions) {as well as dynamical localization for this model}.
%An important ingredient of the proof is a non-stationary analog of the Furstenberg Theorem on random matrix products, which is also of independent interest.
An important ingredient of the proof is a non-stationary version of the parametric Furstenberg Theorem on random matrix products, which is also of independent interest.
\end{abstract}

\maketitle

\tableofcontents

\section{Introduction}

In this paper we prove spectral and dynamical localization for 1D discrete Schr\"odinger operators with potential given by independent but not necessarily identically distributed random variables. In particular, the setting in which our theorem is applicable includes non-stationary Anderson-Bernoulli Model; the latter leads to several interesting examples that we discuss in the appendix. In order to prove it, we establish a non-stationary parametric version of the Furstenberg Theorem on random matrix products.

\subsection{Anderson Localization}

The 1977 Nobel Prize in Physics was awarded to P.\,W.\,Anderson, N.\,F.\,Mott, and J.H.\,van\,Vleck ``for their fundamental theoretical investigations of the electronic structure of magnetic and disordered systems''. One of the main phenomena that contributed to the award was the suppression of electron transport due to disorder, which is nowadays called {\it Anderson localization}. Since then Anderson localization attracted enormous amount of attention and was heavily studied both by physicists and mathematicians. For a brief survey of the history of the subject from a perspective of a physicist see \cite{LTW}.

Mathematically Anderson localization can be formulated either as spectral localization (sometimes the term ``Anderson localization'' in math literature is specifically attributed to the spectral localization with exponentially decreasing eigenfunctions) or as dynamical localization, which is closer to physical intuition. The very first rigorous results related to Anderson localization were obtained by Goldsheid, Molchanov, Pastur \cite{GMP} in 1977 and Molchanov \cite{Mol} in 1978. In 1980, Kunz and Souillard \cite{KuS} proved localization for 1D discrete Schr\"odinger operators with  the potential given by independent random variables with nice densities. Localization for 1D Anderson-Bernoulli model turned out to be a harder problem, and was established by Carmona, Klein, Martinelli \cite{CKM} in 1987 (see also \cite{DSS} for the continuum 1D Anderson-Bernoulli case).

In the multidimensional lattice case, the key original articles are
those of Fr\"ohlich and Spencer \cite{FS}, Martinelli and Scoppola \cite{MS}, Simon and Wolff \cite{SW}, Kotani and Simon \cite{KotS}, Delyon, Levy, Souillard \cite{DLS},  von Dreifus and Klein \cite{vDK}, and Aizenman and Molchanov \cite{AM}. Once again, Anderson-Bernoulli case is essentially harder; for the first results on multi-dimensional Anderson localization (at the lower edge of the spectrum) in continuum Anderson-Bernoulli case see \cite{B1}, \cite{BK}, \cite{B2}. In the discrete case similar results were obtained only recently, see \cite{DSm}, \cite{LZ}.

For the detailed description of the existing methods and results see the classical \cite{BL}, \cite{CFKS}, \cite{CL}, \cite{FP} and more recent \cite{AW}, \cite{DKKKR} monographs, as well as lecture notes and surveys \cite{D}, \cite[Section 4]{D16}, \cite{His}, \cite{Hu}, \cite{Kir}, \cite{Sp1}, \cite{Sp2}, \cite{Sp3}, \cite{S}, \cite{St}, \cite{St2}.

%
%\
%
%************************
%
%\
%
%Three most straightforward ways to generalize the model:
%
%1) to consider Schr\"odinger operator not on $\ell^2(\mathbb{Z})$, but  on $\ell^2(\mathbb{Z}^2)$ or $\ell^2(\mathbb{Z}^3)$   -- the most challenging and physically relevant question;
%
%2) to remove the assumption on regularity of the potential, e.g. to consider random potential that can take only one of two possible  values (so called {\it Anderson-Bernoulli} model);
%
%3) to replace the assumption that potential is given by i.i.d. random variables by non-stationary choice of independent random variables.
%
%P.W.Anderson, 1958 - introduced, Nobel Prize, ...   Mention the book "50 years of Anderson Localization"
%
%History (who proved what under what assumptions), including higher dimensional cases. Mention that Anderson-Bernoulli is harder, and is still open in discrete higher dimensional case. Mention two major methods (fractional moments and multiscale analysis).

Let us now state our main results. Consider discrete Schr\"odinger operators $H$ acting on $\ell^2(\Z)$ via
\begin{equation}\label{e.oper}
[H u](n) = u(n+1) + u(n-1) + V(n) u(n).
\end{equation}
We will assume that $\{V(n)\}$ are independent (but not necessarily identically distributed) random variables, distributed with respect to some compactly supported non-degenerate (support contains more than one point) probability measures $\{\mgr^\#_n\}$. Notice that we do not require the distributions $\mgr^\#_n$ to be continuous; in particular, the non-stationary Anderson-Bernoulli model (when the distribution $\mgr^\#_n$ is supported on two different values that can depend on $n$) is included in our setting. Denote
$$
\P=\prod_{n=-\infty}^\infty \mgr^\#_n
$$

\begin{theorem}[Spectral Anderson Localization, 1D]\label{t.al}
Suppose the potential $\{V(n)\}$ of the operator $H$ given by (\ref{e.oper}) is random and defined by the independent random variables defined by  distributions $\{\mgr^\#_n\}$ such that

\vspace{5pt}

 1) $\supp \mgr^\#_n\subseteq [-K, K]$;

\vspace{5pt}

 2) $\text{Var}\,(\mgr^\#_n)>\eps$,

\vspace{5pt}

  \noindent where $\eps>0, K<\infty$ are some uniform constants. Then the spectrum of the operator $H$ is $\P$-almost surely pure point, with exponentially decreasing eigenfunctions. The same statement holds for spectrum of discrete Schr\"odinger operator with non-stationary random potential on $\ell^2(\mathbb{N})$ with the Dirichlet boundary condition.
\end{theorem}

Moreover, a stronger version of localization, namely dynamical localization, holds for non-stationary Anderson Model. A self-adjoint operator $H:\ell^2(\mathbb{Z})\to \ell^2(\mathbb{Z})$ has dynamical localization if for any $q>0$ one has
$$
\sup_{t}\sum_{n\in \mathbb{Z}}(1+|n|)^q|\langle \delta_n, e^{-itH}\delta_0\rangle|<\infty.
$$
We will show that slightly stronger version of dynamical localization holds in our setting.
\begin{defi}
Let $H$ be a self-adjoint operator on $\ell^2(\mathbb{Z})$. The operator $H$ has {\it semi-uniform dynamical localization (SUDL)} if there is $\alpha>0$ such that for any $\xi>0$ there is a constant $C_\xi$ so that for all $q, m\in \mathbb{Z}$
$$
\sup_t|\langle\delta_q, e^{-itH}\delta_m\rangle|\le C_\xi e^{\xi |m|-\alpha|q-m|}.
$$
\end{defi}
In fact, we are going to establish a different property, {\it SULE}.
\begin{defi}
A self-adjoint operator $H:\ell^2(\mathbb{Z})$ has {\it semi-uniformly localized eigenfunctions (SULE)}, if $H$ has a complete set $\{\phi_n\}_{n=1}^\infty$ of orthonormal eigenfunctions, and there is $\alpha>0$ and $\mmax_n\in \mathbb{Z}$, $n\in \mathbb{N}$, such that for each $\xi>0$ there exists a constant $C_\xi$ so that
\begin{equation}\label{eq:SULE}
 |\phi_n(m)|\le C_\xi e^{\xi|\mmax_n|-\alpha|m-\mmax_n|}
\end{equation}
 for all $m\in \mathbb{Z}$ and $n\in \mathbb{N}$.
\end{defi}
Theorem 7.5 from \cite{DJLS} claims that SULE $\Rightarrow$ SUDL, and for operators with simple spectrum SUDL $\Rightarrow$  SULE.

\begin{theorem}[Dynamical Localization, 1D]\label{t.dl}
Under assumptions of Theorem \ref{t.al}, $\P$-almost surely
the operator $H$ has SULE and, hence, SUDL.
\end{theorem}

A specific model where Theorems \ref{t.al} and \ref{t.dl} are applicable is a Schr\"odinger operator with a fixed background potential and iid random noise. Namely, the following holds:

\begin{coro}
Suppose $\{V_{back}(n)\}$ is a fixed bounded sequence of real numbers that we will refer to as a background potential, and $\{V_{rand}(n)\}$ is a random sequence given by an iid sequence of random variables defined by a compactly supported distribution. Then almost surely the operator (\ref{e.oper}) with the potential $V(n)=V_{back}(n)+V_{rand}(n)$ has pure point spectrum with exponentially decaying eigenfunctions, and satisfies SULE and SUDL.
\end{coro}

\begin{remark}
\begin{enumerate}[(a)]\

\item Notice that in general spectral localization does not imply dynamical localization. The classical example is given by the random dimer model \cite{DeG}, \cite{JSS}.
\item In the case when the distributions $\{\mgr^\#_n\}$ are absolutely continuous, the result stated in Theorem \ref{t.al} follows from Kunz -- Souillard method \cite{KuS} (see also \cite{Sim1} for alternative presentation and another application of Kunz -- Souillard method). In the case when the distributions $\{\mgr^\#_n\}$ are H\"older continuous, multiscale analysis method should be applicable \cite{Kl} (see \cite{Kl2} for a detailed discussion of the method). But in the case of Anderson-Bernoulli  potential (given by i.i.d. random variables that can take only finite number of values) most previously existing proofs \cite{BDFGVWZ, CKM,  GK, JZh} used Furstenberg Theorem \cite{Fur1, Fur2, Fur3} on random matrix products and positivity of Lyapunov exponent (see  \cite{SVW}  for a proof that uses a different approach, but also does not cover non-stationary case), and therefore could not be adapted to the non-stationary case. Theorem \ref{t.al} closes this gap and covers all the cases in a uniform fashion.
\item The proof of localization in the case of potential given by i.i.d. random variables  presented in \cite{CKM} does not require the potential to be bounded, it requires only existence of finite momenta (see also \cite{Ra}). While in Theorem \ref{t.al} we require uniform boundedness of the potential to reduce technical difficulties, we expect that with some extra effort our methods can be extended to cover an unbounded case as well.
\item  In the paper \cite{Kl3} the so called ``crooked'' Anderson Model (which can be considered as an analog of non-stationary random case) for continuous case is considered. Localization is proved under the H\"older continuity assumption on distributions.
\item In \cite{BMT, Hur, KLS, Sim1} the authors also consider models where the values of the potentials are random independent, but not identically distributed. There the randomness decays at infinity, and the focus is mainly on the rate of decay that still lead to localization or already insufficient to produce localization.
\item Ergodic Schr\"odiger operators with iid random noise are included into our setting. Questions on topological properties of the spectrum of these operators also attracted considerable amount of attention lately \cite{ADG, DFG, DG2, Wood}.
\item While the results of this paper imply that stationary and non-stationary Anderson Models have similar behaviour in terms of spectral type, there is a huge difference between stationary and non-stationary cases in terms of topological properties of the spectrum. In particular, the spectrum of a stationary Anderson Model is always a finite union of intervals, while in the non-stationary case the almost sure essential spectrum does not have to have dense interior. We construct a relevant example in Appendix \ref{a.1} below.
%\item DISCUSS TRIMMED ANDERSON MODEL AND DIMER MODEL  \cite{DeG}, \cite{JSS}
\end{enumerate}
\end{remark}

The crucial ingredient of the proof of Theorems \ref{t.al} and \ref{t.dl} is the parametric version of the non-stationary Furstenberg Theorem that we discuss below. In the stationary case the parametric Furstenberg Theorem is provided in \cite{GK}. There, in order to demonstrate the power of the developed techniques, we gave a geometrical proof of the spectral localization in 1D Anderson Model (including Anderson-Bernoulli case). The proofs of Theorems \ref{t.al} and  \ref{t.dl} are based  on similar (adapted to the non-stationary case) arguments.

%\newpage
\subsection{Parametric non-stationary Furstenberg Theorem}\label{ss.introparamfurst}

To prove Theorems~\ref{t.al} and \ref{t.dl} one needs to study the properties of the products of corresponding transfer matrices, and the way those products behave for different values of the energy. Motivated by this model, here we consider random products of independent but not identically distributed matrices from $\SL(2, \mathbb{R})$ that depend on a parameter. In other words, instead of one matrix in a random product
%Assume now, that all the matrices depend on a parameter, that is,
 we are working with a map $A(\cdot)$ from some compact interval $J=[b_{-}, b_{+}]\subset \R$ to $\SL(2,\R)$. We assume that all these maps are $C^1$; a random matrix, depending on a parameter, is therefore given by a measure on the space $\mA:=C^1(J,\SL(2,\R))$.  %, where $J\subset \R$ is a compact interval of parameters.
For any such measure $\mgr$ on $\mA$ and any individual parameter value $a\in J$ we can consider the distribution of $A(a)$, that is a measure on~$\SL(2,\R$); we denote this measure~$\mgr^a$.

A (non-stationary) product of random matrices, depending on a parameter, is given by a sequence of measures $\mgr_n$ on $\mA$. We assume that all these measures belong to some compact set $\mK$ of measures on $C^1(J,\SL(2,\R))$, i.e. $\mgr_n\in \mK$ for all $n\in \mathbb{N}$.

We impose the following assumptions:
\begin{Bitemize}
\bitem\label{B:Furstenberg} \textbf{Measures condition}: for any measure $\mgr\in\mK$ and any $a\in J$, there are no Borel probability measures $\msp_1$, $\msp_2$ on $\mathbb{RP}^{1}$ such that $(f_A)_*\msp_1=\msp_2$ for $\mgr^a$-almost every matrix $A\in \SL(2, \mathbb{R})$.
 %law $\mgr^a$ of~$A(a)$, where $A$ is distributed w.r.t.~$\mgr$.
\bitem\label{B:C1} \textbf{$C^1$-boundedness}: there exists a constant $\Cn$ such that any map $A(\cdot) \in C^1(J,\SL(2,\R))$
from the support of any $\mgr\in\mK$ has $C^1$-norm at most $\Cn$.
\bitem\label{B:Monotonicity} \textbf{Monotonicity}: there exists $\delta>0$ such that for any $\mgr\in\mK$, any map $A(\cdot)$ from the support of $\mgr$ and any $a_0\in J$ one has
$$
%\forall a_0\in J, \quad
\forall v\in \R^2\setminus \{0\} \quad \left.\frac{d\arg (A(a)(v))}{da}\right|_{a_0}  >\delta.
$$
\end{Bitemize}
These are the exact analogues of the assumptions A1, A2 and A4 from~\cite{GK}.

At the same time, we do not impose any assumptions on absence of uniform hyperbolicity (similar to A3 in~\cite{GK}). Instead we modify some of the conclusions of the theorems. Examples \ref{ex:J-1} and \ref{ex:J-2} below illustrate the necessity of such modifications.

\begin{remark}\label{r:groups}
One can replace the assumptions \ref{B:Furstenberg}--\ref{B:Monotonicity} % Standing Assumption
 by more general ones. Namely, it is enough to assume that for some given $k$ the conditions  \ref{B:Furstenberg}--\ref{B:Monotonicity} hold for laws of compositions
\[
A: J \mapsto \SL(2,\R), \quad A(a):=A_k(a) \dots A_1(a),
\]
where $A_i$ are distributed w.r.t. some $\mgr_i\in\mK$. This is useful, in particular, when working with Schr\"odinger cocycles (in this case, one takes $k=2$).
\end{remark}

Let us introduce some notation. Given a sequence of measures $\mgr_n\in\mK$, satisfying the assumptions above, we consider the probability space~\mbox{$\Omega:=\mA^{\N}$}, equipped with the measure $\Prob:=\prod_{n} \mgr_n$. For any point \mbox{$\omega=(A_1,A_2,\dots)\in\Omega$}
and a parameter $a\in J$ we denote
$$
T_{n,a, \omega}:=A_n(a)\dots A_1(a).
$$
For each $A\in \mA$ denote by $f_{A, a}:\Sc\to\Sc$ the projectivization of the map $A(a)\in \SL(2, \mathbb{R})$, and choose its lift  $\tilde f_{A, a}:\mathbb{R}\to \mathbb{R}$
 so that it depends continuously on $a$, and, to make this choice unique, satisfy  $\tilde f_{A, b_{-}}(0)\in [0,1)$, where $b_-$ is the left endpoint of the  interval of parameters~$J$.

Also, following the same notation as in~\cite{GK}, we denote by $f_{n, a, \omega}:\Sc\to\Sc$  the projectivizations of the maps~$T_{n,a, \omega}$, and choose their lifts $\tf_{n,a, \omega}:\R\to\R$ as $\tf_{n,a, \omega}=\tf_{A_n,a}\circ\ldots\circ \tf_{A_1, a}$.

Fix a point $x_0\in\Sc$, corresponding to some unit vector $v_0$, %=\left({1\atop 0}\right)$,
and its lift $\tx_0\in\R$. For an interval $I\subset J$, $I=[a',a'']$, define
$$
R_{n,\omega}(I):=\tf_{n,a'', \omega}(\tx_0)-\tf_{n,a', \omega}(\tx_0);
$$
in other words, as the parameter varies on $I$, the $n$-th (random) image of the initial vector~$v_0$ turns by the angle $\pi R_{n,\omega}$ in the positive direction.

In the stationary setting, $R_{n,\omega}(I)$ is almost surely bounded as $n\to\infty$ if and only if the random product is uniformly hyperbolic for any internal point $a$ of~$I$. Indeed, a dynamical analog of Johnson's Theorem implies uniform hyperbolicity on any open interval where the \emph{random rotation number} is locally constant (see \cite{J} for the original Johnson's Theorem, and \cite[Proposition C.1]{ABD}, \cite[Theorem A.9]{GK} for the dynamical analog that we refer to).

However, in the non-stationary case there is no notion of a random rotation number (in the same way as there is no well defined Lyapunov exponent), so we have to choose a more direct approach. Namely, we introduce the following

%\todo[inline]{Say a few words that this definition is to replace the (non-)hyperbolicity property, that is not available in a non-stationary context.}

\begin{defi}
An open interval $I\subset J$ is \emph{inessential}, if almost surely
$$
\sup_n R_{n, \omega}(I) <\infty.
$$
The same applies to the intervals $I\subset J=[b_-,b_+]$ of the form $[b_-,a)$ and $(a,b_+]$.

The \emph{essential} set $\Es$ is defined as a complement to the union of all (open in $J$) inessential intervals:
$$
\Es=J \setminus \left(\bigcup_{I \ \text{is} \ \text{inessential}} I\right).
$$
\end{defi}
\begin{remark}
This definition is independent on the choice of the initial vector~$v_0$: using another vector cannot change the rotation angles by more than a complete turn, and hence cannot change the boundedness of the sequence $R_{n,\omega}(I)$.

Also, notice that due to the Kolmogorov's  $0-1$ law for any interval $I$ either the sequence~$R_{n,\omega}(I)$ is almost surely bounded, or it is almost surely unbounded. Indeed, its boundedness does not depend on any finite number $n_0$ of the first factors~$A_1,\dots, A_{n_0}$ in the product. %Therefore the essential set $\Es\subset J$ is a non-random set.
\end{remark}

The choice of terminology is related to the fact that for the random (non-stationary) Schr\"odinger operators and the corresponding products of transfer matrices the defined essential set $\Es$ turns out to be exactly the almost-sure essential spectrum.

%The following examples illustrate the specifics of the non-stationary setting, both in the case of random matrix products and in the case of random Schr\"odinger operators.

Let us  give an example that shows that, contrary to the stationary case, a full turn of the image of some vector does not forbid an interval to be inessential.
\begin{example}\label{ex:J-1}
Consider a stationary random parameter-dependent product $A_n(a)\dots A_1(a)$ that is uniformly hyperbolic for $a$ from on the parameter interval $I=[0,1]$. Now, modify only one (random) factor $A_1$ by taking its composition with a rotation by $2\pi a$:
$$
B_1(a)=\Rot_{2\pi a}\cdot A_1(a), \ \  \text{\rm and}\ \   B_j(a)=A_j(a)\  \ \text{\rm if} \ \  j>1.
$$
Then for the random product $B_n(a)\dots B_1(a)$ we have $R_{n,\omega}(I)\ge 1$, though the interval $I$ is inessential.
\end{example}

Similar example can be given in the random Schr\"odinger operators setting:
\begin{example}
Consider the Anderson-Bernoulli potential $V(n)$, $n\in \mathbb{Z}$, that with probabilities $1/2$  takes values 0 or $V_n$,  where
$$
V_n= \begin{cases}
1 & n \neq 0, \\
100, & n=0,
\end{cases}
$$
and the corresponding random Schr\"odinger operator~$H_V$.
Then, the \emph{essential} spectrum of this operator (almost surely) is the interval $[-2,3]$, and if $V(0)=0$, it is also the full spectrum.
However, if the random value $V(0)=100$, the spectrum of $H_V$ contains an additional eigenvalue. % somewhere in the interval~$[98,102]$.
Thus, the non-essential parts of the spectrum may be random, and therefore one cannot talk about ``almost sure spectrum'' in the non-stationary case.
\end{example}

%\subsection{Infinite product version}

We are now ready to state the non-stationary parametric theorems for infinite products. Denote $L_n(a):=\E \log \|T_{n,\omega,a}\|$.

\begin{theorem}[Non-stationary parametric version of Furstenberg Theorem]\label{t:product}
Under the assumptions~\ref{B:Furstenberg}--\ref{B:Monotonicity} above, for $\Prob$-almost every $\omega\in\Omega$ the following holds:
\begin{itemize}
\item \textbf{(Regular upper limit)} For every $a\in J$ we have
$$
\limsup_{n\to\infty} \frac{1}{n} (\log \|T_{n,a, \omega} \| - L_n(a)) =0.
$$
%
%$$
% \quad \limsup_{n\to\infty} \frac{1}{n} \log \|T_{n,a,  \omega} \| = \lambda_F(a)>0.
%$$
\item \textbf{($G_{\delta}$-vanishing)} The set
$$
S_0(\omega):=\left\{a \in J \mid \liminf_{n\to\infty} \frac{1}{n} \log \|T_{n,a, \omega} \| =0 \right\}
$$
contains a (random) dense $G_{\delta}$-subset of the interior of the essential set~$\Es$.
\vspace{4pt}
\item \textbf{(Hausdorff dimension)} The (random) set of parameters with exceptional behaviour,
$$
S_{e}(\omega):=\left\{a \in J \mid \liminf_{n\to\infty} \frac{1}{n} (\log \|T_{n,a, \omega} \| -L_n(a))<0 \right\},
$$
has zero Hausdorff dimension:
$$
\dim_H S_{e}(\omega)=0.
$$
\end{itemize}
\end{theorem}

\begin{remark}
Regarding ``$G_{\delta}$-vanishing'' conclusion, we would like to emphasize that in the non-stationary setting the essential spectrum does not have to have dense interior. We construct the corresponding example in Appendix \ref{a.1}.
\end{remark}

%\vspace{1cm}

In order to study random discrete Schr\"odinger operators on $\ell_2(\Z)$, we will have to consider sequences of
random (parameter-dependent) matrices, indexed by $n\in \Z$ instead of $n\in \N$. We thus denote $\OZ:=\mA^{\Z}$,
and for a given bi-infinite sequence $\{\mgr_n\}_{n\in \Z}$ (again, with all the $\mgr_n$'s from some compact set $\mK$
of measures) equip $\OZ$ with the measure $\Prob:=\prod_{n\in \Z} \mgr_n$.

We then denote for a sequence $\omega=(\dots, A_{-1},A_0, A_1,\dots)\in\OZ$
$$
T_{-n,a, \omega}:=(A_{-n}(a))^{-1} \dots (A_{-1}(a))^{-1} (A_{0}(a))^{-1}
$$
and
$$
L_{-n}(a):=\E \log \|T_{-n,a,\omega} \|.
$$

\begin{theorem}\label{t.vector}
Under the assumptions~\ref{B:Furstenberg}--\ref{B:Monotonicity} we have:
\begin{itemize}
\item For almost all $\omega\in \Omega$, for all $a\in J$ the following holds. If
\begin{equation}\label{eq:lim-less}
\limsup_{n\to+\infty} \frac{1}{n} \left( \log |T_{n,a,  \omega} \left( \begin{smallmatrix} 1 \\ 0 \end{smallmatrix} \right)| - L_n(a) \right)<0,
\end{equation}
then in fact $|T_{n,a, \omega} \left(  \begin{smallmatrix} 1 \\ 0 \end{smallmatrix} \right)|$ tends to zero exponentially as $n\to \infty$. Namely, %, as fast as it can:
$$
 \log |T_{n,a, \omega} \left( \begin{smallmatrix} 1 \\ 0 \end{smallmatrix} \right)| = -L_n(a) + o(n).
$$

\item[$\bullet$] For almost all $\omega\in\OZ$, for all $a\in J$ the following holds. If for some $\bar{v}\in \R^2\setminus \{0\}$ we have
\begin{equation}\label{eq:lim-both-less}
\limsup_{n\to+\infty} \frac{1}{n} \left( \log |T_{n,a, \omega} \bar v| - L_n(a) \right)<0, \ \ \text{and}\ \ \
\limsup_{n\to+\infty} \frac{1}{n} \left( \log |T_{-n,a, \omega} \bar v| - L_{-n}(a) \right)<0,
\end{equation}
then both sequences $|T_{n,a, \omega} \bar v|, |T_{-n,a, \omega} \bar v|$ in fact tend to zero exponentially. More specifically,
$$
\log |T_{n,a, \omega} \bar v| = - L_n(a) + o(n), \ \ \text{and}\ \ \
\log |T_{-n,a, \omega} \bar v| = - L_{-n}(a) + o(n).
$$
\end{itemize}
\end{theorem}

%\newpage

%\subsection{Examples}

%\subsection{Finite product case}
Similarly to the stationary case treated in \cite{GK}, we obtain the results describing the behavior of infinite products by obtaining a description of the ``most probable'' behavior of a large finite product. To obtain such a description, for any given $n$ set $N=[\exp(\sqrt[4]{n})]$.
We divide the  interval of parameters $J$  into $N$ equal intervals $J_1,\dots, J_N$; let $b_-=b_0<b_1<\dots<b_{N-1}<b_N=b_+$ be their endpoints, i.e. $J_i=[b_{i-1},b_i]$.

By $U_{\varepsilon}(x)$ we denote the $\varepsilon$-neighborhood of the point $x$.

\begin{theorem}\label{t:main}
For any $\eps>0$ there exist $n_0=n_0(\eps)$ and $\delta_0=\delta_0(\eps)$ such that
for any $n>n_0$ the following statement hold. With probability $1-\exp(-\delta_0 \sqrt[4]{n})$, there exists a (random)
number $M\in \mathbb{N}$, exceptional intervals $J_{i_1},\dots,J_{i_{M}}$ (each of length $\frac{|J|}{N}$), and the corresponding numbers $m_1,\dots, m_{M}\in \{1,\dots,n\}$, such that:
\vspace{4pt}
\begin{mainit}
\mitem \textbf{(Uniform growth in typical subintervals)}
\label{i:m2}
For any $i$ different from $i_{1},\dots,i_{M}$, for any $a\in J_i$, and for any $m=1,\dots, n$ one has
$$
\log \| T_{m,a,\omega} \| \in U_{n\eps} (L_m(a)).
$$
\mitem \textbf{(Uniform growth in exceptional subintervals)}
\label{i:m3}
For any $k=1,\dots, M$, for any $a\in J_{i_k}$, and for any $m=1,\dots,m_k$ one has
$$
\log \| T_{m,a,\omega} \| \in U_{n\eps} (L_m(a));
$$
for any $m=m_k+1,\dots, n$ one has
$$
\log \| T_{[m_k,m],a,\omega} \| \in U_{n\eps} (L_{[m_k,m]}(a)),
$$
where%\todo[inline]{Add somewhere that $L_{[m',m]}(a):=\E\log\|T_{[m',m],a,\omega}\| \approx L_m(a)-L_{m'}(a)$.}
$$
T_{[m_k,m],a,\omega}:=T_{m,a,\omega} T_{m_k,a,\omega}^{-1} = A_m(a) A_{m-1}(a) \dots A_{m_k+1}(a)
$$
and
$$
L_{[m',m]}(a):=\E\log\|T_{[m',m],a,\omega}\|.% \approx L_m(a)-L_{m'}(a)
$$
%\vspace{1pt}
\mitem \textbf{(Cancelation)}
\label{i:m4}
For any $k=1,\dots,M$ there exists $a_{k}\in J_{i_k}$ such that for any $m=1,\dots, n$
\begin{equation}\label{e.IV}
\log \|  T_{m,a_k,\omega} \| \in U_{n\eps} (\psi_{m_k}(m,a_k)),
\end{equation}
where %\todo{$L_{[m',m]}$ is the difference or the expectation?}
$$\psi_{m'}(m,a)=
\begin{cases}
L_m(a) , & m<m', \\
|L_{m'}(a)-L_{[m',m]}(a)|, & m\ge  m'.
%| 2L_{m'}(a)-L_m(a) |
\end{cases}$$
%and $L_{[m',m]}(a):=\E \log \|  T_{m,a_k,\omega} \|$.
In other words, for $m\ge m_k$ the parts of the product over the intervals $[1,m_k]$ and
$[m_k,m]$ cancel each other in the best possible way.
\vspace{4pt}
%\mitem \textbf{(Quantity)}
%\label{i:m1}
%The number $M$ is belongs to $[R_{(1-\eps)n,\omega}(J)-1, R_{n,\omega}(J)+1]$.
%\vspace{4pt}

\mitem \textbf{(Distribution)}
\label{i:m5}
For any $m\le n$ and any interval $I\subset J$ of the form $I=[b_i,b_{i'}]$, $0\le i< i'\le N$, the number
$$
M_{I;m} := \# \{ k \mid a_{i_k}\in I, m_k\le m \}
$$
belongs to the interval $[R_{m,\omega}(I)-\eps n, R_{m,\omega}(I)+\eps n].$
In particular,
$$
M\in U_{n\eps}(R_{n,\omega}(J)).
$$

%\mitem \textbf{(Distribution)}
%\label{i:m5}
%For any $m\in [\eps n, n]$ and any interval $I\subset J$ of the form $I=[b_i,b_{i'}]$, $0\le i< i'\le N$, the number
%$$
%M_{I;m} := \# \{ k \mid a_{i_k}\in I, m_k\le m \}
%$$
%belongs to the interval
%$$
%[R_{m-\eps n,\omega}(I)-1, R_{m,\omega}(I)].
%$$
%In particular,
%$$
%M\in [R_{(1-\eps)n,\omega}(J)-1, R_{n,\omega}(J)].
%$$

%
%For each $k=1, \ldots, M$ consider the point $\left(\frac{m_k}{n}, a_{k} \right)\in [0,1]\times J.$  The measure
%$$
%\frac{1}{n}\sum_{k=1}^{M}\Dirac_{\left(\frac{m_k}{n}, a_{k} \right)}
%$$
%is $\eps$-close (in Levy-Prokhorov metric, i.e. in the metric that defines weak-* topology) to the one
%defined by
%$$
%\mathbf{m}([0,x]\times [a,a']) = \frac{1}{n} (R_{[nx]}(a')-R_{[nx]}(a)).
%$$
\end{mainit}
\end{theorem}

Theorem \ref{t:main} is a non-stationary analog of Theorem 1.19 from \cite{GK}. In the stationary case the distribution of the exceptional intervals had to converge to some measure analogous to the density of states measure in the context of random Schr\"odinger operators. This is where Theorem \ref{t:main} essentially differs  from \cite[Theorem 1.19]{GK}. While some attempts to generalize the notion of the density of states to the case of non-ergodic %or multi-dimensional
Schr\"odinger operators were made (e.g. see \cite{BKl}), the situation is quite delicate in this case. In particular, contrary to the Johnson's Theorem in the stationary case, even if the random rotation number exists (e.g., see Appendix A in \cite{GK} for details), an interval of its constancy is not necessarily inessential: %may not be inessential:
\begin{example}\label{ex:J-2}
As in Example~\ref{ex:J-1}, consider a stationary product $A(a)$, uniformly hyperbolic for $a\in I=[0,1]$, but now add an additional rotation at the steps that are perfect squares. That is, let
$$
B_j(a) = \begin{cases}
\Rot_{2\pi a} \cdot A_j(a), & j=k^2 \text{ for some } k\in \Z \\
A_j(a), & \text{otherwise}.
\end{cases}
$$
Then, for the product
$$
B_n(a)\dots B_1(a),
$$
on the one hand, the interval $I$ is not inessential, as $R_{\omega,k^2}(I)\ge k$ for any $k$.
On the other hand,
$$
\lim_n \frac{1}{n} R_{\omega, n}(I)=0,
$$
and the random rotation number is constant on the interval $I$.
\end{example}

Using the same ideas for the random Schr\"odinger operators setting, we get the following

\begin{example}
Consider the Anderson-Bernoulli potential $V(n)=\xi_n+r(n)$, where $\xi_n\in \{0,1\}$ are i.i.d. $(1/2, 1/2)$-Bernoulli random variables, and
$$
r(n)= \begin{cases}
20, & \text{if } n \text{ is a perfect square}, \\
0 & \text{otherwise}.
\end{cases}
$$

Then, due to Theorem~\ref{t.al}, the corresponding random Schr\"odinger operator almost surely has a pure point spectrum,
and its eigenfunctions are exponentially localized. However, the proportion of eigenfunctions with energy in the interval $[18,23]$ among those with the localization center in $[-N,N]$ tends to zero as $N\to\infty$. Hence for any natural definition of the density of states measure, it will not charge the interval $[18,23]$, though a nonempty subset of it belongs to the essential spectrum.
%will belong to the union\todo{exactly like this, or with a neighborhood?} $[-2,3]\cup [18,23]$, and
\end{example}

%\newpage

\subsection{Ideas of the proof and structure of the paper}

The proofs of non-stationary spectral and dynamical localization results, i.e. Theorems \ref{t.al} and  \ref{t.dl}, as well as the Non-stationary Parametric Furstenberg Theorem, i.e. Theorem \ref{t:product}, are based mostly on the results and strategy from two recent papers, \cite{GK} and \cite{GK22}. More specifically, our new proof of spectral localization in 1D Anderson Model, that we provided to demonstrate the power of the technics developed in \cite{GK}, used the results on parametric products of random $\SL(2, \mathbb{R})$ matrices. That result, in turn, was based on existence and positivity of Lyapunov exponents of random matrix products. In the non-stationary case  the norms of random matrix products do not have to have any exact exponential rate of growth, but their behaviour can be described by an exponentially growing non-random sequence $\{L_n\}$, as was shown in \cite{GK22}. This allows to use the general strategy of the proof of spectral localization from \cite{GK} in the non-stationary setting. Certainly, non-stationarity of the model brings many technical challenges. Just to give one example, in the stationary case the Lyapunov exponent $\lambda(a)$ is a continuous function of the parameter $a\in J$, and uniform continuity of that function over a compact interval in the parameter space is used in \cite{GK}. In the non-stationary setting it has to be replaced by equicontinuity of the sequence of functions $\{\frac{1}{n}L_n(a)\}_{a\in J}$, see Lemma \ref{l:equicontinuous} below.

Let us now describe the structure of the rest of the paper.

In Section \ref{s:PLD} we provide the statement of the Non-stationary Furstenberg Theorem and Large Deviations Estimates from \cite{GK22}, and explain why its parametric version, Theorem  \ref{t.2param}, also holds.
In Section \ref{ss.upperbound} we prove equicontinuity of the sequence $\{\frac{1}{n}L_n(a)\}_{a\in J}$, and use it to prove the first part (``Regular upper limit'') of Theorem \ref{t:product}, which can be considered as a dynamical analog of Craig-Simon's result on Schr\"odinger cocycles. After that, in Section \ref{ss.prrofofparam}, we deduce the rest of the Non-stationary Parametric Fursetenberg Theorem (Theorem \ref{t:product}) from Theorem~\ref{t:main} on properties of finite products of random matrices.

The central technical statement of this paper is Theorem~\ref{t:main}, describing the typical behaviours of finite length random products, and we prove it in Section \ref{s.proofof114}.  The number $N=[\exp\sqrt[4]{n}]$ of parameter intervals in Theorem~\ref{t:main} grows subexponentially in the number $n$ of iterations. Due to the large deviations type bounds, the growth of the log-norm at each of their endpoints is $n\eps$-close to its expected value with the probability that is exponentially close to~$1$, and hence (as the number of endpoints growth subexponentially), the same applies to all endpoints~$b_j$ simultaneously. We establish these large deviations type bounds %in Section~\ref{s:PLD} (see Theorems~\ref{t.LDET2} %and~\ref{t.2param}) and further
in Section~\ref{ss.5.3}, see Lemma~\ref{l:LD-RD}.

Next, we extend the control from the endpoints to full parameter intervals. To do so, we study how the image of a given initial point $x_0\in\Sc$ after a given number of iterations $m$ varies when the parameter varies over the corresponding parameter interval $J_i=[b_{i-1},b_i]$. If all such variations are sufficiently small, the distortion control ideas (see Section~\ref{ss:distortion} and Lemma~\ref{l:DC}) allow us to observe the same growth of log-norms for all intermediate parameter values. The same applies if the variation becomes large at some intermediate moment $m$, but then ``quickly'' decreases.

Proposition~\ref{p:classes} in Section~\ref{s:finite-products} states that with the probability close to~$1$ there are only three possible types of behaviour for such variations: the two mentioned above and the third one, when the image point at some moment $m_0$ makes a full turn around the circle. In the latter case the product $T_{[m_0,m],a,\omega}$ after this moment turns out to be growing uniformly in $a\in J_j$, and such a parameter interval is thus exceptional in terms of Theorem~\ref{t:main}. In Sections~\ref{ss:norms} and~\ref{ss:cancellation}  we deduce Theorem~\ref{t:main} from Proposition~\ref{p:classes} described above. Then, in Section~\ref{ss.5.6}  we provide the proof of contraction estimates in the non-stationary setting, which is the only part where the proof of Proposition~\ref{p:classes} differs from its stationary version (see Remark~\ref{r.proofoftypes}).

%Other results are deduced from Theorem~\ref{t:main}.
In Section~\ref{s:spectral}, we prove Theorem~\ref{t.al} and therefore establish the spectral localization in our model. To do that, we study the possible growth of log-lengths of the images of a particular vector. It turns out that these log-lengths up to an error term $\eps n$ (with arbitrarily small~$\eps$) follow one of possible patterns, see Lemmata~\ref{l:line-shape}, \ref{l:V-shape}, and \ref{l:W-shape}. This allows to prove Theorem~\ref{t.vector} that claims that almost surely every vector whose images are not growing with respect to the non-random sequence $L_n(a)$ must actually decay exponentially fast. In the case of transfer matrices associated with Anderson Model  that implies that any solution that does not grow exponentially fast must in fact decay exponentially fast, and due to Shnol's Lemma the corresponding Schr\"odinger operator must enjoy spectral localization.

In Section \ref{s.dynloc} we establish that discrete 1D Schr\"odinger operator with non-stationary random potential has semi-uniformly localized eigenfunctions (SULE), which implies Dynamical Localization and proves Theorem~\ref{t.dl}.

Finally, we provide two appendices with examples that show that the properties of a non-stationary Anderson Model and non-stationary random matrix products can be drastically different from their stationary versions. Namely, in Appendix \ref{a.1} we show that the essential spectrum of a discrete Schr\"odinger operator with non-stationary random potential does not have to have dense interior. More specifically, we give an explicit example of such an operator whose essential spectrum intersected with an interval forms a Cantor set of zero measure. And in Appendix \ref{s:limsup} we give a surprising example of a non-stationary parametric random matrix sequence $A_1(a), A_2(a),\ldots, A_n(a), \ldots$ such that $\limsup_{n\to \infty} \frac{1}{n}\log\|A_n(a)A_{n-1}(a)\ldots A_1(a)\|$ is not a continuous function of the parameter. In the stationary setting this function corresponds to the Lyapunov exponent, that is known to be continuous in this context (under very mild conditions that are satisfied, say, by transfer matrices of Schr\"odinger cocycle associated with Anderson Model).

\section{Non-stationary Parametric Furstenberg Theorem}

\subsection{Parametric Large Deviation Estimates Theorem}\label{s:PLD}

First of all, we will need a version of the Large Deviation Estimates Theorem \cite[Theorem 1.8]{GK22}  %(i.e. Theorem \ref{t.2})
that is uniform in the parameter. Let us provide the original statement first:
\begin{theorem}[Large Deviations for Nonstationary Products, \cite{GK22}]\label{t.LDET2}
Let $\mathbf{K}$ be a compact set of probability measures on $\SL(d, \mathbb{R})$. Assume that the following hold:
\begin{itemize}
\item
\textbf{(finite moment condition)} There exists $\gamma>0$, $C$ such that
$$%\begin{equation}\label{eq:finite-moment}
\forall \mgr\in \mathbf{K}\quad  \int_{\SL(d,\R)} \|A\|^{\gamma} d \mgr(A) < C
$$%\end{equation}
\item
\textbf{(measures condition)} For any $\mgr\in \mathbf{K}$ there are no Borel probability measures $\msp_1$, $\msp_2$ on $\mathbb{RP}^{d-1}$ such that $(f_A)_*\msp_1=\msp_2$ for $\mgr$-almost every $A\in \SL(d, \mathbb{R})$
\item
\textbf{(spaces condition)} For any $\mgr\in \mathbf{K}$ there are no two finite unions $U$, $U'$ of proper subspaces of~$\R^{d}$ such that $A(U)=U'$ for $\mgr$-almost every $A\in \SL(d, \mathbb{R})$.
\end{itemize}
Then for any $\varepsilon>0$ there exists $\delta>0$ such that for any sequence of distributions $\mgr_{1}, \mgr_{2},\ldots , \mgr_{n}, \ldots$ from $\mathbf{K}$, for all sufficiently large $n\in \mathbb{N}$ we have
$$
\mathbb{P}\left\{\left|\log\|T_n\|-L_n\right|>\varepsilon n\right\}<e^{-\delta n},
$$
where $T_n=A_nA_{n-1}\ldots A_1$, $\{A_j\}$ are chosen randomly and independently with respect to $\{\mu_j\}$, $\mathbb{P}=\mgr_{1}\times \mgr_{2}\times\ldots\times \mgr_{n}$, and $L_n=\mathbb{E}(\log \|T_n\|)$. Moreover, the same estimate holds for the lengths of random images of any given initial unit vector $v_0$:
$$
\forall v_0\in \R^d, \, |v_0|=1 \quad  \mathbb{P}\left\{\left|\log\|T_n v_0\|-L_n\right|>\varepsilon n\right\}<e^{-\delta n}.
$$
Finally, the expectations $L_n$ satisfy a lower bound
\[
L_n \ge n h,
\]
where the constant $h>0$ can be chosen uniformly for all possible sequences $\{\mgr_n\}\in \mathbf{K}^{\N}$.
\end{theorem}

 Let us now state a version of the Large Deviation Estimates that we need:

\begin{theorem}\label{t.2param}
Under the assumptions of Theorem \ref{t:product} for any $\varepsilon>0$ there exists $\delta>0$ such that for all sufficiently large $n\in \mathbb{N}$ and all values of the parameter $a\in J$ we have
$$
\mathbb{P}\left\{\left|\log\|T_{n, a, \omega}\|-L_n(a)\right|>\varepsilon n\right\}<e^{-\delta n},
$$
where $\mathbb{P}=\mgr_{1}^a\times \mgr_{2}^a\times\ldots\times \mgr_{n}^a$. Moreover, the same estimate holds for the lengths of random images of any given initial unit vector $v_0$:
$$
\forall v_0\in \R^2, \, |v_0|=1 \quad  \mathbb{P}\left\{\left|\log\|T_{n, a, \omega} v_0\|-L_n(a)\right|>\varepsilon n\right\}<e^{-\delta n}.
$$
Finally, the expectations $L_n$ satisfy a lower bound
\begin{equation}\label{eq:L-nh}
L_n \ge n h,
\end{equation}
where the constant $h>0$ can be chosen uniformly for all possible sequences $\{\mgr_n\}\in \mK^{\N}$ and all parameters $a\in J$.
\end{theorem}
%\begin{proof}[Proof of Theorem \ref{t.2param}]%
Theorem \ref{t.2param} can be considered as a particular case of Theorem \ref{t.LDET2}. Indeed, finite moment condition trivially holds under the assumptions of Theorem \ref{t.LDET2} due to condition (B2). Since in Theorem \ref{t.2param} we are only interested in $\SL(2, \mathbb{R})$ matrices, spaces condition holds as soon as measures condition holds; otherwise an atomic measure on the finite invariant set in $\mathbb{RP}^1$ would be a measure with a deterministic image. Finally, the collection of measures $\left\{\mu^a\ |\ \mu\in \mathcal{K}, a\in J\right\}$ forms a compact subset in the space of measures on~$\SL(2, \mathbb{R})$, and due to condition (B1), every measure $\mu^a$ from that compact satisfies the measures condition from Theorem \ref{t.LDET2}. Therefore, Theorem \ref{t.LDET2} is applicable with $\mathbf{K}=\left\{\mu^a\ |\ \mu\in \mathcal{K}, a\in J\right\}$. In particular, it is applicable to any sequence of the form $\mgr_{1}^a,  \mgr_{2}^a, \ldots, \mgr_{n}^a, \ldots$, which gives Theorem~\ref{t.2param}.

\subsection{Upper bound for the upper limit}\label{ss.upperbound}

We will need the following statement:
\begin{lemma}\label{l:equicontinuous}
The sequence of functions $\left\{\frac{1}{n}L_n(a)\right\}_{n\in \mathbb{N}}$ is equicountinuous.
\end{lemma}

\begin{remark}
Lemma \ref{l:equicontinuous} can be considered as a non-stationary analog of continuity of the Lyapunov exponent in the classical stationary setting. Indeed,
it implies that $\limsup_{n\to \infty} \frac{1}{n}L_n(a)$ is a continuous function. Nevertheless, this analogy does not go too far. Namely, in spite of the first claim of Theorem \ref{t:product}, in the non-stationary setting there are examples where almost surely $\limsup_{n\to\infty} \frac{1}{n} \log \|T_{n,a, \omega} \|$ is a discontinuous function of the parameter $a$. We construct such an example in Appendix \ref{s:limsup}.
\end{remark}

\begin{proof}[Proof of Lemma~\ref{l:equicontinuous}]
For any fixed $k\in \N$ we can decompose the product $T_{n, a, \omega}$ into product of groups by $k$,
$$
T_{n,a, \omega}= B_m(a)\dots B_1(a),
$$
%\, \text{
where %(see~\eqref{eq:B-j})
%} \,
$$
B_j(a):=(A_{kj}(a)\dots A_{k(j-1)+1}(a)), \quad j=1,\dots, m.
$$

Now, take any unit vector $v_0\in \mathbb{R}^d$; in order to control the cancellations in the action of the above product on $v_0$, define
$$
\xi_{j,a}:=\log \|B_j(a) \|, \quad S_{j,a}:=\log |T_{kj,a, \omega}(v_0)|, \quad R_{j,a}=S_{j,a}+\xi_{j+1,a}-S_{j+1,a}.
$$
Then %(see~\eqref{eq:glue},~\eqref{eq:xi-S})
\begin{equation}\label{eq:glue-param}
\log |T_{km,a, \omega} v_0| = S_{m,a} = (\xi_{1,a}+\dots+ \xi_{m,a})-(R_{0,a}+R_{1,a}+\dots+R_{m-1,a}).
\end{equation}
and hence
\begin{multline}\label{eq:xi-S-param}
(\xi_{1,a}+\dots+ \xi_{m,a})\ge \log \|T_{n,a, \omega}\| \ge \log |T_{n,a, \omega} v_0| \\
 =(\xi_{1,a}+\dots+ \xi_{m,a})-(R_{0,a}+R_{1,a}+\dots+R_{m-1,a}).
\end{multline}

%As the conclusion of Lemma~\ref{l:psi-param} on dissolving of atoms is uniform in~$a\in J$, in the same way as before we recover a parametric version
By direct examining of the proof of \cite[Proposition~3.2]{GK22}, one can see that it actually implies its parametric version as well,
%for each value of the parameter $a\in J$,
that provides a uniform in $a\in J$ value of $\delta^*>0$ for each given $\eps^*>0$:

\begin{prop}\label{p:R-LD-param}
For any $\eps^*>0$ there exists $k_1$, such that for any $k>k_1$ for some $\delta^*>0$ one has for all $n=km$ and for any $a\in J$
% sufficiently large
$$
\Prob(R_{0,a}+R_{1,a}+\dots+R_{m-1,a} > n\eps^*) < e^{-\delta^* m}.
$$
\end{prop}
Let $\eps>0$ be fixed. Take $\eps^*:=\eps/5$, and let us apply Proposition~\ref{p:R-LD-param}:
for any sufficiently large $k$ there exists $\delta^*>0$ such that  for any $a\in J$
\begin{equation}\label{eq:R-bound}
R_{0,a}+R_{1,a}+\dots+R_{m-1,a} \le \frac{\eps}{5} n
\end{equation}
with probability at least $1-e^{-\delta^* m}$. Once~\eqref{eq:R-bound} holds, we have from~\eqref{eq:xi-S-param}
\begin{equation}\label{eq:T-xi}
\left | \log \|T_{n,a, \omega}\| - (\xi_{1,a}+\dots+ \xi_{m,a}) \right| <\frac{\eps}{5} n.
\end{equation}
Thus, for any two parameter values $a_1,a_2\in J$ both inequalities
\begin{equation}\label{eq:xi-L}
\left | \log \|T_{n,a_i, \omega}\| - (\xi_{1,a_i}+\dots+ \xi_{m,a_i}) \right| <\frac{\eps}{5} n, \quad i=1,2
\end{equation}
hold with the probability at least $1-2e^{-\delta^* m}$.

Finally, note that due to the assumption~\ref{B:C1} (for a fixed $k$) the random variables $\xi_{j,a}$, considered as (random) functions of the parameter $a\in J$, are equicontinuous. Hence, there exists $\delta_J$ such that for any $a_1,a_2\in J$ with $|a_1-a_2|<\delta_J$ we have
$$
|\xi_{j,a_1}-\xi_{j,a_2}|< \frac{\eps}{5},
$$
and thus
$$
\sum_{j=1}^m |\xi_{j,a_1}-\xi_{j,a_2}|< \frac{\eps}{5} m.
$$
Combining this estimate with~\eqref{eq:xi-L}, we conclude that with the probability at least $1-2e^{-\delta^* m}$
\begin{equation}\label{eq:a1-a2}
\left | \frac{1}{n}\log \|T_{n,a_1, \omega}\| - \frac{1}{n}\log \|T_{n,a_2, \omega}\| \right| < \frac{3}{5} \eps.
\end{equation}
Let us now consider the expectations $L_n(a_i) = \E \log \|T_{n,a_i, \omega}\|$. We have
\begin{multline*}
\frac{1}{n} |L_n(a_1) - L_n(a_2)| \le \frac{1}{n} \E \left | \log \|T_{n,a_1, \omega}\| - \log \|T_{n,a_2, \omega}\| \right|  =
\\
= \frac{1}{n} \E \Ind_{\{\text{Eq. \eqref{eq:xi-L} holds}\} } \left | \log \|T_{n,a_1, \omega}\| - \log \|T_{n,a_2, \omega}\| \right|
\\ + \frac{1}{n} \E \Ind_{\{\text{Eq. \eqref{eq:xi-L} does not hold}\} }  \left | \log \|T_{n,a_1, \omega}\| - \log \|T_{n,a_2, \omega}\| \right|
\end{multline*}
As we have $\frac{1}{n}\log \|T_{n,a, \omega}\| < \log \Cn$ everywhere due to the assumption~\ref{B:C1},
the contribution of the part where~\eqref{eq:xi-L} does not hold cannot exceed $2e^{-\delta^*n}\log \Cn$,
and hence tends to zero as $n\to\infty$. In particular, for all $n$ sufficiently large it does not exceed~$\frac{\eps}{5}$.

On the other hand, once $|a_1-a_2|<\delta_J$, due to~\eqref{eq:a1-a2} the contribution of the part where~\eqref{eq:xi-L} holds
does not exceed~$\frac{3\eps}{5}$. We finally get for all sufficiently large~$n=km$
$$
 \left|\frac{1}{n}L_n(a_1) - \frac{1}{n}L_n(a_2)\right| < \frac{3\eps}{5}+\frac{\eps}{5} <\eps.
$$
As $\eps>0$ was arbitrary (and with an easy handling of $n$ not divisible by $k$
and of a finite number of $n$ that are too small) we obtain the desired equicontinuity.

% and the inequality~\eqref{eq:xi-S-param}, for any parameter $a\in J$ we have with the probability at least $1-2e^{-n\delta}$
%$$
%\ge \log \|T_{n,a}\|
%$$

%Now,~\eqref{}
\end{proof}

Notice that the proof above in fact gives a bit more. Namely, equicontinuity holds for all the functions $\{\frac{1}{n}L_n(a)\}$ regardless of the specific choice of the sequence of the measures $\mu_1, \mu_2, \ldots \in \mathcal{K}$. Indeed, for any specific $n_0\ge 1$ the set of $n_0$-tuples of measures from $\mathcal{K}^{n_0}$ is compact, the map $(\mu_1, \mu_2, \ldots, \mu_{n_0})\mapsto \frac{1}{n_0}L_{n_0; \mu_{n_0}*\ldots *\mu_1}(\cdot)\in C(J, \mathbb{R})$ is continuous. Therefore, its image is compact, hence is an equicontinuous family of functions. That is, the following statement holds:
\begin{lemma}
The family of functions $\{\frac{1}{n}L_n(a)\}_{n\in \mathbb{N}, \mu_1, \mu_2, \ldots\in \mathcal{K}}$ is equicontinuous.
\end{lemma}
As an immediate consequence, we get the following:
\begin{lemma}\label{l.equim1m2}
For any given sequence of distributions $\{\mu_1, \mu_2, \ldots\}\subset \mathcal{K}^\mathbb{N}$, the family of functions $\{\frac{1}{m_2-m_1}L_{[m_1, m_2]}(a)\}_{0\le m_1<m_2}$ is equicontinuous.
\end{lemma}

We are now ready to prove half of the first part ({\it Regular upper limit}) of Theorem~\ref{t:product}. Namely, we have the following
\begin{prop}\label{p.upper}
For a.e. $\omega\in \Omega$ and \emph{any} $a\in J$ one has $$
\limsup_{n\to\infty} \frac{1}{n} (\log \|T_{n,a, \omega} \| - L_n(a)) \le 0.
$$
\end{prop}
In a sense it can be considered as a non-stationary dynamical analog of Craig-Simon's result \cite[Theorem~2.3]{CS} on Schr\"odinger cocycles. Its stationary counterpart is Proposition~2.1 from~\cite{GK}.

This statement is implied by a large deviation-type bound for products of a given length $n$, that will be useful for us later:
\begin{prop}\label{p.upper.LD}
For any $\beps>0$ there exists $c_2,C_2>0$ and $n_1\in\N$  %(depending only on the set of measures $\mK$)
such that for any $n>n_1$ with the probability at least $1- C_2 \exp(-c_2 n)$ the following statement holds: for any $a\in J$ one has
\begin{equation}\label{eq:m-all}
%\forall a\in J \qquad
\log \|T_{n,a,\bo}\| - L_{n}(a) \le n\beps.
\end{equation}
%
%
%For every $\beps>0$ there exists $\delta'''>0$ such that for any $a\in J$ the probability of the event
%\begin{equation}\label{eq:lower-beps}
%\frac{1}{n} (\log \|T_{n,a, \omega} \| - L_n(a)) > \beps
%\end{equation}
%does not exceed $e^{-\delta'''n}$.
\end{prop}

Deducing Proposition~\ref{p.upper} from Proposition~\ref{p.upper.LD} is quite straightforward:

\begin{proof}[Proof of Proposition~\ref{p.upper}]
%[Deducing Proposition~\ref{p.upper} from Proposition~\ref{p.upper.LD}]
%Indeed, to deduce Proposition~\ref{p.upper} from Proposition~\ref{p.upper.LD}, it suffices to
Applying Borel--Cantelli Lemma, we notice that for every $\beps>0$ and $a\in J$ the inequality~\eqref{eq:m-all} almost surely takes place for all sufficiently large~$n$. Hence, for every $\beps>0$ and $a\in J$ one has almost surely
$$
\limsup_{n\to\infty} \frac{1}{n} (\log \|T_{n,a, \omega} \| - L_n(a)) \le \beps.
$$
Taking a countable family of $\beps$'s, tending to~0, and intersecting the corresponding events, we obtain the desired conclusion.
\end{proof}

\begin{proof}[Proof of Proposition~\ref{p.upper.LD}]
%\todo[inline]{REW}
Let $\eps>0$ be given; choose and fix $\eps^*,k,\delta_J$ as in the proof of Lemma~\ref{l:equicontinuous}. Take points
$\{b_1,\dots, b_N\}\subset J$, dividing $J$ into intervals of length less than~$\delta_J$.
The number $N'$ of these intervals, $N'>\frac{|J|}{\delta_J}$, does not depend on the number $m$ of summands; meanwhile, independent random variables $\xi_{j,a}= \log \|B_j(a)\|$ satisfy uniform upper bound
\[
|\xi_{j,a}| \le k \log \Cn
\]
and hence uniform large deviation estimates. That is, there exists $c_2>0$ such that for every $a\in J$ the probability of the event
\[
|\xi_{1,a}+\dots+\xi_{m,a}-\E(\xi_{1,a}+\dots+\xi_{m,a})| > \frac{\eps}{5} n
\]
does not exceed $e^{-c_2 n}$. The number $N$ of values $b_i$ does not depend on $n$, hence, the probability of the event
\begin{equation}\label{eq:a-i-N}
\forall i=1,\dots,N : \quad  |\xi_{1,b_i}+\dots+\xi_{m,b_i}-\E(\xi_{1,b_i}+\dots+\xi_{m,b_i})| \le \frac{\eps}{5} n
\end{equation}
is at least $1-N'e^{-c_2 n}$. %, that is also exponentially decreasing in~$n$.

Next, taking the expectation of the left hand side of~\eqref{eq:T-xi} and (as before) taking into account the uniform upper bound $\frac{1}{n}\log\|T_{n,a, \omega}\|\le \log \Cn$ to control the expectation where~\eqref{eq:xi-L} does not hold, we get that for all sufficiently large $n$ for all $i=1,\dots, N'$
$$
\frac{1}{n} \left| L_n(b_i) -\E(\xi_{1,b_i}+\dots+\xi_{m,b_i}) \right| < \frac{2\eps}{5}.
$$
Joining it with~\eqref{eq:a-i-N}, when the latter holds, we get that for all $i=1,\dots, N'$
$$
\frac{1}{n} \left| L_n(b_i) -(\xi_{1,b_i}+\dots+\xi_{m,b_i}) \right| < \frac{3\eps}{5}.
$$

Now, take any $a\in J$, and choose $i$ such that $b_i$ is $\delta_J$-close to~$a$. Then we have
$$
\log \|T_{n,a, \omega}\| \le \xi_{1,a}+\dots + \xi_{m,a} \le (\xi_{1,b_i}+\dots + \xi_{m,b_i}) + \frac{\eps}{5} m
$$
due to the equicontinuity of $\xi_{j,a}$'s and the choice of~$\delta_J$.
Joining it with the previous inequality, we get
$$
\frac{1}{n}\log \|T_{n,a, \omega}\| \le \frac{1}{n} L_{n}(b_i) + \frac{4\eps}{5}
$$
Finally, we have $\frac{1}{n} L_n(a) \ge \frac{1}{n} L_n(b_i)-\eps$, and thus
$$
\frac{1}{n}\log \|T_{n,a, \omega}\| \le \frac{1}{n} L_{n}(a) + \frac{9}{5} \eps
$$
holds with the probability is at least $1-N'e^{-c_2 n}$.

%, and thus
%$$
%\frac{1}{n}(\log \|T_{n,a}\| - L_n(a)) \le \frac{9}{5} \eps.
%$$
We have obtained the desired upper bound for $\beps=\frac{9}{5}\eps$.
As $\eps>0$ was arbitrary, Proposition~\ref{p.upper.LD} follows.
\end{proof}

\begin{coro}\label{c:upper-finite-m}
For any $\beps>0$ there exists $c_3>0$ and $n_1\in\N$ such that for any $n>n_1$ with the probability at least $1- \exp(-c_3 n)$ the following statement holds: for any $a\in J$ and any $0\le m'\le m''\le n$ one has
\begin{equation}\label{eq:m-all-a}
\log \|T_{[m',m''],a,\bo}\| - L_{[m',m'']}(a) \le n\beps.
\end{equation}
\end{coro}
\begin{proof}
If $m''-m'< n \frac{\beps}{\log \Cn}$, the inequality~\eqref{eq:m-all-a} holds automatically due to the upper bound on norms of $A_i(a)$'s.

If $m''-m'\ge n \frac{\beps}{\log \Cn}$, the probability of the the corresponding event
\[
\log \|T_{[m',m''],a,\bo}\| - L_{[m',m'']}(a) \ge n\beps
\]
is bounded from above by
\[
C_2 \exp (-c_2 (m''-m')) \le C_2 \exp\left(-c_2  \frac{\beps}{\log \Cn} \cdot n\right)
\]
 due to Proposition~\ref{p.upper.LD}. There are less than $n^2$ such events, and choosing sufficiently small $c_3>0$ allows to ensure the upper bound
 \[
n^2\cdot  C_2 \exp\left(-c_2  \frac{\beps}{\log \Cn} \cdot n\right) < \exp(-c_3 n)
 \]
 for all sufficiently large $n$.
\end{proof}

\subsection{Proof of parametric non-stationary Furstenberg Theorem via parameter discretization}\label{ss.prrofofparam}

Here we derive Theorem~\ref{t:product} (parametric non-stationary Furstenberg Theorem) from Theorem~\ref{t:main} (on properties of finite products of random matrices). The proof is parallel to the content of the Section 3 from \cite{GK}.

\begin{proof}[Proof of~Theorem~\ref{t:product}] Combining Borel-Cantelli Lemma with Theorem~\ref{t:main} we observe that for any $\eps>0$ almost surely there exists $n_0=n_0(\eps)$ such that for any $n\ge n_0$ there are $M_n\in \mathbb{N}$ and exceptional intervals $J_{i_1, n}, J_{i_2, n}, \ldots J_{i_{M_n}, n}$ such that the properties {\bf I--IV} from Theorem \ref{t:main} hold. Notice that comparing to the notation used in Theorem \ref{t:main} we add $n$ as an index to emphasize the dependence of these objects on $n$. Let us also define
$$
V_{n',\eps}:= \bigcup_{n\ge n'} \, \bigcup_{k=1,\dots, M_n} J_{i_k,n},
$$
and
$$H_{\eps}=\bigcap_{n'\ge n_0(\eps)} V_{n',\eps}.$$

{\bf Regular upper limit:} Due to Proposition \ref{p.upper} we only need to show that almost surely for all $a\in J$ we have
\begin{equation}\label{e.uplim}
\limsup_{n\to\infty} \frac{1}{n} \left(\log \|T_{n,a, \omega} \| - L_n(a)\right) \ge 0.
\end{equation}
If a given $a\in J$ does not belong to $H_{\eps}$, then it does not belong to exceptional intervals $J_{i_k,n}$ for all sufficiently large $n$. Therefore due to property~\ref{i:m2} from Theorem \ref{t:main} for all sufficiently large $n$ we have  $\log \|T_{n,a,\omega} \|-L_n(a)\ge -\eps n$, or $\frac{1}{n}(\log \|T_{n,a,\omega} \|-L_n(a))\ge -\eps.$ Hence
\begin{equation}\label{e.n11}
\limsup_{n\to\infty} \frac{1}{n} \left(\log \|T_{n,a, \omega} \| - L_n(a)\right) \ge -\eps.
\end{equation}
If $a\in H_{\eps}$, there is an arbitrarily large $n$ such that $a\in J_{i_k,n}$ for some exceptional interval $J_{i_k,n}$. Consider the corresponding value $m_{k,n}$ and notice that the property~\ref{i:m3} from Theorem \ref{t:main} implies the following. If $\frac{m_{k,n}}{n}>\sqrt{\eps}$, then $\log \|T_{m_{k,n},a,\omega} \|-L_{m_{k,n}}(a)\ge -\eps n$, and hence
\begin{equation}\label{e.n1}
\frac{1}{m_{k,n}}\left(\log \|T_{m_{k,n},a,\omega}\|-L_{m_{k,n}}(a)\right) \ge  -\eps \frac{n}{m_{k,n}}\ge  -\sqrt{\eps}.
\end{equation}
Now, assume that $\frac{m_{k,n}}{n}\le\sqrt{\eps}$. Then, we have
\[
L_n(a) \le L_{[m_{k,n},n]}(a)+L_{m_{k,n}}(a),
\]
and
\[
\log \|T_{n,a,\omega} \| \ge \log\|T_{[m_{k,n},n],a,\omega} \| -\log\|T_{m_k,a,\omega} \|.
\]
Subtracting, we get
\begin{multline*}
    \log\|T_{n,a,\omega} \|-L_{{n}}(a)\ge (\log \|T_{[m_{k,n}, n],a,\omega} \| - L_{[m_{k,n},n],a}) - \\
    - (\log \|T_{m_{k,n},a,\omega} \| - L_{m_{k,n},a}) - 2 L_{m_{k,n},a} \ge
    \\
    \ge
-\eps(n-m_{k,n})-\eps\,m_{k,n} -2L_{m_{k,n}}(a)\ge \\
-\eps n-2L_{m_{k,n}}(a)\ge -\eps n-2m_{k,n}C_{\max},
\end{multline*}
%\todo[inline]{REW the formula}
%\begin{multline*}
%    \log\|T_{n,a,\omega} \|-L_{{n}}(a)\ge \log \|T_{[m_{k,n}, n],a,\omega} \| -\log \|T_{m_{k,n},a,\omega}\| -L_{{n}}(a)\ge\\
%\left(\log \|T_{[m_{k,n}, n],a,\omega} \|-L_{[m_{k,n}, n]}(a)\right)-\left(\log \|T_{m_{k,n},a,\omega}\|-L_{m_{k,n}}(a)\right)+\\ L_{[m_{k,n}, n]}(a)-L_{m_{k,n}}(a) -L_n(a)\ge \\
%-\eps(n-m_{k,n})-\eps\,m_{k,n} +\left(L_{[m_{k,n}, n]}(a)+L_{m_{k,n}}(a)-L_n(a)\right) -2L_{m_{k,n}}(a)\ge\\
%-\eps n-2L_{m_{k,n}}(a)\ge -\eps n-2m_{k,n}C_{\max},
%\end{multline*}
where we denote
\begin{equation}\label{e.defCa}
C_{\max}=\max_{\mgr\in \mathcal{K}, \ A\in \supp \mgr, \,a\in J}\log \|A(a)\|<\infty.
\end{equation}
Notice that $C_{\max}$ is finite due to condition (B2). Hence
\begin{equation}\label{e.n2}
\frac{1}{{n}}\left(\log\|T_{n,a,\omega} \|-L_{{n}}(a)\right)\ge -\eps -2C_{\max}\frac{m_{k,n}}{n}\ge -\eps-2C_{\max}\sqrt{\eps}.
\end{equation}

Therefore, in any case from (\ref{e.n1}) and (\ref{e.n2}) we get
\begin{equation}\label{e.n22}
\limsup_{n\to \infty}\frac{1}{{n}}\left(\log\|T_{n,a,\omega} \|-L_{{n}}(a)\right)\ge -\max(\sqrt{\eps}, \eps+2C_{\max}\sqrt{\eps}).
\end{equation}
Finally, applying (\ref{e.n11}) and (\ref{e.n22}) along a sequence of values of $\eps>0$ that tends to zero, we observe that almost surely  (\ref{e.uplim}) holds, and hence the first claim of Theorem~\ref{t:product} (on regular upper limit) follows.

\vspace{4pt}

{\bf $G_\delta$ vanishing:} Let us recall that the constant $C_{\max}$ given by (\ref{e.defCa}) is an upper bound for all log-norms of all the matrices that can be encountered in the random product. Due to Theorem~\ref{t.LDET2} above (or Theorem~1.4 from \cite{GK22}), there is $h>0$ such that $L_{[m_1, m_2]}(a)\ge h(m_2-m_1)$ for any $m_2>m_1$ and any $a\in J$. For each $n, p\in \mathbb{N}$ introduce the set
$$
W_{n, p}=\left\{a\in J\ \mid \ \text{for some}\ m\ge n\ \text{we have} \ \frac{1}{m}\log\|T_{m, a, \omega}\|<\frac{4\left(({C_{\max}}/{h})+1\right)}{p}\right\}.
$$
We claim that $W_{n,p}$ is open and dense in the interior of the essential set $\Es$ for any $n,p\in \mathbb{N}$. Indeed, it is clear that each set $W_{n,p}$ is open.  Apply Theorem \ref{t:main} for $\eps=\frac{1}{p}$ with sufficiently large $p$; namely, we require $p>10\left(\frac{C_{\max}}{h}+1\right)$. Denote by $\{a_{k,n}\}$ the set of exceptional parameters provided by Property {\bf III} for a given $n\ge n_0$. Since any interval $I\subset \Es$ is not inessential, Property {\bf IV} implies that $\cup_{n\ge n_0}\{a_{k,n}\}$ is dense in $\text{int}\,\Es$. Moreover, since $R_{n, \omega}(I)\to \infty$ as $n\to \infty$, we must have
\[
R_{\left(\frac{1}{({C_{\max}}/{h})+1}-\eps\right)n, \omega}(I)>R_{\frac{1}{2}\frac{n}{({C_{\max}}/{h})+1}, \omega}(I)
\]
for infinitely many large values of $n$. Therefore,  Property {\bf IV} implies that the set of exceptional parameters $\{a_{k,n}\}$ with
\[
\frac{m_{k,n}}{n}\in \left[\frac{1}{2}\frac{1}{({C_{\max}}/{h})+1}, \frac{1}{({C_{\max}}/{h})+1}\right]
\]
is also dense in $\text{int}\,\Es$. For any such parameter there exists $\tilde m_{k,n}\in [m_{k,n}, n]$ such that
\[
|L_{m_{k,n}}(a)-L_{[m_{k,n}, \tilde m_{k,n}]}(a)|\le C_{\max};
\]
therefore, Property {\bf IV} implies that
\[
\frac{1}{\tilde m_{k,n}}\log \|T_{\tilde m_{k,n}, a, \omega}\|\le \frac{C_{\max}+\eps n}{\tilde m_{k,n}}< 2\eps\frac{n}{m_{k,n}}\le 4\left(({C_{\max}}/{h})+1\right)\eps=\frac{4\left(({C_{\max}}/{h})+1\right)}{p}.
\]
This implies that  $W_{n,p}$ is dense in $\text{int}\,\Es$. Hence, the intersection $\bigcap_{n,p=1}^\infty (W_{n,p}\cap\, \text{int}\,\Es)$ is a dense $G_\delta$-subset of $\text{int}\,\Es$, and for any $a\in \bigcap_{n,p=1}^\infty W_{n,p}$ we have
$$
\liminf_{n\to \infty}\frac{1}{n}\log \|T_{n, a, \omega}\|=0.
$$

{\bf Hausdorff dimension:} First of all, notice that $H_{\eps}\subseteq J$ has zero Hausdorff dimension. Indeed, $H_{\eps}$ is contained in $V_{n',\eps}$, which is covered by $\left\{J_{i_k, n}\right\}_{n\ge n', \, k\le M_n}$. Property {\bf IV} from Theorem \ref{t:main} implies that $M_n$ cannot grow faster than a linear function in $n$. Taking into account Property {\bf I} from Theorem \ref{t:main}, $d$-volume of this cover can be estimated as follows:
$$
\sum_{n\ge n'} M_n\left(\frac{|J|}{N(n)}\right)^d\le  \sum_{n\ge n'} \const \cdot n \frac{|J|^d}{N(n)^d} \le \const' \sum_{n\ge n'} n \exp(-d \delta_0 \sqrt[4]{n}).
$$
Therefore it tends to zero as $n'$ tends to $\infty$. Since this holds for any $d>0$, we have $\text{dim}_H\, H_\eps=0$.

If $a\not\in H_\eps$, then due to Property \ref{i:m2} from Theorem \ref{t:main} for all sufficiently large $n$ we have $\frac{1}{n}\left(\log \|T_{n,a,\omega} \|-L_n(a)\right)\ge -\eps,$ hence
$$
\liminf_{n\to\infty} \frac{1}{n}\left(\log \|T_{n,a,\omega} \|-L_n(a)\right)\ge -\eps.
$$
Taking a countable union of sets $H_\eps$ over a sequence of values of $\eps>0$ that tend to zero, we get a set of zero Hausdorff dimension that contains all values of $a\in J$ such that
$$
\liminf_{n\to\infty} \frac{1}{n}\left(\log \|T_{n,a,\omega} \|-L_n(a)\right)< 0.
$$
This proves the last part of Theorem \ref{t:product}.
\end{proof}

\section{Finite products: proof of Theorem \ref{t:main}}\label{s.proofof114}

\subsection{Finite products of random matrices}\label{s:finite-products}

Here we prove Theorem~\ref{t:main}. This theorem is a non-stationary version of~\cite[Theorem~1.5]{GK}, and its proof closely follows~\cite{GK}.

First, let us remind some notation from Section~\ref{ss.introparamfurst}. Together with the initial linear dynamics of $\SL(2,\R)$-matrices $A(a)$, $a\in J$, we consider
their projectivizations that act on the circle of directions $\Sc\cong \mathbb{RP}^1$, and lift this action to the action on
the real line $\R$ for which $\Sc=\R/\Z$:
let
$$
f_{A,a}:\Sc\to \Sc
$$
be the map induced by $A(a):\mathbb{R}^2\to \mathbb{R}^2$, and let
$$
\tf_{A,a}:\mathbb{R}\to \mathbb{R}
$$
be the lift of $f_{A, a}:\mathbb{S}^1\to \mathbb{S}^1$.
The lifts $\tf_{A, a}$ can be chosen continuous in $a\in J$ and so that $\tf_{A, b_-}(0)\in [0,1)$. Also, given $\omega=(A_1, A_2, \ldots)\in \Omega=\mA^{\mathbb{N}}$, denote by
$$
f_{n, a, \omega}:\Sc\to \Sc
$$
the map induced by $T_{n,a, \omega}:\mathbb{R}^2\to \mathbb{R}^2$, $T_{n,a, \omega}=A_n(a)\cdot \ldots\cdot A_1(a)$,  and by
$$
\tf_{n, a,  \omega}:\mathbb{R}\to \mathbb{R}
$$
 the lift of $f_{n, a, \omega}:\mathbb{S}^1\to \mathbb{S}^1$, $\tf_{n,a, \omega}=\tf_{A_n, a}\circ \ldots \circ \tf_{A_1, a}$.

 For any fixed value of parameter $a\in J$, the (exponential)
growth of norms of $T_{n,\omega, a}$ is related to the (exponential) contraction on the circle of the projectivized dynamics.
Namely, standard easy computation shows that for a unit vector $v_0$ in the direction given by the point $x_0$, one has
\begin{equation}\label{eq:der-norm}
f'_{n,a,\omega} (x_0) = \frac{1}{|T_{n,a,\omega}(v_0)|^2}.
\end{equation}

Fix some point $x_0\in \Sc$, for example, the point that corresponds to the vector $\left( \begin{smallmatrix} 1 \\ 0 \end{smallmatrix} \right)$. Denote by $\tx_0\in [0,1)$ its lift to $\mathbb{R}^1$.
Recall that the interval $J=[b_-, b_+]$ was divided into $N=[\exp(\sqrt[4]{n})]$ equal intervals $J_1,\dots,J_N$ that were denoted by $J_i=[b_{i-1},b_{i}]$, $i=1, \ldots, N$.

Let $\tx_{m,i}$ be the image of $\tx_0$ after $m$ iterations of the lifted maps that correspond to the value of the parameter $b_i$,
$$
\tx_{m,i}:=\tf_{m,b_i,\omega}(\tx_0)
$$
(we omit here the explicit indication of the dependence on the $\omega$), and
let
\begin{equation}\label{e.Xmi}
X_{m,i}:=[\tx_{m,i-1},\tx_{m,i}]
\end{equation}
be the interval that is spanned by $m$-th (random) image of the initial point $\tx_0$ while the
parameter~$a$ varies in~$J_i=[b_{i-1},b_{i}]$.

%\begin{proof}[Proof of Theorem~\ref{t:main}]
The main step in the proof of Theorem~\ref{t:main} is the following proposition, describing possible types of behavior for lengths of the intervals $X_{m,i}$. It is a word-for-word analogue of Proposition~4.1 of \cite{GK}, that still holds in the non-stationary setting:
\begin{prop}[Types of the behavior]\label{p:classes}
For any $\eps'>0$ there exists $c_1>0$ such that for any sufficiently large $n$
with the probability at least $1-\exp(-c_1 \sqrt[4]{n})$ the following  holds. For each $i=1,\dots,N$ the lengths $|X_{m,i}|$ behave in one of the three possible ways:
\begin{itemize}
\item[\textbf{$\bullet$}]\textbf{(Small intervals)} The lengths $|X_{m,i}|$ do not exceed $\eps'$ for all $m=1,\dots,n$;
%\vspace{1pt}
\item[\textbf{$\bullet$}]\textbf{(Opinion-changers)} There is $m_0$ such that $|X_{m_0, i}|>\eps'$, and
$$
|X_{m,i}|<\eps' \quad \text{if } m<m_0 \ \text{ or } \  m>m_0+\eps'n;
$$
\item[\textbf{$\bullet$}]\textbf{(Jump intervals)} There is $m_0$ such that $|X_{m_0,i}|>\eps'$, and
$$
|X_{m,i}|<\eps' \quad \text{if } m<m_0,
$$
$$
1<|X_{m,i}|<1+\eps' \quad \text{if } m>m_0+\eps'n.
$$
\end{itemize}
%Also, the number $M$ of the jump intervals is $n\eps'$-close to the full length~$|\tx_{N,n}-\tx_{0,n}|$.
\end{prop}
\begin{remark}\label{r.proofoftypes}
Notice that while the statement of Proposition \ref{p:classes} is a verbatim repetition of \cite[Proposition~4.1]{GK}, one cannot just give a reference to \cite{GK}, since \cite[Proposition~4.1]{GK} was proven in the stationary setting. However, the only part of the proof of \cite[Proposition~4.1]{GK} that has to be modified is the proof of \cite[Corollary~4.25]{GK}, and we prove its non-stationary analog, Corollary~\ref{c:l-a1}, in Section~\ref{ss.5.6}. That proves Proposition \ref{p:classes}.
\end{remark}

\begin{figure}[!h!]
\begin{center}
\includegraphics{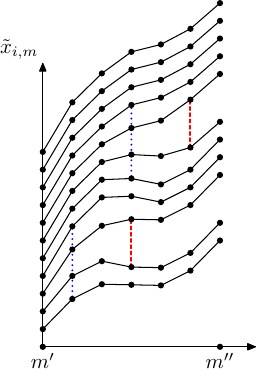}
\end{center}
\caption{Graphs of $\tilde x_{m,i}$ (consecutive iterations are linked), with the occurring suspicious intervals marked with blue (dotted) lines and the jumping ones with red (dashed) lines.}\label{f:jumps-2}
\end{figure}

\begin{figure}[!h!]
\begin{center}
\includegraphics[scale=0.8]{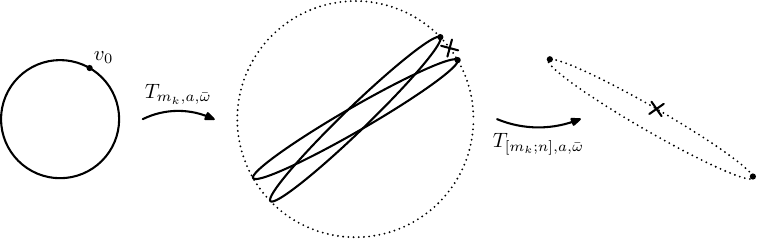}
\end{center}
\caption{Left: a unit circle with a marked point $x_0$. Center: its image after $m_k$ iterations under two different values of parameter $a=b_{i_k-1}$ and $a=b_{i_k}$,
together with a most contracted direction for $T_{[m_k,n],a,{\omega}}$ for some $a\in J_{i_k}$, marked by a cross. Right: final image after $n$ iterations; note that the images of $x_0$ are almost opposite, meaning that they have made a full turn on the projective line of the directions.}\label{f:jumps-1}
\end{figure}

\subsection{Large deviations: convenient versions}\label{ss.5.3}

Here we formulate several consequences of the Large Deviation Theorem (i.e. Theorem~\ref{t.2param})  in the context of random matrix products that will be specifically useful in our setting.

\begin{lemma}\label{l.prelimLD}
For any $\ep>0$ there exists $\zeta_1>0$ such that for all  sufficiently large $n\in \mathbb{N}$  the following holds. For any $a\in J$, any given $0\le m_1<m_2\le n$, and $\tx_0\in \mathbb{R}$ with probability at least $1-\exp(-\zeta_1n)$ one has
%  any given~, some $\zeta_1>0$, and any given~$m=1,\dots, n$ one has
\begin{equation}\label{eq:LD-RD}
\log \tf'_{[m_1, m_2],a,\omega}(\tf_{m_1, a, \omega}(\tx_0)) \in U_{\ep n} (-2L_{[m_1, m_2]}(a)).
\end{equation}
%with probability at least $1-\exp(-\zeta_1 n)$. Moreover, $\zeta_1>0$ can be chosen uniform in $a\in J$.
\end{lemma}
\begin{remark}
Notice that in the case $m_1=0, m_2=n$ the statement of Lemma~\ref{l.prelimLD} turns into Theorem~\ref{t.2param}. %\cite[Theorem~1.8]{GK22}.
\end{remark}
\begin{proof}[Proof of Lemma \ref{l.prelimLD}]
Let us recall that if $C_{\max}$ is defined as in (\ref{e.defCa}), then
$$
L_{[m_1, m_2]}(a)=\E\log\|T_{[m_1, m_2], a, \omega}\|\le (m_2-m_1)C_{\max}.
$$
Also, if $A\in \SL(2, \mathbb{R})$, $v$ is a unit vector, $f$ and $x$ are projectivizations of $A$ and $v$, and $\tf$ and $\tx$ are lifts of $f$ and $x$, then
$$
|\log \tf'(\tx)|=2|\log|Av|\,|\le 2\log\|A\|,
$$
see (\ref{eq:der-norm}).

Set $$\eps^*=\frac{\eps'}{4C_{\max}}.$$
If $m_2-m_1<\eps^*n$, then
\begin{multline*}
\left|\log \tf'_{[m_1,m_2], a, \omega}(\tf_{m_1, a, \omega}(\tx_0))\right|\le \sum_{k=m_1+1}^{m_2}\left|\log\tf'_{A_k, a}(\tf_{{k-1}, a, \omega}(\tx_0))\right|\le \\  (m_2-m_1)\cdot 2 C_{\max}\le (m_2-m_1)\frac{\ep}{2\eps^*}\le \frac{\ep}{2}n,
\end{multline*}
and
$$
2L_{[m_1, m_2]}(a)\le 2(m_2-m_1)C_{\max}=2(m_2-m_1)\frac{\eps'}{4\eps^*}\le \frac{\ep}{2}n.
$$
Therefore,
$$
\log \tf'_{[m_1,m_2], a, \omega}(\tf_{m_1, a, \omega}(\tx_0))\in U_{\ep n}(-2L_{[m_2,m_1]}(a)).
$$
If $\eps^*n\le m_2-m_1\le n$, then by Theorem~\ref{t.2param} %\cite[Theorem~1.8]{GK22}
for some $\zeta>0$  we have
\begin{multline*}
\P\left(\log \tf'_{[m_1,m_2], a, \omega}(\tf_{m_1, a, \omega}(\tx_0))\not\in U_{\ep n}(-2L_{[m_2,m_1]}(a))\right)\le \\
\P\left(\log \tf'_{[m_1,m_2], a, \omega}(\tf_{m_1, a, \omega}(\tx_0))\not\in U_{\ep (m_2-m_1)}(-2L_{[m_2,m_1]}(a))\right)\le \\ e^{-\zeta(m_2-m_1)}\le e^{-\zeta\eps^*n}.
\end{multline*}
Hence, Lemma \ref{l.prelimLD} holds with $\zeta_1=\eps^*\zeta$.
\end{proof}
Let us recall that the interval $J$ is divided into $N=[\exp(\sqrt[4]{n})]$ equal subintervals $J_1,\dots,J_N$ detoted  $J_i=[b_{i-1},b_{i}]$, $i=1, \ldots, N$. With large probability (\ref{eq:LD-RD}) holds simultaneously for all possible $m_1, m_2$ with $0\le m_1<m_2\le n$ and all parameter values that form the grid $\{b_0, b_1, \ldots, b_N\}$. Namely, the following statement holds:
\begin{lemma}\label{l:LD-RD}
For any $\ep>0$ there exists $\zeta_2>0$ such that for all  sufficiently large $n\in \mathbb{N}$ the following holds. For a given $\tx_0\in \mathbb{R}$ with probability at least $1-\exp(-\zeta_2n)$ one has
%  any given~, some $\zeta_1>0$, and any given~$m=1,\dots, n$ one has
\begin{equation}\label{eq:LD-RD-2}
\log \tf'_{[m_1, m_2],b_i,\omega}(\tf_{m_1, b_i, \omega}(\tx_0)) \in U_{\ep n} (-2L_{[m_1, m_2]}(b_i)).
\end{equation}
 for all $m_1, m_2$ with $0\le m_1<m_2\le n$ and all $i=0,1,\dots, N$.
%with probability at least $1-\exp(-\zeta_1 n)$. Moreover, $\zeta_1>0$ can be chosen uniform in $a\in J$.
\end{lemma}

\begin{proof}
Let $\zeta_1$ be given by Lemma \ref{l.prelimLD}, and take any positive $\zeta_2<\zeta_1$. For a given $a\in\{b_0,b_1,\dots, b_N\}$ and given $m\in \{1,\dots, n\}$ the event \eqref{eq:LD-RD} holds with probability at least $1-\exp(-\zeta_1 n)$. Intersecting the events \eqref{eq:LD-RD} for all  $a\in\{b_0,b_1,\dots, b_N\}$ and  all $m_1, m_2=0, 1,\dots, n$ with $m_1<m_2$  we observe that  \eqref{eq:LD-RD-2} holds with probability at least
\[
1-\frac{n(n+1)}{2}(N+1)\exp(-\zeta_1 n).
\]
Since $N=[\exp(\sqrt[4]{n})]$ and $\zeta_2<\zeta_1$, we get
$$
1-\frac{n(n+1)}{2}(N+1)\exp(-\zeta_1 n)> 1-\exp(-\zeta_2n)
$$
for all sufficiently large $n$.
\end{proof}

Exactly the same arguments that prove Lemma \ref{l.prelimLD}  and Lemma  \ref{l:LD-RD} provide a very similar but formally different statement:

\begin{lemma}\label{l.TvLm1m2}
For any $\ep>0$ there exists $\zeta_2>0$ such that for all  sufficiently large $n\in \mathbb{N}$ the following holds. For any given unit vector $v\in \mathbb{R}^2$, $\|v\|=1$, with probability at least $1-\exp(-\zeta_2n)$ for all $m_1, m_2$ with $0\le m_1<m_2\le n$ and all $i=0,1,\dots, N$ one has
%  any given~, some $\zeta_1>0$, and any given~$m=1,\dots, n$ one has
\begin{equation}
\log \|T_{[m_1, m_2],b_i,\omega}v\| \in  U_{\ep n} (L_{[m_1, m_2]}(b_i)).
\end{equation}
\end{lemma}

Combining the statements of Lemma \ref{l.TvLm1m2} and Corollary \ref{c:upper-finite-m} we get the following:
\begin{lemma}\label{l.UnLm1m2}
For any $\ep>0$ there exists $\zeta_3>0$ such that for all  sufficiently large $n\in \mathbb{N}$ the following holds. With probability at least $1-\exp(-\zeta_3n)$ for all $m_1, m_2$ with $0\le m_1<m_2\le n$ and all $i=0,1,\dots, N$ one has
%  any given~, some $\zeta_1>0$, and any given~$m=1,\dots, n$ one has
\begin{equation}
\log \|T_{[m_1, m_2],b_i,\omega}\|\in  U_{\ep n} (L_{[m_1, m_2]}(b_i)).
\end{equation}
 \end{lemma}

 Besides, Lemma \ref{l.TvLm1m2} allows to prove the following useful ``additivity'' property of the expectations $\{L_n\}$:
\begin{prop}\label{p.additivity}
For any $\beps>0$ there exists $n_0\in \mathbb{N}$ such that for any $n\ge n_0$, any $m_1, m_2, m_3\in \mathbb{N}$ with $0\le m_1<m_2<m_3\le n$, and any $a\in J$ we have
$$
0\le L_{[m_1, m_2]}(a)+L_{[m_2, m_3]}(a)-L_{[m_1,m_3]}(a)\le n\beps.
$$
\end{prop}
\begin{proof}
Recall that for any $a\in J$ we have $a\in [b_{i-1}, b_i]$ for some $i\le N$. Equicontinuity result provided by Lemma \ref{l.equim1m2} implies that for a given $\eps'>0$ and all sufficiently large $n\in \mathbb{N}$ we have
\begin{multline*}
  \left|(L_{[m_1, m_2]}(a)+L_{[m_2, m_3]}(a)-L_{[m_1,m_3]}(a))- \right.\\
  - \left.(L_{[m_1, m_2]}(b_i)+L_{[m_2, m_3]}(b_i)-L_{[m_1,m_3]}(b_i))\right|\le \\
  \le |(L_{[m_1, m_2]}(a) -L_{[m_1, m_2]}(b_i)|+|L_{[m_2, m_3]}(a)-L_{[m_2, m_3]}(b_i)|+\\
  +|L_{[m_1,m_3]}(a)-L_{[m_1,m_3]}(b_i)|
  \le \eps'(m_2-m_1)+\eps'(m_3-m_2)+\eps'(m_3-m_1)\le 3\eps'n.
\end{multline*}
Take any unit vector $v\in \mathbb{R}^2$, $\|v\|=1$. We have
$$
T_{[m_1, m_3], b_i, \omega}v=T_{[m_2, m_3], b_i, \omega}T_{[m_1, m_2], b_i, \omega}v,
$$
and hence
$$
\log\|T_{[m_1, m_3], b_i, \omega}v\|=\log\left\|T_{[m_2, m_3], b_i, \omega}\left(\frac{T_{[m_1, m_2], b_i, \omega}v}{\|T_{[m_1, m_2], b_i, \omega}v\|}\right)\right\|+\log\|T_{[m_1, m_2], b_i, \omega}v\|.
$$
Due to Lemma \ref{l.UnLm1m2}, with large probability we have
$$
\log \|T_{[m_1, m_2],b_i,\omega}\|\in  U_{\ep n} (L_{[m_1, m_2]}(b_i)), \log \|T_{[m_2, m_3],b_i,\omega}\|\in  U_{\ep n} (L_{[m_2, m_3]}(b_i)),
$$
and
$$
\log \|T_{[m_1, m_3],b_i,\omega}\|\in  U_{\ep n} (L_{[m_1, m_3]}(b_i)).
$$
Therefore,
$$
0\le L_{[m_1, m_2]}(b_i)+L_{[m_2, m_3]}(b_i)-L_{[m_1,m_3]}(b_i)\le 3n\eps',
$$
and hence
$$
0\le L_{[m_1, m_2]}(a)+L_{[m_2, m_3]}(a)-L_{[m_1,m_3]}(a)\le 6n\eps'.
$$
Taking $\eps'=\frac{\beps}{6}$, we get the desired result.
\end{proof}

\subsection{Distortion control}\label{ss:distortion}
Here we collect several distortion estimates that will be needed later.

\begin{lemma}[Distortion control]\label{l:DC}
For any $\omega\in \Omega$, $\omega=(A_1, A_2, \ldots)$, the following holds. Given $m'<m''$, $y_1<y_2$, and $\bar a_1<\bar a_2$, define the sequence of intervals $Y_{m}=[y_{m,1},y_{m,2}]$, $m=m',...,m''$,
 by
$$ %  \, y_{m',2}=y_2,
y_{m',j}=y_j, \quad  y_{m+1,j}=\tf_{A_m, \bar a_j}(y_{m,j}), \quad j=1,2, \ \  m=m',...,m''-1.
$$
Then for any $\bar a_3 \in [\bar a_1,\bar a_2]$, any $m=m',\dots,m''$, and any $y_3\in [y_1,y_2]$ we have
$$
\left| \log \tilde f'_{[m',m], \bar a_3,\omega}(y_3)- \log \tilde f'_{[m',m],\bar a_1,\omega}(y_1) \right| \le \kappa \sum_{k=m'}^{m''-1} |Y_k| + C |\bar a_2-\bar a_1| \cdot (m''-m'),
$$
where the constants $\kappa$ and $C$ are defined by %depend only on the family~$f$ (and not on $n$ or $\eps$); more precisely, one can take
$$
\kappa:= \sup_{y\in \mathbb{R}^1, \, \mu\in \mathcal{K},\, A\in \supp \mu,  \, a\in J} |\partial_y \log \tilde f'_{A, a}(y)|,
\quad C:= \sup_{y\in \mathbb{R}^1, \, \mu\in \mathcal{K},\, A\in \supp \mu, \,  a\in J} |\partial_a \log \tilde f'_{a,\omega}(y)|.
$$
\end{lemma}
\begin{proof}[Proof of Lemma \ref{l:DC}]
The proof is a verbatim repetition of the proof of Lemma 4.3 from \cite{GK}, with some obvious adjustments of the notation.
\end{proof}
Another estimate that we will need shows how fast nearby points can diverge under iterates of different but close maps.
\begin{lemma}\label{l.shift}
In notations of Lemma \ref{l:DC}, we have
\begin{equation}\label{eq:shift}
|y_{m'', 1}-y_{m'', 2}|\le \Lx^{m''-m'}|y_{m', 1}-y_{m', 2}|+\La(m''-m')\cdot \Lx^{m''-m'-1}|\bar a_2-\bar a_1|,
\end{equation}
where $$\Lx=  \sup_{y\in \mathbb{R}^1, \, \mu\in \mathcal{K},\, A\in \supp \mu,\,  a\in J}|\tf'_{A, a}(y)|$$ and $$\La=\sup_{y\in \mathbb{R}^1, \, \mu\in \mathcal{K},\, A\in \supp \mu,\,  a\in J}|{\partial_a}\tf_{A, a}(y)|$$ are the Lipschitz constants for the maps~$\tf_{A, a}(y)$ in space and parameter directions respectively.
\end{lemma}
\begin{proof}[Proof of Lemma \ref{l.shift}]
The proof is a verbatim repetition of the proof of Lemma 4.4 from \cite{GK}, with some obvious adjustments of the notation.
\end{proof}

\subsection{Uniform growth estimates}\label{ss:norms}
Here we deduce parts \ref{i:m2} and~\ref{i:m3}
of Theorem~\ref{t:main} from Proposition~\ref{p:classes}.

First let us show that the distortion control given by Lemma \ref{l:DC} together with Proposition~\ref{p:classes} allows us to use Lemma~\ref{l:LD-RD}
to estimate the derivatives at~$\tx_0$ at all parameter values $a\in J$:

\begin{prop}\label{p:derivatives-control}
There exists a constant $C_1$ such that for any $\varepsilon'>0$  the following property holds for all sufficiently large~$n$.
Assume that $\omega$ is such that the conclusions of Lemma~\ref{l:LD-RD} and Proposition~\ref{p:classes} hold.
Then, for any $a\in J$:
\begin{itemize}
\item If $a\in J_i$, and $J_i$ is either ``small'' or ``opinion-changing'' interval in terms of Proposition~\ref{p:classes}, then
\begin{equation}\label{eq:derivative-u}
\forall\ m=1,\dots, n \quad \log \tf'_{m,a,\omega}(\tx_0) \in U_{C_1\eps' n} (-2L_m(a)).
\end{equation}
\item If $a\in J_i$, and $J_i$ is a ``jump'' interval in terms of Proposition~\ref{p:classes}, with the associated moment $m_0$, then
\begin{equation}\label{eq:derivative-upl}
\forall\ m=1,\dots, \mjump \quad \log \tf'_{m,a,\omega}(\tx_0) \in U_{C_1\eps' n} (-2L_m(a)),
\end{equation}
where $\mjump:=m_0+\eps' n$
is what we will call  \emph{``jump index''} below,
and
\begin{equation}\label{eq:derivative-upr}
\forall\ m={\mjump}+1,\dots, n \quad \log \tf'_{[\mjump, m],a,\omega}(\tx_1) \in U_{C_1\eps' n} (-2L_{[\mjump, m]}(a)),
\end{equation}
for any $\tx_1\in X_{\mjump, i}'$, where $\mjump:=m_0+\eps'n$ and we denote $ X_{\mjump, i}':=[\tx_{\mjump,i-1}+1,\tx_{\mjump,i}]$.
\end{itemize}
\end{prop}

\begin{proof}[Proof of Proposition \ref{p:derivatives-control}]
%The upper bound for the norm comes from the proof of Proposition~\ref{p.upper}.
In the first case, regardless of whether the interval $J_i$ is a ``small'' one or an ``opinion-changer'', we have an
upper bound for the sum of the corresponding lengths
\begin{equation}\label{eq:sum-no-jump}
\sum_{m=0}^{n-1} |X_{m,i}|=\sum_{|X_{m,i}|<\eps'} |X_{m,i}|+\sum_{|X_{m,i}|\ge \eps'} |X_{m,i}| \le n\cdot \eps'+ n\eps'\cdot 1= 2n\eps'.
\end{equation}
%Hence, applying Lemma~\ref{l:DC} with $y_1=y_2=y_3=\bar{x}$, we get
%$$
%\left| \log f'_{m,a,\omega}(\bar x) - \log f'_{m,b_i,\omega}(\bar x) \right| \le
%$$
Lemma~\ref{l:DC} implies that for all $a\in J_i$ and all $m=1, \ldots, n$ we have
$$
|\log \tf'_{m,a,\omega}(\tx_0)-\log \tf'_{m,b_i,\omega}(\tx_0)| \le 2\kappa \eps'  n + C \cdot \frac{|J|}{N} n.
$$

Due to Lemma \ref{l:equicontinuous}, the sequence of functions $\{\frac{1}{n}L_n(a)\}$ is equicontinuous. Therefore, for a given $\eps'>0$ and any sufficiently large $n$ we have:
 $$
 \left|\frac{1}{n}L_n(a)-\frac{1}{n}L_n(b_i)\right|\le \eps', \ \ \text{and}\ \ \ \frac{|J|}{N}<\eps'.
 $$
Together with the estimate~\eqref{eq:LD-RD-2} this gives
%
%and with the inequalities $|\lR(a)-\lR(b_i)|\le \eps'$, \, $\frac{|J|}{N}<\eps'$, that hold
%for all sufficiently large $n$ (the former due to the continuity of $\lR(\cdot)$), we get
%$$
%\log f'_{m,a_i,\omega}(\bar x) \le -2\lambda_F(a)m +n\eps',
%$$
%thus
\begin{multline}\label{eq:log-f-prim}
|\log \tf'_{m,a,\omega}(\tx_0) + 2L_m(a)| \le |\log \tf'_{m,a,\omega}(\tx_0) - \log \tf'_{m,b_i,\omega}(\tx_0) |+
\\
 +  |\log \tf'_{m,b_i,\omega}(\tx_0) + 2 L_m(b_i)| + |2L_m(b_i) - 2L_m(a)|\le \\ 2\kappa \eps'  n + C \eps' n + \eps' n + 2\eps' m   \le  (2\kappa +C+3) \eps' n.
\end{multline}
Therefore~\eqref{eq:derivative-u} holds once $C_1>2\kappa+C+3$.

Suppose now that $J_i$ is a ``jump'' interval. Checking~\eqref{eq:derivative-upl} goes exactly in the same way as in~\eqref{eq:sum-no-jump}:
$$
\sum_{m=0}^{\mjump} |X_{m,i}|=\sum_{m=0}^{m_0-1} |X_{m,i}|+\sum_{m=m_0}^{\mjump-1} |X_{m,i}| \le n\cdot \eps'+ n\eps'\cdot 2= 3n\eps'.
$$
Hence, in the same way as in~\eqref{eq:log-f-prim}, we have for any $m\le \mjump$
$$
|\log \tf'_{m,a,\omega}(\tx_0) + 2 L_m(a)| \le 3\kappa \eps'  n + C \eps' n + \eps' n + 2\eps' m   \le  (3\kappa +C+3) \eps' n,
$$
and we have the desired~\eqref{eq:derivative-upl} once $C_1>3\kappa+C+3$.

Finally, the intervals $X_{m,i}'$ for $m\ge \mjump$ also satisfy the assumptions of Lemma~\ref{l:DC}. One has
$$
\sum_{m=\mjump}^{n} |X_{m,i}| \le \eps' n,
$$
and thus (again, together with~\eqref{eq:LD-RD-2}) we get
$$
|\log \tf'_{[\mjump,m],a,\omega}(\tx_1) +2 L_{[\mjump, m]}(a)| \le \kappa \eps'  n + C \eps' n + \eps' n + 2\eps' m   \le  (\kappa +C+3) \eps' n.
$$
This proves~\eqref{eq:derivative-upr} for any $C_1>\kappa +C+3$, and thus concludes the proof of Proposition \ref{p:derivatives-control}.
\end{proof}

%the equality~\ref{eq:der-norm} immediately implies a lower bound
Proposition \ref{p:derivatives-control} implies the parts \ref{i:m2} and \ref{i:m3} of Theorem~\ref{t:main}. Indeed,
 for any $A\in \SL(2,\R)$ and for any vector $v\neq 0$ one has
\begin{equation}\label{eq:der-norms}
f_A'(x_v)= \frac{|v|^2}{|Av|^2},
\end{equation}
where $x_v\in \Sc$ is the direction corresponding to the vector~$v$. In particular, for any point $x$ on the circle
one has $\log \|A\| \ge -\frac{1}{2} \log f_A'(x)$ (as the right hand side of~\eqref{eq:der-norms} is not less than~$\frac{1}{\|A\|^2}$).
%taking~$v$ to be unit vector of the direction $\bar x$,
%one gets from~\eqref{eq:derivative-u} the lower bound for the norm
In particular, for any $m, a,\omega$ we have
\begin{equation}\label{eq:norms-lower}
\log \|T_{m,a,\omega}\| \ge -\frac{1}{2} \log f'_{m,a,\omega}(\bar x).
\end{equation}

Joining this estimate with~\eqref{eq:derivative-u}, we obtain a lower bound for the norm
$$
\log \|T_{m,a,\omega}\| \ge -\frac{1}{2} \cdot (-2L_m(a) + C_1 n \eps') = L_m(a) - \frac{C_1}{2} \eps' n.
$$
Hence, to obtain the lower bound in the ``Uniformity'' part, it suffices to take
$$
\eps'<\frac{2\eps}{C_1}.
$$

On the other hand, Proposition~\ref{p.upper} states that the upper bound
$$
\log \|T_{m,a,\omega}\| < L_m(a) + n\eps
$$
holds with the probability $1-\exp(c_3 n)$. We thus obtain the desired
$$
\log \|T_{m,a,\omega}\| \in U_{n \eps}(L_m(a))
$$
for all $a\in J_i$, provided that the interval $J_i$ was ``small'' or ``opinion-changing''. Now, assume that $a\in J_i$,
and the interval $J_i$ is a ``jump'' interval. Then again, joining~\eqref{eq:norms-lower} with~\eqref{eq:derivative-upl}--\eqref{eq:derivative-upr}, we obtain
$$
\forall m=1,\dots, {\mjump} \quad \log \|T_{m,a,\omega}\| \ge L_m(a)  - \frac{C_1}{2} \eps' n> L_m(a)  - \eps n
$$
and
$$
\forall m={\mjump}+1,\dots, n \quad \log \|T_{[{\mjump},m],a,\omega}\| \ge L_{[\mjump, m]}(a) - \frac{C_1}{2} \eps' n> L_{[\mjump, m]}(a) -\eps n,
$$
where the last inequalities come from the choice of~$\eps'$.

Again, Proposition~\ref{p.upper} gives the upper bounds
$$
\forall m=1,\dots, {\mjump} \quad \log \|T_{m,a,\omega}\| < L_m(a) + n\eps
$$
and
$$
\forall m={\mjump}+1,\dots, n \quad \log \|T_{[{\mjump},m],a,\omega}\| < L_{[\mjump, m]}(a)  + n\eps.
$$
This implies the desired ``Uniformity'' estimates
$$
\forall m=1,\dots, {\mjump} \quad \log \|T_{m,a,\omega}\| \in U_{n\eps}(L_m(a))
$$
$$
\forall m={\mjump}+1,\dots, n \quad \log \|T_{[{\mjump},m],a,\omega}\| \in U_{n\eps}( L_{[\mjump, m]}(a)),
$$
thus concluding the proof of parts \ref{i:m2} and \ref{i:m3}  of Theorem~\ref{t:main}.

\subsection{Cancelation lemmata}\label{ss:cancellation}

Here we prove the ``Cancellation'' part~\ref{i:m4} of Theorem~\ref{t:main}. The content of this subsection is parallel to the Section~4.5 from \cite{GK}.

For any $A\in \SL(2, \mathbb{R})$ denote by $f_A$ the corresponding projective map of $S^1$. Also, for $A\notin SO(2,\R)$ let $x^-(A)\in \Sc$ be the point where $f_A$ has the largest derivative, and $x^+(A)\in \Sc$ be the image under $f_A$ of the point where $f_A$ has the smallest derivative. Equivalently, $x^+(A)$ is the direction of the large axis of the ellipse, obtained by applying $A$ to the unit circle, and $x^-(A)=x^+(A^{-1})$.

Let $\alpha$ and $\beta$ be the angles of $x^-(A)$ and $x^+(A)$ respectively. Then, it is easy to see that
$$
A=\pm R_{\beta} \left(
\begin{matrix}
\|A\| & 0 \\
0 & \|A\|^{-1}
\end{matrix}
\right) R_{\alpha+\pi/2}^{-1}.
$$
In particular, the following statement holds:
\begin{lemma}[Cancellation for matrices, Lemma 4.13 from \cite{GK}]\label{l:cancel}
Let $A,B\in\SL(2,\R)\setminus SO(2,\R)$ be two matrices such that $x^+(A)=x^-(B)$. Then
$$
\|BA\|=\max \left(\frac{\|B\|}{\|A\|}, \frac{\|A\|}{\|B\|}\right).
$$
\end{lemma}

We will also use the following lemma: %, saying, roughly speaking, that a direction that is expanded is sent close to the maximally expanded direction.
\begin{lemma}[Lemma 4.14 from \cite{GK}]\label{l:x-C-image}
Let $A\in \SL(2,\R)\setminus SO(2,\R)$, $x\in \Sc$ be a point on the circle, and $v_x$ be some vector in the corresponding direction. Then:
\begin{itemize}
\item $\dist (f_A(x),x^+(A))\le \frac{\pi}{2} \cdot \frac{|v_x|/|Av_x|}{\|A\|}$,
\item $\dist (x,x^-(A))\le \frac{\pi}{2} \cdot \frac{|Av_x|/|v_x|}{\|A\|}$,
\item If we have $f_A'(x)<\frac{1}{C}$, then $\|A\|\ge \sqrt{C}$ and $x^+(A)$ belongs to $\frac{\pi }{2C}$-neighborhood of $f_A(x)$.
\end{itemize}
\end{lemma}

Let us now prove the ``Cancellation'' part~\ref{i:m4} of the conclusions of Theorem~\ref{t:main}; to do that, we have to handle the ``jump'' intervals. %in terms of Proposition~\ref{p:classes}.
Namely, assume that the conclusions of Lemma~\ref{l:LD-RD} hold, and $J_i$ is a ``jump'' interval in terms of Proposition~\ref{p:classes}. Recall that we denoted $\mjump:=m_0+\eps' n$, where $m_0$ is given by the definition of ``jump interval'' in Proposition~\ref{p:classes}.
Notice (we will use it later) that by increasing
$\mjump$ by~$1$ we can (and do) assume that
\begin{equation}\label{e.assumption}
|X_{\mjump,i}|\ge 1+ \delta \frac{|J|}{N},
\end{equation}  %(and we will do so).
where $\delta>0$ is given by the monotonicity assumption \ref{B:Monotonicity}.

We start by handling the case when the jump moment happens too close to the first or the last iteration.
\begin{lemma}\label{l:close}
Let $\eps', \eps''>0$, and assume that the conclusions of  Proposition~\ref{p:classes} hold,
that $J_i$  is a ``jump'' interval with associated index $m_0$, and set $\mjump:=m_0+\eps' n$.
Assume also that the conclusions of the part~\ref{i:m3} of Theorem~\ref{t:main} hold with the value $\eps'$ instead of $\eps$, and that $\mjump\le \eps'' n $ or $\mjump \ge (1-\eps'') n$.
Then the conclusions of the ``Cancellation'' part~\ref{i:m4} of Theorem~\ref{t:main} are satisfied for arbitrary $a\in J_i$, provided that one has
$$
2 \eps' +2C_{\max} \eps''<\eps. %\ \eps'<\frac{\eps}{10}, \ and\ \frac{h}{4C_1} \eps'' > \eps',
$$
%where $C_1>1$ is given by Proposition \ref{p:derivatives-control}, and $h>0$ is given by Theorem \ref{t.2param}.
\end{lemma}
\begin{proof}
Consider first the case $\mjump\le \eps''n$.
For $m\le \mjump$, due to the conclusions of part~\ref{i:m3}
we have
$$
\log \|T_{m,a,\omega}\|\le n\eps' + L_{\mjump}(a), \quad \psi_{\mjump}(m, a)=L_{m}(a),
$$
and hence
$$
\left| \log \|T_{m,a,\omega}\| - \psi_{\mjump}(m,a) \right| \le n\eps' + 2 L_{m}(a) \le (\eps'+ 2 C_{\max} \eps'') n < \eps n
$$
thus guaranteeing the desired~\eqref{e.IV}. On the other hand, once $m\ge \mjump$, we have
$$
\log \|T_{\mjump,a,\omega} \| \le n\eps' + L_{\mjump}(a), \quad \log \|T_{[\mjump, m],a,\omega}\| \in U_{n\eps'} (L_{[\mjump, m]}(a)),
$$
hence
\begin{equation}\label{eq:left}
\log \|T_{m,a,\omega}\| \in U_{2 n\eps'+L_{\mjump}(a)} (L_{[\mjump, m]}(a)).
\end{equation}
Therefore, we have
\begin{multline*}
\left| \log \|T_{m,a,\omega}\| -\psi_{\mjump}(m,a)\right|=\left| \log \|T_{m,a,\omega}\| -|L_{\mjump}(a)-L_{[\mjump, m]}(a)| \right| \le \\
2 n\eps'+L_{\mjump}(a) + L_{\mjump}(a) \le 2 n\eps'+2C_{\max}\mjump\le (2 \eps'+2C_{\max}\eps'')n<\eps n.
\end{multline*}
The case $\mjump\ge (1-\eps'') n$ is completely analogous. %handled in the same way: for $m\le \mjump$, the conclusions of part~\ref{i:m4} coincide with the conclusions of part~\ref{i:m3}. At the same time, if $m\ge \mjump$, one has
%$$
%T_{m,a,\bo} = T_{[\mjump, m],a,\bo} T_{\mjump,a,\bo},
%$$
%and hence $\log \|T_{m,a,\omega}\|$ is $2n\eps'+ \lambda_F(a) n\eps''$-close to $\lambda_F(a) \mjump$. And the latter is $\lambda_F(a) n\eps''$-close to $\lambda_F(a)\psi_{\mjump}(m)$, finally implying the desired
%$$
%\left| \log \|T_{m,a,\omega}\| - \lambda_F(a)\psi_{\mjump}(m) \right| \le (2\eps'+ 2\lambda_F(a) \eps'') n.
%$$
\end{proof}

Let us now consider the case when the jump moment is ``sufficiently away'' from the endpoints of the interval of iterations, $\eps''n<\mjump<(1-\eps'') n$.
First, we find the corresponding value of the parameter $a\in J_i$.

Notice that if $h$ given by Theorem \ref{t.2param}, then for any $0\le m_1<m_2\le n$ and any $a\in J$ we have
$$
L_{[m_1, m_2]}(a)\ge h(m_2-m_1).
$$
Moreover, for some uniform $c_0>0$ and any $0\le m_1\le m_2\le n$ and any $a\in J$ we have
\begin{equation}
(m_2-m_1)\frac{h}{2}-c_0\le L_{m_2}(a)-L_{m_1}(a)\le L_{[m_1, m_2]}(a)\le (m_2-m_1) C_{\max},
\end{equation}
see \cite[Lemma 3.7]{GK22}.
%Denote
%$$
%\lambda_{\min}:=\min_{a\in J} \lambda_F(a), \quad \lambda_{\max}:=\max_{a\in J} \lambda_F(a).
%$$
\begin{lemma}\label{l:m-half}
Let $\eps', \eps''>0$ satisfy
\begin{equation}\label{eq:eps-pp}
\frac{h}{4C_1} \eps'' > \eps',
\end{equation}
where $C_1>1$ is given by Proposition \ref{p:derivatives-control}.
For all sufficiently large $n$, the following statement holds.

Assume that the conclusions of Lemma~\ref{l:LD-RD} and of Proposition~\ref{p:classes} hold,
 $J_i$  is a ``jump'' interval with associated index $m_0$, and set $\mjump:=m_0+\eps' n$.
 Assume also that the conclusions of the part~\ref{i:m3} hold with the value~$\eps'$ instead of~$\eps$. Then there exists $a\in J_i$ such that
$$
x^+(T_{\mjump,a,\omega})=x^-(T_{[\mjump, \mjump'],a,\omega}),
$$
where $\mjump':=n$ if $L_{[\mjump, m]}(b_i)\le L_{\mjump}(b_i)$ for all $m=\mjump+1, \ldots, n$, and otherwise $\mjump':=\min\left\{m>\mjump\ |\ L_{[\mjump, m]}(b_i)> L_{\mjump}(b_i)\right\}$.
\end{lemma}
\begin{proof}[Proof of Lemma \ref{l:m-half}]
Notice that equicontinuity of the functions $\frac{1}{m_2-m_1}L_{[m_1, m_2]}(a)$, see Lemma \ref{l.equim1m2}, implies that for large enough $n$ and any $m$ between $\mjump$ and $n$ we have
\begin{equation}\label{e.Lab}
|L_{[\mjump, m]}(a)-L_{[\mjump, m]}(b_i)|<\eps'(\mjump'-m)<\eps'n.
\end{equation}
In particular, we have
$$
|L_{[\mjump, \mjump']}(a)-L_{[\mjump, \mjump']}(b_i)|<\eps'(\mjump'-\mjump)<\eps'n.
$$
We claim that for any $a\in J_i$ one has
\begin{equation}\label{e.Lbm}
L_{[\mjump, \mjump']}(a)>(\eps''h-\eps')n.
\end{equation}
Indeed, if $\mjump'<n$, then
$$
L_{[\mjump, \mjump']}(a)>L_{\mjump}(b_i)-\eps'n>\mjump h-\eps'n>(\eps''h-\eps')n,
$$
and if $\mjump'=n$, then
$$
L_{[\mjump, \mjump']}(a)>(n-\mjump)h> \eps''hn> (\eps''h-\eps')n.
$$
Note that the uniformity estimates imply that the products $T_{\mjump,a,\omega}$ and $T_{[\mjump, \mjump'],a,\omega}$
are of norm bounded away from~1 for all $a\in J_i$. Indeed, the conclusions of the part~\ref{i:m3} imply that
$$
\log \|T_{\mjump,a,\omega} \| > L_{\mjump}(a)- n \eps' > n( \eps'' h -\eps') >0,
$$
and
$$
\log \|T_{[\mjump,\mjump'], a,\omega} \| > L_{[\mjump, \mjump']}(a)-\eps'n>(\eps''h-2\eps')n>0,
$$
where we used (\ref{e.Lbm}), and in both cases the last inequalities are due to~\eqref{eq:eps-pp}.

%Therefore, if $\mjump'<n$, then
%\begin{multline*}
%\log \|T_{[\mjump,\mjump'], a,\omega} \| > L_{[\mjump, \mjump']}(a)-\eps'n>L_{[\mjump, \mjump']}(b_i)-2\eps'n>L_\mjump(b_i)-2\eps'n\ge  \\
%  \ge h\mjump-\eps'n\ge (h\eps''-\eps')n>0
%\end{multline*}
%And if $\mjump'=n$, then
%$$
%\log \|T_{[\mjump,\mjump'], a,\omega} \| > L_{[\mjump, \mjump']}(a)-\eps'n>(\mjump' -\mjump)h - n \eps' > ( \eps'' h -\eps')n >0,
%$$
%where in both cases the last inequalities are due to~\eqref{eq:eps-pp}.

Hence the directions
$x^+(T_{\mjump,a,\omega})$ and $x^-(T_{[\mjump,\mjump'],a,\omega})$ depend continuously on~$a\in J_i$. To shorten the notations, we denote
$$
x^+(a):=x^+(T_{\mjump,a,\omega}), \quad x^-(a):=x^-(T_{[\mjump, \mjump'],a,\omega}).
$$

 Lemma~\ref{l:x-C-image} implies that $x^+(a)$ stays $\frac{\pi}{2} f_{\mjump,a,\omega}'(x_0)$-close to the image~$f_{\mjump,a,\omega}(x_0)$
as $a$ varies in $J_i$. At the same time, for any $a\in J_i$, due to Proposition \ref{p:derivatives-control}, we have
\begin{multline*}
  \frac{\pi}{2}f_{\mjump,a,\omega}'(x_0)< \frac{\pi}{2}\exp(-2L_{\mjump}(a)+ C_1n \eps') <\frac{\pi}{2}\exp\left(-2\mjump h + C_1n \eps' \right) < \\
< \frac{\pi}{2}\exp\left((-2h\eps'' + C_1\eps')n \right)  <\frac{\delta |J|}{2N},
\end{multline*}
where we used the assumption $\mjump \ge n\eps''$, inequality \eqref{eq:eps-pp}, and the subexponential growth of~$N=\exp(\sqrt[4]{n})$.

At the same time, due to (\ref{e.assumption}),
 we have $|X_{\mjump, i}'|\ge  \frac{\delta |J|}{N}$.
Hence, as $a$ varies over $J$, the point $x^+(a)$ passes through the midpoint
$$
r:=\pi\left(\frac{(\tx_{\mjump, i-1}+1)+\tx_{\mjump, i}}{2}\right)
$$
of the interval $\pi(X_{\mjump,i}')=\pi([\tx_{\mjump, i-1}+1,\tx_{\mjump, i}])$
at least twice, making the full turn in between; see Figure~\ref{f:x-plus-minus}.

\begin{figure}[!h!]
\begin{center}
\includegraphics{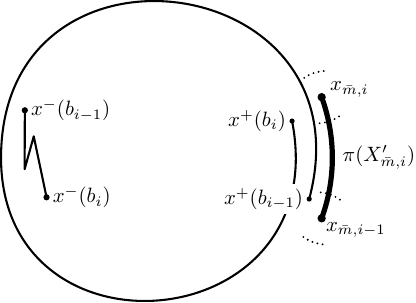}
\end{center}
\caption{While the parameter $a$ varies over a jump interval $J_i$, the $x^+(a):=x^{+}(T_{\mjump,\omega,a})$ makes more than a full turn,
staying in a neighborhood of the corresponding image $f_{\mjump,a,\omega}(\bar x)$. At the same time, the point $x^-(a):=x^-(T_{[\mjump,\mjump'],a,\omega})$
never enters the interval $X'_{\mjump, i}$ (the arc shown in bold).}\label{f:x-plus-minus}
\end{figure}

At the same time, we know from the
distortion control estimates in the proof of Proposition~\ref{p:derivatives-control}
that the derivatives of $f_{[\mjump,\mjump'],a,\omega}$ on $X_{\mjump,i}'$ do not exceed
$$
\exp(-2L_{[\mjump, \mjump']}(a)+C_1\eps'n)< \exp(-2(\eps''h-\eps')n +C_1\eps' n)  < 1,
$$
 using (\ref{e.Lbm}) for the first inequality and~\eqref{eq:eps-pp} for the last one.

Hence the point $x^-(a)$ never crosses $r$ for $a\in J_i$. Thus, we can choose the lifts $\tilde x^+(a)$ and $\tilde x^-(a)$ on the real line of
$x^+(a)$, $x^-(a)$ respectively such that the difference $\tilde x^+(a)-\tilde x^-(a)$ changes sign while $a$ varies in $J_i$. Hence, there
exists a point $a\in J_i$ for which the directions $x^+(a)$ and $x^-(a)$ coincide.
\end{proof}

We are now ready to conclude the proof of the ``Cancellation'' part~\ref{i:m4}. Take $\eps', \eps''>0$ such that~\eqref{eq:eps-pp}
holds, as well as
$$
\eps'<\frac{\eps}{10}, \quad 2(\eps'+ C_{\max} \eps'') < \eps.
$$

Assume that the conclusions of Lemma~\ref{l:LD-RD} hold and of Proposition~\ref{p:classes} hold,
that $J_i$ in its terms is a ``jump'' interval, with $\mjump:=m_0+\eps' n$ being the corresponding jump moment.
Assume also that the conclusions of the part~\ref{i:m3} hold with the value~$\eps'$ instead of~$\eps$.

Let us show that then the part~\ref{i:m4} of conclusions of Theorem~\ref{t:main} are satisfied.
Indeed, if $\mjump\le \eps'' n$ or $\mjump\ge (1-\eps'') n$, this directly follows from Lemma~\ref{l:close}. Otherwise
we can apply Lemma~\ref{l:m-half}; take $a_i$ to be the value of the parameter $a$ given by Lemma~\ref{l:m-half},
and let us check that \eqref{e.IV} holds for all $m= 1,\dots, n$.

Note that for any $m\in [1,\mjump]$ the estimates of the part~\ref{i:m3} imply
\begin{equation}\label{eq:m1}
\log \|T_{m,a_i,\omega}\| \in U_{\eps' n}(L_m(a_i)) = U_{\eps' n}(\psi_{\mjump}(m, a_i)).
\end{equation}

We have now to handle the case $m\in [\mjump,n]$. The next steps depend on whether $\mjump'<n$ or $\mjump'=n$. %~$L_{\mjump}(a_i)$ is greater or less than $L_{[\mjump, n]}(a_i)$.
%Now, due to the estimates of the part~\ref{i:m3} we have
%$$
%\log \|T_{\mjump,a_i,\bo}\| \in U_{n\eps'}(\lambda_F(a_i) \mjump), \quad  \log \|T_{[\mjump;\mjump'],a_i,\bo}\| \in U_{n\eps'}(\lambda_F(a_i) (\mjump'-\mjump)),
%$$
%and thus finally
%\begin{equation}\label{eq:bm-p}
%\log \|T_{\mjump',a_i,\bo}\| \in U_{2n\eps'} (\lambda_F(a_i) (2\mjump -\mjump'))= U_{2n\eps'} (\lambda_F(a_i) \psi_{\mjump}(\mjump')).
%\end{equation}

Consider first the case $\mjump'<n$.
Notice that by definition of $\mjump'$ we have $L_{[\mjump, \mjump']}(b_i)>L_{\mjump}(b_i)$ and $L_{[\mjump, \mjump'-1]}(b_i)\le L_{\mjump}(b_i)$, hence
$$
|L_{[\mjump, \mjump']}(b_i)-L_{\mjump}(b_i)|\le C_{\max}.
$$
Together with (\ref{e.Lab}) this implies that
\begin{equation}\label{e.symmetry}
|L_{[\mjump, \mjump']}(a_i)-L_{\mjump}(a_i)|\le 3\eps'n.
\end{equation}
Then, applying Lemma~\ref{l:cancel} and the uniformity estimates on the intervals $[1,\mjump]$ and $[\mjump, \mjump']$, we get
\begin{multline}\label{eq:almost-zero}
\log \|T_{\mjump',a_i,\omega}\| = \left| \log \|T_{\mjump,a_i,\omega}\| -  \log \|T_{[\mjump,\mjump'],a_i,\omega}\| \right|
\\
\le \left| \log \|T_{\mjump,a_i,\omega}\| - L_{\mjump}(a_i) \right| +\left| \log \|T_{[\mjump,\mjump'],a_i,\omega}\| - L_{[\mjump, \mjump']}(a_i)\right|+\left|L_{\mjump}(a_i) - L_{[\mjump, \mjump']}(a_i)\right|\le \\
\le \eps'n+\eps'n+3\eps'n=5\eps'n.
\end{multline}
For any $m\in [\mjump,\mjump']$, due to the part~\ref{i:m3} of Theorem~\ref{t:main} we have
$$
\log\|T_{\mjump, a_i,\omega}\|\in U_{\eps'n}(L_{\mjump}(a_i)), \ \ \ \log\|T_{[\mjump,  m],a_i,\omega}\|\in U_{\eps'n}(L_{[\mjump, m]}(a_i)),
$$
and hence we get the following  estimate from below:
\begin{multline*}
  \log \|T_{m,a_i,\omega}\|\ge \log\|T_{\mjump, a_i, \omega}\|-\log\|T_{[\mjump, m], a_i, \omega}\|\ge \\
  \ge  L_{\mjump}(a_i)-L_{[\mjump, m]}(a_i)-2\eps'n=\psi_{\mjump}(m, a_i)-2\eps'n.
\end{multline*}
To get an estimate from above, we can represent
$$
T_{m,a_i,\omega} = T_{[m, \mjump'],a_i,\omega}^{-1} T_{\mjump',a_i,\omega}.
$$
The log-norm of the latter factor does not exceed $5\eps'n$ by~\eqref{eq:almost-zero}, while for the former factor we have
$$
\|T_{[m, \mjump'],a_i,\omega}^{-1}\|=\|T_{[m, \mjump'],a_i,\omega}\|.
$$
Due to Corollary \ref{c:upper-finite-m}, Proposition \ref{p.additivity}, and by using (\ref{e.symmetry}), we get
\begin{multline*}
  \log\|T_{[m, \mjump'],a_i,\omega}\|\le L_{[m, \mjump']}(a_i)+\eps'n\le L_{[\mjump, \mjump']}(a_i)-L_{[\mjump, m]}(a_i) + 2\eps'n\le \\
  \le L_{\mjump}(a_i)-L_{[\mjump, m]}(a_i) + 5\eps'n=\psi_{\mjump}(m, a_i)+ 5\eps'n,
\end{multline*}
and, therefore,
$$
\log\|T_{m,a_i,\omega}\|\le \log\|T_{\mjump',a_i,\omega}\|+ \log\|T_{[m, \mjump'],a_i,\omega}\|\le \psi_{\mjump}(a_i, m)+ 10n\eps'.
$$
Combining lower and upper bounds we obtain
$$
\log\|T_{m,a_i,\omega}\|\in U_{10n\eps'}(\psi_{\mjump}(a_i, m)).
$$
As we have $\eps'<\frac{\eps}{10}$, we obtained the desired estimate.

Now let us consider the case when $\mjump'<n$ and $m\in [\mjump',n]$. Again, due to (\ref{e.symmetry}), (\ref{eq:almost-zero}), Corollary \ref{c:upper-finite-m}, and Proposition \ref{p.additivity}, we have
\begin{multline*}
  \log\|T_{m,a_i,\omega}\|\le \log\|T_{\mjump',a_i,\omega}\|+\log\|T_{[\mjump', m],a_i,\omega}\|\le \log\|T_{[\mjump', m],a_i,\omega}\|+5\eps'n\le \\
\le L_{[\mjump', m]}(a_i)+6\eps'n\le L_{[\mjump, m]}(a_i)-L_{[\mjump, \mjump']}(a_i)+7\eps'n\le L_{[\mjump, m]}(a_i)-L_{\mjump}(a_i)+10\eps'n=\\
=\psi_{\mjump}(m, a_i)+10\eps'n.
\end{multline*}
Similarly, since
$$
T_{m,a_i,\omega}=T_{[\mjump', m],a_i,\omega}T_{\mjump',a_i,\omega},
$$
we get an estimate from below:
$$
  \log\|T_{m,a_i,\omega}\|\ge \log\|T_{[\mjump', m],a_i,\omega}\| - \log\|T_{\mjump',a_i,\omega}\|\ge \log\|T_{[\mjump', m],a_i,\omega}\|-5\eps'n
$$
Now, since
$$
T_{[\mjump', m],a_i,\omega}=T_{[\mjump, m],a_i,\omega}T_{[\mjump, \mjump'],a_i,\omega}^{-1},
$$
using (\ref{e.symmetry}) and the estimates of the part~\ref{i:m3} we have
\begin{multline*}
  \log\|T_{[\mjump', m],a_i,\omega}\|\ge \log\|T_{[\mjump, m],a_i,\omega}\| - \log\|T_{[\mjump, \mjump'],a_i,\omega}\|\ge \\
\ge L_{[\mjump, m]}(a_i)-\eps'n - L_{[\mjump, \mjump']}(a_i)-\eps'n \ge  L_{[\mjump, m]}(a_i) -L_{\mjump}-5\eps'n=\psi_{\mjump}(m, a_i)-5\eps'n,
\end{multline*}
and hence
$$
\log\|T_{m,a_i,\omega}\|\ge \psi_{\mjump}(m, a_i)-10\eps'n.
$$
Putting estimates from above and from below together, we obtain
$$
\log\|T_{m,a_i,\omega}\|\in U_{10n\eps'}(\psi_{\mjump}(a_i, m)).
$$

Finally, consider the case $\mjump'=n$. %Then in the same way as in~\eqref{eq:almost-zero}
The estimates of the part~\ref{i:m3} imply
$$
\log \|T_{\mjump,a_i,\omega}\| \in U_{n\eps'}(L_{\mjump}(a_i)), \quad  \log \|T_{[\mjump,n],a_i,\omega}\| \in U_{n\eps'}(L_{[\mjump, n]}(a_i)),
$$
and thus we have
\begin{equation}\label{eq:bm-p}
\log \|T_{n,a_i,\omega}\| \in U_{2n\eps'} (|L_{\mjump}(a_i)-L_{[\mjump, n]}(a_i)|)= U_{2n\eps'} (\psi_{\mjump}(a_i, n)).
\end{equation}

Now, for any $m\in [\mjump,n]$ we have two representations for $T_{m,a_i,\omega}$:
\begin{equation}\label{eq:steps}
T_{m,a_i,\omega}=T_{[\mjump,m],a_i,\omega}T_{\mjump,a_i,\omega} = T_{[m,n],a_i,\omega}^{-1} T_{n,a_i,\omega}.
\end{equation}
By Corollary \ref{c:upper-finite-m} we have
$$
\log \| T_{[\mjump,m],a_i,\omega}\| \le L_{[\mjump, m]}(a_i) + n\eps', \quad \log \| T_{[m,n],a_i,\bo}\| \le L_{[m,n]}(a_i) + n\eps',
$$
so using Proposition \ref{p.additivity},  from~\eqref{eq:bm-p}  and~\eqref{eq:steps} we get both a bound from above
\begin{multline*}
\log \|T_{m,a_i,\bo}\| \le \log \|T_{n,a_i,\bo}\|  + \log \|T_{[m,n],a_i,\bo}\| \le
\\
\le (L_{\mjump}(a_i)-L_{[\mjump, n]}(a_i)+2\eps'n)+(L_{[m,n]}(a_i)+\eps'n)\le \\
\le (L_{\mjump}(a_i)-L_{[\mjump, m]}(a_i)) +(L_{[m,n]}(a_i) -L_{[\mjump, m]}(a_i)+L_{[\mjump, m]})+3\eps'n\le \\
\le \psi_{\mjump}(a_i, m)+4\eps'n
\end{multline*}
and from below
\begin{multline*}
\log \|T_{m,a_i,\omega}\| \ge \log \|T_{\mjump,a_i,\omega}\|  - \log \|T_{[\mjump,m],a_i,\omega}\| \ge
\\
\ge (L_{\mjump}(a_i)-\eps'n)-(L_{[\mjump, m]}(a_i)+\eps'n)= \psi_{\mjump}(a_i, m)- 2  \eps'n.
\end{multline*}

Thus, in this case we also get the desired
$$
\log \|T_{m,a_i,\omega}\| \in U_{4n\eps'} (\psi_{\mjump}(m, a_i)),
$$
concluding the proof of the ``Cancellation'' part~\ref{i:m4} of Theorem~\ref{t:main}.

\subsection{Exponential contraction: quantitative statements}\label{ss.5.6}
This part is parallel to Section~4.6 from \cite{GK}. Notice that we cannot use the statements proven in \cite{GK} directly, since in the non-stationary setting some of the proofs must be essentially modified.

We start by establishing the following two contraction-type statements. The first one is a negative Lyapunov exponent type of a statement:

\begin{prop}\label{p.derivparam}
There exist $k_0\in \mathbb{N}$ such that for any $k\ge k_0$, any $x_0\in S^1$, any $\mgr_1, \ldots, \mgr_k\in \mathcal{K}$ and any $ a\in J$ we have
$$
\mathbb{E}_{\mgr_{k}^a, \ldots, \mgr_{1}^a}\log f_{k, a, \omega}'(x_0)\le -1.
$$
\end{prop}

The second is an actual contraction:

\begin{lemma}\label{l:r-contr}
For any $\eps_1,\eps_2>0$ there exists $K_1\in \N$ such that for any $a\in J$ and any $x,y\in \Sc$
we have
$$
\P \left( \dist( f_{K_1, a, \omega} (x), f_{K_1, a, \omega} (y)) <\eps_1 \right) > 1- \eps_2.
$$
\end{lemma}
%Notice that the statement of Lemma \ref{l:r-contr} repeats exactly the statement of \cite[Lemma 4.18]{GK}. Nevertheless, that statement was formulated in the stationary setting, and therefore some adjustments in the proof are needed.

%\vspace{1cm}
\begin{proof}[Proof of Proposition~\ref{p.derivparam}]
Recall that for $A\in \SL(2,\R)$ and a point $x_0\in \Sc$, corresponding to the direction of a unit vector $v_0\in \R^2$, one has
\[
(f_A)'(x_0) = \frac{1}{|Av_0|^2}.
\]
Hence,
\begin{equation}\label{eq:log-f-prime}
\log (f_A)'(x_0) = -2 \log |Av_0|.
\end{equation}
Now, recall that Theorem~\ref{t.2param} provides a lower bound $L_n\ge hn$ and a large deviations type bound for every $\eps>0$: there exists $\delta>0$ such that for all sufficiently large $n$ and any $a\in J$
\[
\mathbb{P}\left\{\left|\log |T_{n, a, \omega} v_0|-L_n(a)\right|>\eps n\right\}<e^{-\delta n}.
\]
Take $\eps=\frac{h}{2}$; this implies that for any $a\in J$ one has $L_n(a)-n\eps \ge \frac{nh}{2}$. Joining this with the lower deviations bound, we get that there exists $\delta>0$ such that  for every sufficiently large $n$ we have
\begin{equation}\label{eq:v0-h-growth}
\forall a\in J \quad \forall v_0, \, |v_0|=1 \quad \mathbb{P}\left\{ \log|T_{n, a, \omega} v_0| >\frac{h}{2} n\right\}> 1-e^{-\delta n}.
\end{equation}
Joining with~\eqref{eq:log-f-prime}, we get for all sufficiently large $n$ an upper bound for the expectation
\begin{multline*}
\E_{\mgr_1,\dots,\mgr_n} \log f_{n,a,\omega}'(x_0) =  \E_{\mgr_1,\dots,\mgr_n}  (-2 \log |T_{n,a,\omega}(v_0)|) \le
\\
\le -hn \cdot (1-e^{-\delta n}) + n \log \Cn \cdot e^{-\delta n} = - n (h -e^{-\delta n} (h+\log \Cn)).
\end{multline*}
where we have used a uniform upper bound $\|A\|\le \Cn$ for any $A\in \supp \mgr^a$, any $a\in J$ and any $\mgr\in \mK$. As the second factor in the right hand side tends to $h$ as $n\to\infty$, for all sufficiently large $n$ we get the desired
\[
\E_{\mgr_1,\dots,\mgr_n} \log f_{n,a,\omega}'(x_0)  \le - n (h -e^{-\delta n} (h+\log \Cn)) < -1.
\]
\end{proof}

\begin{proof}[Proof of Lemma~\ref{l:r-contr}]
Recall that Lemma~\ref{l:x-C-image} states
\[
\dist (f_A(x),x^+(A))\le \frac{\pi}{2} \cdot \frac{|v_x|/|Av_x|}{\|A\|}
\]
Applying~\eqref{eq:v0-h-growth} for $v_x$, we get that for all sufficiently large $n$ for every $a\in J$
\[
\P \left( \dist( f_{n, a, \omega} (x), x^+(T_{n,a,\omega})) < e^{-nh} \right) \ge 1- e^{-\delta n}.
\]
The same applies to the vector $v_y$, and thus
\[
\P \left( \dist( f_{n, a, \omega} (x), f_{n, a, \omega} (y)) < 2 e^{-nh} \right) \ge 1- 2e^{-\delta n}.
\]
Taking $n$ sufficiently large so that $2 e^{-nh}<\eps_1$, $2 e^{-\delta n}<\eps_2$ concludes the proof.
\end{proof}

\begin{defi}
For every $s\in (0,1]$ let the function $\varphi_s(x,y)$ be defined as
\begin{equation}\label{eq:phi-def}
\varphi(x,y):=(\dist_{\Sc}(x,y))^s.
\end{equation}
\end{defi}

The next statement provides another view on the contraction of orbits on the projective line; it states that for a sufficiently small $s$ the $s$-th power of the distance decreases in average under the random dynamics.
%In the stationary setting it was established in~\cite[Proposition~4.17]{GK}).
It is deduced from the two above contraction statements, joined with the estimate $d^s =1+ s \log d +O(s^2)$, in the same way as its stationary counterpart was established in~\cite[Proposition~4.17]{GK}).

\begin{prop}\label{p:phi}
There are constants $s\in (0,1]$ and $K_{\varphi}\in \N$ such that for any $a\in J$ one has
\begin{multline}\label{eq:phi-contraction}
\E \varphi(f_{K_{\varphi}, a,\omega}(x),f_{K_{\varphi}, a,\omega}(y)) =
\\
\int \varphi(f_{K_{\varphi}, a,\omega}(x),f_{K_{\varphi}, a,\omega}(y)) d\mgr_1^a\ldots d\mgr_{K_\varphi}^a
\le  \frac{1}{2} \varphi(x,y).
\end{multline}
\end{prop}

\begin{proof}[Proof of Proposition \ref{p:phi}]
The proof repeats the proof of \cite[Proposition~4.17]{GK} modulo the following adjustments:

\vspace{5pt}

1) The arguments leading to the formula (45) in \cite{GK} should be replaced by the statement of Proposition \ref{p.derivparam} above;

\vspace{5pt}

2) The proof of \cite[Lemma~4.19]{GK} should be replaced by the proof of Lemma \ref{l:r-contr} above.
\end{proof}

Finally, we use Proposition~\ref{p:phi} to estimate the behavior of random iterations
with different parameters:
\begin{coro}\label{c:l-a}
Fix constants $K_{\varphi},s$ given by Proposition~\ref{p:phi}. There exists a constant $C_\varphi$ such that for any $a,a'\in J$, $x,y\in \Sc$ one has
\begin{equation}\label{eq:C-a}
\E \varphi(f_{K_{\varphi},a,\omega}(x),f_{K_{\varphi},a',\omega}(y)) \le \frac{1}{2} \varphi(x,y) + C_{\varphi} |a-a'|^s.
\end{equation}
\end{coro}
\begin{proof}[Proof of Corollary \ref{c:l-a}]
The proof is the verbatim repetition of the proof of Corollary~4.25 from \cite{GK}.
\end{proof}
Iterating Corollary \ref{c:l-a}, we get
\begin{coro}\label{c:l-a1}
There are positive constants $C_{\varphi}'$ and $C_{\varphi}''$ (that depend on $K_{\varphi}, s$, and constants $\Lx$, $\La$ from Lemma~\ref{l.shift}) such that for any $l\in \N$, $k'<K_{\varphi}$, and any $a,a'\in J$, $x,y\in \Sc$ we have
\begin{equation}\label{e.in}
\E \varphi(f_{lK_{\varphi}+k',a,\,\omega}(x),f_{lK_{\varphi}+k',a',\,\omega}(y)) \le \frac{C_{\varphi}'}{2^l}\varphi(x,y)+ C_{\varphi}'' |a-a'|^s.
\end{equation}
\end{coro}
\begin{proof}[Proof of Corollary \ref{c:l-a1}]
The proof is the verbatim repetition of the proof of Corollary~4.26 from~\cite{GK}.
\end{proof}

\begin{proof}[Sketch of the proof of Proposition~\ref{p:classes}]
Once Corollary~\ref{c:l-a1} is obtained, Proposition~\ref{p:classes} follows by repeating verbatim the same arguments as in~\cite{GK} (see Remark~\ref{r.proofoftypes}).
Namely, consider the sequence of intervals $|X_{m,i}|$,  $m=1,\dots,n$. If all of them are of length at most $\eps'$, we are in the first (``small intervals'') case. Otherwise, there is a first iteration number~$m'$ for which $|X_{m',i}|>\eps'$.

Denote $\gamma:=\exp(-\sqrt[4]{n})$. Then, for each such interval, Corollary~\ref{c:l-a1} (together with the Markov inequality) implies that with the probability at least $1-\gamma^{s/3}$ the images
\[
x_{m,i-1}=f_{m,b_{i-1},\omega}(x_0) \quad \text{and} \quad x_{m,i}=f_{m,b_i,\omega}(x_0)
\]
approach each other at the time $m=m'+K\sqrt[3]{n}$ at the distance at most~$\gamma^{1/12}$, and stay close to each other until $m=n$. Now, their lifts $\tx_{m,i-1}$ and $\tx_{m,i}$ can either approach each other --- in which case the length of the corresponding interval $X_{m,i}$ becomes small, and this is the ``opinion-changer'' option. Or the difference between these lifts can be close to~$1$, and this is the ``jump interval'' case.

Finally, for every $m$ the intervals $X_{m,i}$ have disjoint interiors, hence there are at most $\const\cdot n^2$ of them that are larger than~$\eps'$. Thus, with the probability at least $1-\const \cdot n^2 \gamma^{s/3}$ the above description applies simultaneously to all non-small intervals, and this concludes the proof.
\end{proof}

%As soon as we obtained Corollary~\ref{c:l-a1}, Proposition~\ref{p:classes}  follows, see Remark~\ref{r.proofoftypes}.

\subsection{Distribution of jump intervals}

Here we provide a sketch of the proof of Property \ref{i:m5} in Theorem \ref{t:main}. The proof is almost a verbatim repetition of the proof of Parts {\bf I} and {\bf V} from \cite[Theorem 1.19]{GK}. Here we just explain what steps of the proof has to be adjusted in the non-stationary setting.

Recall that we denoted $\tx_{m, i}=\tf_{m, b_i, \bo}(\tx_0)$, the intervals $X_{m,i}$ were defined by (\ref{e.Xmi}), and for an interval $I\subset J$, $I=[a',a'']$, we defined
$$
R_{n,\omega}(I):=\tf_{n,a'', \omega}(\tx_0)-\tf_{n,a', \omega}(\tx_0).
$$
The Property \ref{i:m5} follows immediately from the following statement, which is analogous to \cite[Proposition 4.27]{GK}:
\begin{prop}\label{p:ones}
 For any $\eps'>0$ there exists $\zeta_5>0$ such that for any $m\le n$
\begin{equation}\label{eq:ones}
\Prob\left( \frac{(\tx_{m,N}-\tx_{m,0}) - \#\{j \,:\, |X_{m,j}|\ge 1 \}}{n}>\eps'\right) < \exp(-\zeta_5 \sqrt[4]{n}).
\end{equation}
\end{prop}
Proposition \ref{p:ones} applied to any interval $I\subset J$ of the form $I=[b_i,b_{i'}]$, $0\le i< i'\le N$ instead of $J$, implies that with probability at least $1-\exp(-\zeta_5 \sqrt[4]{n})$, the number
$$
M_{I;m} := \# \{ k \mid a_{i_k}\in I, m_k\le m \}
$$
is $\eps'n$-close to $\tf_{n, b_i, \bo}(\tx_0)-\tf_{n, b_{i'}, \bo}(\tx_0)=R_{n,\omega}(I)$. And applied to $m=n$, it gives that
with probability at least $1-\exp(-\zeta_5 \sqrt[4]{n})$,
 $M=\#\{j \,:\, |X_{n,j}|\ge 1 \}$ is $\eps'n$-close to $\tf_{n, b_+, \bo}(\tx_0)-\tf_{n, b_-, \bo}(\tx_0)=R_{n,\omega}(J)$. This gives the part {\bf V} of Theorem \ref{t:main}.

 The proof of Proposition~\ref{p:ones} is exactly the same as the proof of \cite[Proposition~4.27]{GK}, where the only difference is that, in order to accommodate the shift from stationary to non-stationary setting, one should use Proposition~\ref{p:phi} instead of \cite[Proposition~4.18]{GK}.

%\newpage

\section{Spectral Localization: proof of Theorems~\ref{t.vector} and~\ref{t.al}}\label{s:spectral}

\subsection{Deducing spectral localization from Theorem~\ref{t.vector}}

Let us first show that  Theorem~\ref{t.vector} implies spectral localization.
\begin{proof}[Proof of Theorem~\ref{t.al}]
We will need the following result, that is usually referred to as “Shnol Theorem”, due to a similar
result in the paper \cite{Sch} (see also \cite{Gl1, Gl2}):
\begin{theorem}[Shnol Theorem]\label{t.Schnol}
Let $H: \ell_2(\Z)\to \ell_2(\Z)$
be an operator of the form
\[
H u(n) = u(n-1) + u(n + 1) + V (n)u(n),
\]
with a bounded potential $\{V (n)\}_{n\in \Z}$. If every polynomially bounded solution to
$H u = Eu$ is in fact exponentially decreasing, then $H$ has pure point spectrum,
with exponentially decaying eigenfunctions. Similar statement holds for operators
on $\ell_2(\N)$ with Dirichlet boundary condition.
\end{theorem}
In the continuum case Theorem~\ref{t.Schnol} follows also from \cite[Theorem~1.1]{Sim}. The formal proof
in the discrete case can be found, for instance, in \cite[Theorem~7.1]{Kir}; we also refer the reader to some
improved versions of this result in~\cite[Lemma~2.6]{JZ} or~\cite{H}.

Due to Theorem~\ref{t.Schnol}, it suffices to show that (almost surely) every polynomially bounded
solution $u$ to the eigenvector problem $Hu=Eu$ is in fact exponentially decreasing. Now, for a random
Schr\"odinger operator~$H$, given by~\eqref{e.oper}, this relation can be written as
$$
u_{n+1} = (E-V(n)) u_n - u_{n-1},
$$
that transforms into a recurrent relation for vectors $v_n:=\left( {u_{n+1} \atop u_n }\right)$:
\begin{equation}\label{eq:Pi-v}
\left( {u_{n+1} \atop u_n }\right) = \Pi_{n,E} \left( {u_{n} \atop u_{n-1} }\right),
\end{equation}
where
$$
\Pi_{n, E}=\left(
               \begin{array}{cc}
                 E-V(n) & -1 \\
                 1 & 0 \\
               \end{array}
             \right).
$$

Note that the product of matrices
corresponding to the random Schr\"odinger operator~\eqref{e.oper}, satisfies the assumptions
of Theorem~\ref{t.vector} after grouping these matrices in pairs (that is, the condition
in Remark~\ref{r:groups} for $k=2$).
Indeed, these matrices are independent (as random variables $V(n)$ are), and satisfy
the $C^1$-boundedness assumption~\ref{B:C1}. Now, they can be represented as
$$
\Pi_{n, E}=
\left(
               \begin{array}{cc}
                 1 & E-V(n) \\
                 0 & 1 \\
               \end{array}
             \right)
\left(
               \begin{array}{cc}
                 0 & -1 \\
                 1 & 0 \\
               \end{array}
             \right);
$$
this implies non-strict monotonicity, as the first (parabolic) factor turns everything in the positive direction, except for the vector $\left( {1 \atop 0 }\right)$, that is the image of the $\left( {0 \atop 1 }\right)$ vector under the action of the second matrix. As these two vectors are different, a composition of any two such matrices $\Pi_{n+1,E} \Pi_{n,E}$ satisfies the strict monotonicity condition~\ref{B:Monotonicity}.

Finally, if for two measures $\msp_1,\msp_2$ and some homeomorphisms $f,g$ one has $f_*\msp_1=g_*\msp_1=\msp_2$, then $\msp_2$ is an invariant measure of the quotient~$(f\circ g^{-1})$. However, the quotient of any two different maps $\Pi_{n,E}$ is a parabolic map of the form
$$
\left(
               \begin{array}{cc}
                 1 & * \\
                 0 & 1 \\
               \end{array}
             \right),
$$
and the only invariant measure of its projectiviation is the Dirac one, concentrated at the direction of the vector $\left( {1 \atop 0 }\right)$. This measure is the image of the measure concentrated at the direction of the vector~$\left( {0 \atop 1 }\right)$. And as these two measures are different, there is no measure with a deterministic image under a composition of two matrices, and hence for such a composition the condition~\ref{B:Furstenberg} is also satisfied. Hence, the assumptions of Theorem~\ref{t.vector} are satisfied.

Returning to polynomially growing solutions of $Hu=Eu$, note that for any such solution,
$$
\limsup_{n\to\infty} \frac{1}{n}(\log |v_n| - L_n(E)) = \limsup_{n\to\infty} \, - \frac{1}{n} L_n(E) \le -h <0,
$$
and hence due to Theorem~\ref{t.vector}
$$
\log |v_n| = -L_n + o(n),
$$
thus $u_n$ is exponentially decreasing. This (due to Shnol's lemma) completes the proof of the spectral localization.
\end{proof}

\subsection{Hyperbolic-like products and behaviour of log-norms}

We will need the following two definitions; roughly speaking, the first one is the condition that means that in a product of given matrices there is ``not too much cancellation'':

\begin{defi}
Given matrices $A_1,\dots,A_n \in \SL(2,\R)$, and a sequence $L_j$, $j=1,\dots,n$ of real numbers,
we say that the product $A_n\dots A_1$ is \emph{$(L,r)$-hyperbolic}
if for any $0\le m< m'\le n$ for the product $T_{[m,m']}:=A_{m'}\dots A_{m+1}$ one has
\begin{equation}\label{eq:U-r-L}
\log \|T_{[m,m']}\| \in U_{r}(L_{m'}-L_m).
\end{equation}
Also, we say that a part $[n_1,n_2]$ of this product is \emph{$(L,r)$-hyperbolic}, if~\eqref{eq:U-r-L} holds for all $m,m'$ such that $[m,m']\subseteq [n_1,n_2]$.
\end{defi}
The second definition imposes restrictions on the sequences $L$ we will be using:
\begin{defi}
A sequence $L=(L_j)_{j=1,\dots,n}$ is \emph{$(h,\tilde{C})$-growing}, if for any $0\le m<m'\le n$ one has
$$
L_{m'}-L_m \ge h(m'-m)-\tilde{C},
$$
where we set $L_0:=0$.
\end{defi}

%The following definition is motivated by the conclusion of Lemma~\ref{l.prep}:
%

%%%%In particular, we have the following
%%%%
%%%%\begin{prop}\label{p.hyper}
%%%%For every $\eps>0$ there exists $n_0$ with the following property. Assume that for given $\eps$ and $n>n_0$ the events in the conclusions of Theorem~\ref{t:main} and Proposition~\ref{p:classes} \todo{And uniformity/control: no product $T_{[m,m']}$ has log-norm more than $(L_{m'}-L_m)+n\eps$} take place. Then:
%%%%\begin{itemize}
%%%%\item For every parameter $a$ belonging to any typical interval $J_i$, $i\neq i_1,\dots,i_M$, the corresponding product
%%%%\[
%%%%A_n(a) \dots A_1(a)
%%%%\]
%%%%is $(L,3n\eps)$-hyperbolic.
%%%%\item For every parameter $a$ belonging to any exceptional interval $J_{i_k}$, both parts
%%%%\[
%%%%A_{m_k}(a) \dots A_1(a) \quad \text{and} \quad A_{n}(a) \dots A_{m_k+1}(a)
%%%%\]
%%%%are $(L,3n\eps)$-hyperbolic.
%%%%\end{itemize}
%%%%\end{prop}
%%%%
%%%%%p.additivity
%%%%\begin{proof}
%%%%\todo[inline]{Follows almost immediately from the estimates we have already}
%%%%\end{proof}

Now, to establish Theorem~\ref{t.vector}, we will study possible behaviours of the sequence of log-norms $\log |T_{m,a,  \omega} \left( \begin{smallmatrix} 1 \\ 0 \end{smallmatrix} \right)|$. To do so, assume again that we are given a (finite) sequence of matrices~$A_1,\dots,A_n \in \SL(2,\R)$. Then, given a (nonzero) vector $v_0\in \R^2$, we can consider the sequence of its iterations
\begin{equation}\label{eq:v-def}
v_m=A_m v_{m-1}, \quad m=1,\dots, n.
\end{equation}

The following statements, describing possible behaviours of the sequence of log-norms $\log |v_m|, \, m=0,\dots, n$, are non-stationary analogues of Lemmata~5.2, 5.4, and~5.5 and of Remark~5.3 from~\cite{GK}. Their proofs are almost verbatim reproduction of the arguments from~\cite{GK}, but we present them here for completeness.

\begin{lemma}[Growth curve]\label{l:line-shape}
For any $\Cn, h, \eps>0$ there exist $\eps',n_1>0$ with the following property. Assume that $n>n_1$ and the following conditions hold:
\begin{itemize}
\item a part $[m_0, m_1]$ of the product $A_n\dots A_1$ is $(L,n\eps')$-hyperbolic,
\item all $A_i$ satisfy $\|A_i\|\le \Cn$,
%\item for any $0\le m\le m'\le n$ one has $L_{m'}-L_m \ge h (m'-m) - \eps' n$
\item the sequence $L$ is $(h, n\eps')$-growing
\item and $v_{m_0}$ is the least norm vector in the sequence~\eqref{eq:v-def} in the index interval $[m_0,m_1]$, i.e. $|v_{m_0}|\le |v_m|$ for all $m=m_0+1,\ldots, m_1$.
\end{itemize}
Then
$$
\forall m=m_0,m_0+1,\dots, m_1 \quad \log |v_m| - \log |v_{m_0}| \in U_{n\eps}(L_m-L_{m_0}).
$$
\end{lemma}

\begin{figure}[!h!]
\begin{center}
\includegraphics[width=0.3\textwidth]{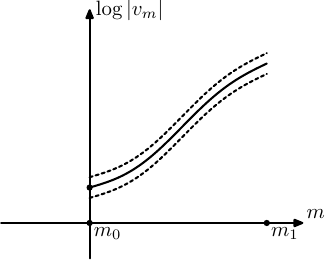} \hfill  \includegraphics[width=0.3\textwidth]{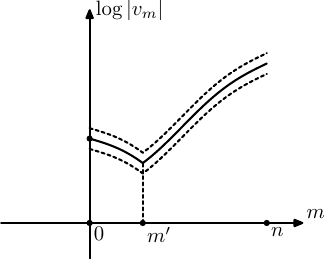}
\hfill \includegraphics[width=0.3\textwidth]{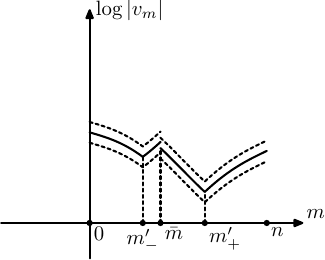}
\end{center}
\caption{Behaviour of log-norm of iterations of a given vector as in Lemmata~\ref{l:line-shape}, \ref{l:V-shape},~\ref{l:W-shape}. Bold line corresponds to the prediction curve (mid-point of the vertical neighborhood), dashed region shows its $\eps n$-neighborhood.}\label{f:norms}
\end{figure}

\begin{proof}
Without loss of generality, we can assume that $v_{m_0}$ is a unit vector. Take another unit vector, $\vu_{m_0}$, that realizes the norm of the full product until given $m\in [m_0,m_1]$,
$$
|T_{[m_0,m]} \vu_{m_0}| = \|T_{[m_0,m]}\|,
$$
and consider the sequence of the corresponding intermediate images,
$$
\vu_j=A_j \vu_{j-1}, \quad j=m_0+1,\dots, m.
$$
Then, we have a lower bound for their norms: as $\vu_m=T_{[j,m]}\vu_{j}$,
\begin{multline}\label{eq:u-lower}
\log |\vu_j|\ge \log \frac{|\vu_m|}{\|T_{[j,m]}\|} = \log \|T_{[m_0,m]}\| - \log \|T_{[j,m]}\|  \\ \ge ((L_m-L_{m_0})-n\eps')- ((L_m-L_j) +n\eps') =
\\
= (L_j-L_{m_0}) - 2n\eps' \ge h(j-m_0) - 3n\eps',
\end{multline}
where we have used the $(h,n\eps')$-growth assumption for the sequence~$L$.

Now, let $\phi_A$ be the function on the circle of directions that describes the change of the length:
$$
\phi_A([v])=\log \frac{|Av|}{|v|}
$$
for $v\in \R^2\setminus \{0\}$, where $[v]$ is the corresponding point of $\Sc=\R P^1$.

Then the log-length of an image of a vector is given by a sum:
\begin{equation}\label{eq:phi-sum}
\log |v_m|=\sum_{j=m_0+1}^m \log \frac{|A_j v_{j-1}|}{|v_{j-1}|} = \sum_{j=m_0+1}^m \phi_{A_j}([v_{j-1}]),
\end{equation}
\begin{equation}\label{eq:phi-sum-2}
\log |\vu_m|=\sum_{j=m_0+1}^m \log \frac{|A_j \vu_{j-1}|}{|\vu_{j-1}|} = \sum_{j=m_0+1}^m \phi_{A_j}([\vu_{j-1}]).
\end{equation}
Family of the functions $\phi_A$ for $A\in \SL(2,\R)$, $\|A\|\le \Cn$, is equicontinuous on~$\R P^1$. Hence, for any $\eps>0$
there exists $\delta>0$ such that
\begin{equation}\label{eq:eps-delta}
|\phi_A([u])-\phi_A([v])|< \frac{\eps}{2}
\end{equation}
for all $A\in \SL(2,\R)$ with $\|A\|\le M$ and all $u,v$ with the
angle between the corresponding lines less than~$\delta$. At the same time, subtracting~\eqref{eq:phi-sum-2} from~\eqref{eq:phi-sum} gives
\begin{equation}\label{eq:Am-v}
\log |v_m| = \log |\vu_m| + \sum_{j=m_0+1}^m \left(\varphi_{A_j}([v_{j-1}]) - \varphi_{A_j}([\vu_{j-1}]) \right).
\end{equation}
The first summand is within~$2n\eps'$ from $(L_m-L_{m_0})$ due to~\eqref{eq:u-lower} and the assumption on $(n\varepsilon', L)$-hyperbolicity.

Let us decompose the sum in the second summand depending on whether we can guarantee that the directions of~$\vu_{j-1}$ and~$v_{j-1}$ are less than~$\delta$ apart. To do so, note that the initial $v_{m_0}$ and $\vu_{m_0}$ are unit vectors, and thus they form a parallelogram of area at most~$1$. Hence
the same holds for their images $v_j$ and $\vu_j$ for any~$j=m_0,\dots,m_1$ (see Fig.~\ref{f:angles}).
As we have assumed $|v_j|\ge 1$ and as $|\vu_j|\ge \exp(L_j-L_{m_0}-2n\eps')$, the angle between the lines
passing through $v_j$ and $\vu_j$ does not exceed
\begin{equation}\label{eq:angles-bound}
\dist([v_j],[\vu_j])\le \frac{\pi}{2}\exp(-L_j+L_{m_0}+2n\eps').
\end{equation}

\begin{figure}[!h!]
\includegraphics{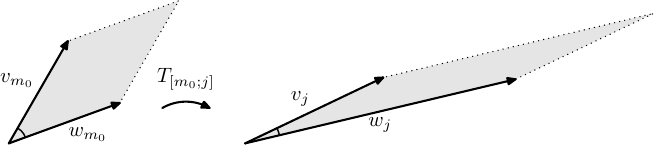}
\caption{Controlling angles between $v_j$ and $\vu_j$.}\label{f:angles}
\end{figure}

Hence, due to~\eqref{eq:angles-bound} we can guarantee that the angle between the lines containing $v_j$ and $\vu_j$ is at most~$\delta$ once
$$
h(j-m_0)-3n\eps' \ge \log(\frac{\pi}{2\delta}),
$$
or, equivalently, once
\[
j-m_0 \ge \frac{3\eps'}{h} n + \frac{\log (\pi/2\delta)}{h};
\]
%as due to the assumption we have a lower bound $L_m\ge mh-\tilde{C}$, the images are $\delta$-close once
%$$
%mh \ge 2\eps' n +\left( \tilde{C} +  \log(\frac{\pi}{2\delta})\right),
%$$
in particular, for all $n$ sufficiently large, it suffices to assume that
$$
j-m_0 \ge \frac {4\eps'}{h} n.
$$
Hence, in~\eqref{eq:Am-v} there are at most $\frac {4\eps'}{h}n$ summands with the angles exceeding $\delta$, each of which does not exceed $2\log \Cn$; hence, their contribution does not exceed is
at most
\[
\frac{4\eps'}{h} n \cdot 2\log \Cn
\]
At the same time, the contribution of the other ones is bounded by $\frac{\eps}{2}n$ due
%to the condition~\eqref{eq:eps-delta}
the choice of~$\delta$. Thus, we get an estimate
$$
\left|\log |v_m| - \log |\vu_m| \right| \le  n \frac{\eps}{2} + 2\log M \cdot \frac{4\eps'}{h} n= \left( \frac{8\log M}{h} \eps' + \frac{\eps}{2} \right) n ,
%\le \left(\left(2 +  \frac{6 \log M}{\lambda} \right) \eps' + \frac{\eps}{2} \right) n.
$$
and adding it with~\eqref{eq:u-lower}, we finally get
$$
\left |\log |v_m| - L_m \right| \le
\left(\left(2 +  \frac{8 \log M}{h} \right) \eps' + \frac{\eps}{2} \right) n
$$

Fixing $\eps' = \left(2 +  \frac{8 \log M}{h} \right)^{-1} \cdot \frac{\eps}{2}$, we get (for all sufficiently large~$n$) the desired upper bound
$$
|\log |v_m| - L_m | \le \eps n.
$$
This completes the proof of Lemma \ref{l:line-shape}.
\end{proof}

Removing the assumption that $v_{m_0}$ is the least norm vector in the sequence, we then immediately get the following
\begin{lemma}[Curved-V-shape]\label{l:V-shape}
For any $\Cn, \eps,h>0$ there exist $\eps',n_2>0$ with the following property. Assume that $n>n_2$ and
\begin{itemize}
\item a part $[m_0, m_1]$ of the product $A_n\dots A_1$ is $(L,n\eps')$-hyperbolic,
\item all $A_i$ satisfy $\|A_i\|\le \Cn$,
%\item for any $0\le m\le m'\le n$ one has $L_{m'}-L_m \ge h (m'-m) - \eps' n$
\item the sequence $L$ is $(h, n\eps')$-growing
\item and $(v_m)$ be a sequence of intermediate images associated to some $v_0\in\R^2\setminus \{0\}$ given by~\eqref{eq:v-def}
\end{itemize}
Then there exists $\mV\in \{m_0,\dots,m_1\}$, such that
$$
\forall m=m_0,m_0+1,\dots, m_1 \quad \log |v_m| -\log |v_{\mV}| \in U_{n\eps}(|L_m-L_{\mV}|).
$$
\end{lemma}
\begin{proof}
It suffices to take $\mV$ to be the index of the least norm $v_m$, $m=m_0,\dots,m_1$, and apply Lemma~\ref{l:line-shape} to intervals $[m_0,\mV]$ and $[\mV,m_1]$ separately.

To handle the case of one of these intervals being too small (of length less than~$n_0$), we choose~$n_1$ sufficiently large so that $n_1\eps > 2 n_0 \log \Cn$.
\end{proof}

Now, the conclusions~\ref{i:m2} and~\ref{i:m3} of Theorem~\ref{t:main} together
imply (for $n,\eps$ for which these conclusions hold) that for any $a\in J$ the product $T_{n,a,\omega}$ either is $(n\eps,(L_m))$-hyperbolic
itself, or can be divided into two hyperbolic products; also, Proposition~\ref{p.additivity} implies that $\{L_m\}$ is $(h, \eps n)$ growing. Thus, under the conclusions of Theorem~\ref{t:main} and Proposition~\ref{p.additivity} we have the following lemma.

\begin{lemma}[Curved-W-shape]\label{l:W-shape}
For any $\eps>0$ there exists $\eps'>0$, $n_1\in \mathbb{N}$ such that for all $n>n_1$ the following holds. Assume that the conclusions of Theorem~\ref{t:main} with the given $\eps'$
are satisfied for some finite product $T_{n,a,\omega}$. Then for any sequence $\bv_m$ of nonzero vectors such that
$\bv_{m}=T_{m,a,\omega}(\bv_{0})$, there exist numbers $\mV_-\le \mW \le \mV_+$ such that
$$
\forall m\in [0,\mW] \quad \log |v_m|-  \log |v_{\mV_-}| \in U_{n\eps}(|L_{m}-L_{\mV_-}|),
$$
$$
\forall m\in[\mW,n] \quad \log |v_m|-  \log |v_{\mV_+}| \in U_{n\eps}(|L_{m}-L_{\mV_+}|).
$$
\end{lemma}

In the same way as the previous ones, this lemma admits a geometric interpretation in terms of the corresponding graphs; see Fig.~\ref{f:norms}.

%\newpage

\subsection{First part of Theorem~\ref{t.vector}: Dirichlet conditions}

In the same way as in~\cite[Theorem~1.13]{GK}, Lemma~\ref{l:V-shape} allows us to prove the first (one-sided products) part of Theorem~\ref{t.vector}.

%+ Ref to \cite[Lemma~3.7]{GK22}.

\begin{proof}[Proof of the first part of Theorem~\ref{t.vector}]
Denote
\begin{equation}\label{eq:v-def-2}
v_0=\left(\begin{smallmatrix} 1 \\ 0 \end{smallmatrix}\right), \quad v_m=T_{m,a,  \omega} v_0, \quad m=1,\dots, n.
\end{equation}
%To establish Theorem~\ref{t.vector}, we will study possible behaviours of the sequence of log-norms $\log |v_m|$.
%= A_m(a,\omega) v_{m-1}

Assume that~\eqref{eq:lim-less} holds; then for some $\eps_0>0$ one has for all sufficiently large $n$
\begin{equation}\label{eq:n-eps0}
\log |v_n| = \log |T_{n,a,  \omega} \left( \begin{smallmatrix} 1 \\ 0 \end{smallmatrix} \right)| < L_n(a)- n \eps_0.
\end{equation}
Due to the standard argument of a countable intersection (considering a sequence of positive values of  $\eps_0$ that tends to zero) it suffices to show that the conclusion of the theorem holds with~\eqref{eq:lim-less}
replaced with~\eqref{eq:n-eps0}. From now on, fix small $\eps_0>0$.

Take the point $x_0$ on the circle to be the projectivization
image of the vector~$v_0$. As in~\cite{GK}, note that due to the convergence of the series
$\sum_n \exp(-\delta_0 \sqrt[4]{n})$, Borel--Cantelli lemma implies that for any $\eps,\eps'>0$ almost surely
for all sufficiently large~$n$ the conclusions of Theorem~\ref{t:main} and of Proposition~\ref{p:derivatives-control} (for this specific choice of the point~$x_0$) hold.

Take and fix sufficiently small $\eps'$ (we will impose an assumption on its smallness later), and let $n_2=n_2(\eps')$ be such that the conclusions of Theorem~\ref{t:main} (for $\eps'$ instead of $\eps$) and of Proposition~\ref{p:derivatives-control}, as well as~\eqref{eq:n-eps0}, hold for all $n>n_2$.
Note first that for $n>n_2$ the parameter interval $J_i$, containing $a$, cannot be neither small nor opinion-changing, and hence it is a jump interval.
%
%Indeed, otherwise Proposition~\ref{p:derivatives-control} would impliy the norms control~\eqref{eq:derivative-u},
%and thus the norms of the images~$v_m$ would satisfy the lower bound
%$$
%\log |T_{m,a} v_0| \ge L_m(a) - C_1 \eps' n.
%%
%$$
%
%On the other hand,~\eqref{eq:n-eps0} implies that
%$$
%\log |T_{m,a} v_0| \le L_m(a) - \eps_0 m.
%$$
%These two inequalities provide a contradiction once $\eps_0 m >C_1 \eps' n$.
%

Indeed, otherwise due to Proposition~\ref{p:derivatives-control} we get
\[
\log \tf'_{n,a,\omega}(\tx_0) \in U_{C_1\eps' n} (-2 L_n(a)),
\]
thus implying a lower bound for the norm
\[
\log |v_n|= \log | T_{n,a,\omega}(v_0)| \ge L_n(a) - \frac{C_1\eps'}{2} n.
\]
Once~$\eps'$ is sufficiently small to ensure
\[
\frac{C_1 \eps'}{2} < \eps_0,
\]
this lower bound contradicts~\eqref{eq:n-eps0}. This proves that the interval $J_i\ni a$ is actually a jump interval for all $n>n_1$.

Moreover, if $\mjump_{(n)}$ is the corresponding jump index (provided by the conclusion of Proposition~\ref{p:derivatives-control}), using the estimate~\eqref{eq:derivative-upl} for the derivative after $m=\mjump_{(n)}$ iterations again provides a lower bound
\begin{equation}\label{eq:v-break}
\log |v_{\mjump_{(n)}}|= \log | T_{\mjump_{(n)},a,\omega}(v_0)| \ge L_{\mjump_{(n)}}(a) - \frac{C_1\eps'}{2} n.
\end{equation}
Together with an upper bound~\eqref{eq:n-eps0} at the same iteration $\mjump_{(n)}$, still assuming that $\mjump_{(n)}>n_2$, we get
\[
\log |v_{\mjump_{(n)}}| < L_{\mjump_{(n)}}(a)- \mjump_{(n)} \eps_0,
\]
this (see Fig.~\ref{f:bm-n-upper}, left) provides an inequality
\[
 \frac{C_1\eps'}{2} n \ge \mjump_{(n)} \eps_0.
\]
Hence, for all sufficiently large $n$ (namely, for $n>\frac{2\eps_0  n_2}{C_1\eps'}$) we have an upper bound
\begin{equation}\label{eq:mn-upper}
\mjump_{(n)} \le  \frac{C_1\eps'}{2\eps_0} n .
\end{equation}

\begin{figure}[!h!]
\includegraphics{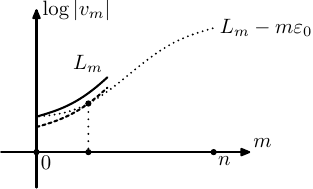} \quad \includegraphics{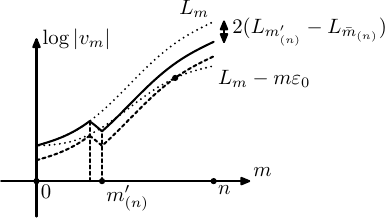}
\caption{Left: Upper estimate for $\mjump_{(n)}$; if it did not hold, lower (dashed) and upper (dotted) estimates for $\log |v_m|$ would contradict each other at $m=\frac{C_1\eps'}{2\eps_0} n$. \newline Right: Lower estimate for $\mV_{(n)}$; otherwise, lower (dashed) and upper (dotted) estimates for $\log |v_n|$ would contradict each other.}\label{f:bm-n-upper}
\end{figure}

Now, for given $\eps>0$ let $\eps'>0$ be chosen sufficiently small, and $n$ be sufficiently large for Lemma~\ref{l:V-shape} to be applicable.
%so that the conclusions of this lemma hold.
Then, the conclusions of Theorem~\ref{t:main} and Proposition~\ref{p:derivatives-control} %VK:m1
imply that the part $[\mjump_{(n)}, n]$ of the product $A_n(a)\dots A_1(a)$ is $(L,n\eps')$-hyperbolic, and hence the conclusions of Lemma~\ref{l:V-shape} hold on this interval of indices. Let $\mV=\mV_{(n)}\in [m_{(n)},n]$ be the corresponding index. Then, we have lower bounds for the log-norms, where the former one is~\eqref{eq:v-break}, and two latter ones are implied by Lemma~\ref{l:V-shape}:
%VK:m2
\begin{equation}\label{eq:incr1}
\log |v_{\mjump_{(n)}}| \ge L_{\mjump_{(n)}}(a) - \frac{C_1\eps'}{2} n,
\end{equation}
\begin{equation}\label{eq:incr2}
\log |v_{\mV_{(n)}}| - \log |v_{\mjump_{(n)}}| \ge  -(L_{\mV_{(n)}}(a)- L_{\mjump_{(n)}}(a)) - 2 n\eps,
\end{equation}
\begin{equation}\label{eq:incr3}
\log |v_n| - \log |v_{\mV_{(n)}}| \ge  (L_n(a)- L_{\mV_{(n)}}(a)) - 2 n\eps.
\end{equation}
Hence,
\begin{equation}\label{eq:incr-total}
\log |v_n| \ge L_n(a) - 2(L_{\mV_{(n)}}(a)- L_{\mjump_{(n)}}(a)) - 4 n\eps - \frac{C_1\eps'}{2} n.
\end{equation}
On the other hand, recall that $\log |v_n| \le L_n(a) - n\eps_0$. Hence,
\begin{equation}\label{eq:L-diff}
2(L_{\mV_{(n)}}(a)- L_{\mjump_{(n)}}(a)) \ge n(\eps_0 - 4 \eps - \frac{C_1 \eps'}{2}),
\end{equation}
and as we can choose $\eps$ and then $\eps'$ arbitrarily small, we can ensure that the right hand side of~\eqref{eq:L-diff} is at least $\frac{\eps_0}{2}n$, finally implying a lower bound (see Fig.~\ref{f:bm-n-upper}, right)
\[
(L_{\mV_{(n)}}(a)- L_{\mjump_{(n)}}(a)) \ge \frac{\eps_0}{4} n
\]
and thus
\begin{equation}\label{eq:bm-p-lower}
\mV_{(n)} - \mjump_{(n)} \ge \frac{\eps_0}{4 \log \Cn} n.
\end{equation}
Let $\lambda:=\frac{\eps_0}{8 \log \Cn}$. Then, for a sufficiently small $\eps'$ and all sufficiently large $n$, from~\eqref{eq:mn-upper} and~\eqref{eq:bm-p-lower} we get
\[
\mjump_{(n)} < \lambda n, \quad \mV_{(n)} > 2 \lambda n.
\]

\begin{figure}
\includegraphics{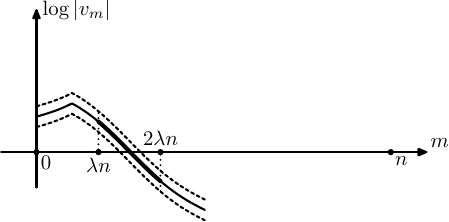}
\caption{Guaranteed decrease of $\log |v_m|$ between $m=\lambda n$ and $m=2\lambda n$.}\label{f:m-2m}
\end{figure}

Now, the conclusions of Lemma~\ref{l:V-shape} imply that (for all sufficiently large $n$ and for all $m_1, m_2$ on the interval $[\lambda n, 2\lambda n]$ one has
\begin{equation}\label{eq:m12}
%\forall m_1, m_2\in [\lambda n, 2\lambda n] \quad
(\log |v_{m_1}| - \log |v_{m_2}|) \in  U_{2n\eps} \left(-(L_{m_1}(a)- L_{m_2}(a))\right);
\end{equation}
see Fig.~\ref{f:m-2m}.
Denote $r_m:=\log |v_m| + L_m(a)$, then we can rewrite~\eqref{eq:m12} as
\[
\forall m_1, m_2\in [\lambda n, 2\lambda n] \quad |r_{m_1} -r_{m_2}| \le  2n\eps.
\]
Hence, for every sufficiently large $m$ we have
\begin{equation}\label{eq:m-half}
|r_{m} - r_{\lceil \frac{m}{2} \rceil } |\le \frac{2}{\lambda} \eps \cdot m,
\end{equation}
where we are taking $n=\lceil m/{2\lambda} \rceil$ to ensure that both $m_1:=m, \, m_2:= \lceil m/2 \rceil $ belong to $[\lambda n, 2\lambda n]$.

Finally, summing~\eqref{eq:m-half} over the decreasing geometric series~$m, \frac{m}{2}, \frac{m}{4},\dots$, we get
\begin{equation}\label{eq:m-estimate}
\forall m \quad  |r_m| \le \frac{4}{\lambda} \eps \cdot m + \const
\end{equation}
for some uniform constant that does not depend on $m$. (One can also see the proof of~\eqref{eq:m-estimate} as an induction argument, and the constant coming from the base of the induction, that is, small values of~$m$.)

Thus,
\[
\limsup_{m\to\infty} \frac{1}{m} \left|\log |v_{m}| +L_{m}(a) \right| \le \frac{4}{\lambda} \eps.
\]
As $\eps>0$ can be chosen arbitrarily small, and $\lambda$ depends only on $\eps_0$, but not on $\eps$, we obtain the desired
\[
\limsup_{m\to\infty} \frac{1}{m} \left( \log |v_{m}| +L_{m}(a) \right) = 0.
\]
This completes the proof of the first part of Theorem~\ref{t.vector}.
\end{proof}

%\newpage

\subsection{Second part of Theorem~\ref{t.vector}}

The proof of Theorem~\ref{t.vector} repeats almost word-for-word the proof of~\cite[Theorem~1.11]{GK}, though there are some modifications adapting it to the non-stationary case.

\begin{proof}[Proof of the second part of Theorem~\ref{t.vector}]
Let $v_n:=T_{n,a,\omega}(v)$ for all $n$. Without loss of generality, we can assume that $|v_0|=1$.
As in the proof of the first part, it suffices to show that
\begin{equation}\label{eq:two-eps0}
\limsup_{n\to\pm \infty} \frac{1}{|n|} (\log |v_n | - L_n(a)) < -\eps_0
\end{equation}
in fact forces
$$
\limsup_{n\to\pm \infty} \frac{1}{|n|} (\log |v_n | +L_n(a)) = 0.
$$
As before,~\eqref{eq:two-eps0} implies that for all sufficiently large $n$ we have
\begin{equation}\label{eq:two-finite-n}
\log |v_n | < L_n(a)-\eps_0 n, \quad \log |v_{-n} | < L_{-n}(a) -\eps_0 n.
\end{equation}

We can assume that for any $\eps,\eps'>0$, $\eps'\ll \eps\ll \eps_0$, for all $n$ sufficiently large
the conclusions of Theorem~\ref{t:main} hold for the product
\begin{equation}\label{eq:T2n}
T_{[-n;n],a,\bo}= A_{n}(a) \dots A_{-n}(a),
\end{equation}
and hence Lemma~\ref{l:W-shape} can be applied.

In the same way as before, for any such $n$ we let $\mV_{-,(n)}<\mW_{(n)}<\mV_{+,(n)}$ be the indices given for the product~\eqref{eq:T2n} by
Lemma~\ref{l:W-shape} (that correspond to the breakpoints of the ''curved W'' graph,  the central one
being the upwards break point). %(see Fig.~\ref{}).

Note first that for all sufficiently large~$n$ one has
\begin{equation}\label{eq:0-between}
\mV_{-,(n)}<0<\mV_{+,(n)}
\end{equation}
Indeed, if $\mV_{+,(n)}\le 0$, then one would have
$$
\log |v_n|-\log |v_0| \in U_{2n\eps} (L_n(a)),
$$
and hence (recall that $|v_0|=1$)
$$
\log |v_n| \ge L_n(a) - 2 n \eps,
$$
contradicting the assumed~\eqref{eq:two-finite-n} as $2\eps< \eps_0$.
In the same way we get $\mV_{-,(n)}< 0$.

Now, in the same way as in the first part, we are going to show that the ``jump'' index $\mW_{(n)}$ is sufficiently close to~$0$.
Indeed, the conclusions of Lemma~\ref{l:W-shape} together with ~\eqref{eq:0-between}  imply that %for all sufficiently large~$n$
$$
\log |v_{\mW_{(n)}}| \ge L_{\mW_{(n)}}(a) - 2n\eps.
$$
%see Fig.~\ref{???} \todo{Arrow from zero}.
Since due to~\eqref{eq:two-finite-n} we know that
$$
\log |v_{\mW_{(n)}}| \le L_{\mW_{(n)}}(a)-\eps_0\cdot |\mW_{(n)}|,
$$
we get
$$
\eps_0\cdot |\mW_{(n)}| \le 2n\eps,
$$
and thus
$$
|\mW_{(n)}| \le \frac{2\eps}{\eps_0} n.
$$

Now, in the same way as in the first part, denoting we are going to prove the auxiliary
\begin{lemma}\label{l:2-half}
$\mV_{+,(n)}, |\mV_{-,(n)}| \ge 2\lambda n$ for all sufficiently large $n$, where $\lambda=\frac{\eps_0}{8 \log \Cn}$.
\end{lemma}
\begin{proof}
We will establish the estimate for $\mV_{+,(n)}$, as the other one is completely analogous. To do so, assume first that $\mW_{(n)}>0$.
Then, we have three inequalities
\begin{equation}\label{eq:incr1-W}
\log |v_{\mW_{(n)}}| \ge L_{\mW_{(n)}}(a) - 2 \eps n,
\end{equation}
\begin{equation}\label{eq:incr2-W}
\log |v_{\mV_{+,(n)}}| - \log |v_{\mW_{(n)}}| \ge  -(L_{\mV_{+,(n)}}(a)- L_{\mW_{(n)}}(a)) - 2 n\eps,
\end{equation}
\begin{equation}\label{eq:incr3-W}
\log |v_n| - \log |v_{\mV_{(+,n)}}| \ge  (L_n(a)- L_{\mV_{+,(n)}}(a)) - 2 n\eps.
\end{equation}
that are analogues of~\eqref{eq:incr1}, \eqref{eq:incr2} and \eqref{eq:incr3} respectively. Adding, we get an analogue of~\eqref{eq:incr-total}
\begin{equation}\label{eq:incr-total-W}
\log |v_n| \ge L_n(a) - 2(L_{\mV_{+,(n)}}(a)- L_{\mW_{(n)}}(a)) - 6 n\eps,
\end{equation}
and hence
\[
2(L_{\mV_{+,(n)}}(a)- L_{\mW_{(n)}}(a)) \ge n(\eps_0 - 6 \eps).
\]
Thus, once $\eps< \frac{1}{12} \eps_0$, we have
\[
\log \Cn \cdot  (\mV_{+,(n)} - \mW_{(n)}) \ge  \frac{n\eps_0}{4}
\]
and hence the desired
\[
\mV_{+,(n)} - \mW_{(n)} \ge \frac{n\eps_0}{4 \log \Cn} = 2 \lambda n.
\]

Now, in the case $\mW_{(n)}\le 0$, instead of~\eqref{eq:incr1-W} and \eqref{eq:incr2-W} we get directly
\[
\log |v_{\mV_{+,(n)}}|  \ge  -L_{\mV_{+,(n)}}(a) - 2 n\eps;
\]
together with~\eqref{eq:incr3-W}, we then get
\[
\log |v_n| \ge L_n(a) - 2(L_{\mV_{+,(n)}}(a)) - 4 n\eps,
\]
and conclude in the same way as before.
\end{proof}

Now, we have that for all sufficiently large $n$
\[
\mW_{(n)}<\lambda n < 2\lambda n < \mV_{+,(n)},
\]
and hence~\eqref{eq:m12} holds for any two $m_1,m_2$ on the interval $[\lambda n, 2\lambda n]$. From this moment
the exact repetition of the arguments of the first part allows to conclude: we denote $r_m:=\log |v_m| + L_m(a)$, obtain
the estimate~\eqref{eq:m-half} for all sufficiently large $m$. By summing over $m, \frac{m}{2}, \frac{m}{4},\dots$, we get
\eqref{eq:m-estimate} and hence
\[
\limsup_{m\to\infty} \frac{1}{m} \left|\log |v_{m}| +L_{m}(a) \right| \le \frac{4}{\lambda} \eps,
\]
thus implying (as $\eps$ can be taken arbitrarily small) the desired
\[
\limsup_{m\to\infty} \frac{1}{m} \left( \log |v_{m}| +L_{m}(a) \right) = 0.
\]
This completes the proof of the first part of Theorem~\ref{t.vector}.

The asymptotics at $-\infty$ can be handled in the same way. This completes the proof of Theorem~\ref{t.vector}.
\end{proof}

%\todo[inline]{End of the text from~\cite{GK}}

%\newpage

\section{Dynamical Localization: proof of Theorem~\ref{t.dl}}\label{s.dynloc}

\subsection{Operator on $\ell_2(\N)$ with the Dirichlet boundary conditions}

\begin{proof}[Proof of the first part of Theorem~\ref{t.dl}]
We will start by considering the case of the operator on $\ell_2(\N)$ with the Dirac boundary condition at the origin.
We already know from Theorem~\ref{t.al} that the spectral localization holds: the operator $H$ admits an orthonormal base of eigenfunctions $u_\ki\in \ell_2(\N)$,
$$
H u_\ki = E_\ki u_\ki.
$$
As before, this equation can be transformed into the recurrent relation~\eqref{eq:Pi-v} on the vectors $v_{n,\ki} =\left( \begin{smallmatrix} u_\ki(n+1) \\ u_\ki(n) \end{smallmatrix} \right)$. Note that due to the Dirichlet boundary condition, the vector $v_{0,\ki}$ is proportional to the vector $v_0=
\left( \begin{smallmatrix} 1 \\ 0 \end{smallmatrix} \right)$.

Take $\alpha=\frac{h}{2}$,
%\frac{h}{4},
where $h$ is chosen for the random product of $\Pi_{n,E}$-matrices (with~$E$ belonging to the interval $J:=[-K,K]$, containing the spectrum) as in  Theorem \ref{t.2param}. Let us show that the conclusion of the theorem holds with this value of~$\alpha$. That is, for any given $\xi>0$ we show the existence of a constant $C_{\xi}$ such that the desired estimate~\eqref{eq:SULE} holds, where $\mmax_\ki$ is always chosen to be the index, at which the norm of the vector $v_{m,\ki}$ is maximal:
\begin{equation}\label{eq:choice-mj}
\forall m \quad |v_{\mmax_\ki,\ki}| \ge |v_{m,\ki}|.
\end{equation}

Actually, we are going to establish a slightly stronger estimate: as the function $u_j$ is orthonormal, one has $|v_{\mmax_\ki,\ki}|\le 1$; we will actually show that
\begin{equation}\label{eq:SULE-v}
 |v_{m,j}|\le C_\xi e^{\xi|\mmax_\ki |-\alpha|m-\mmax_\ki|} \cdot |v_{\mmax_\ki, \ki}|
\end{equation}
for all $m$ and $\ki$.

In order to do so, we will repeat the arguments of the proof of Theorem~\ref{t.vector}. Namely, we first fix sufficiently small $\eps,\eps'>0$; in fact, as we will see, one can take
\begin{equation}\label{eq:eps-choices}
\eps= \lambda' \cdot \min\left(\frac{\xi}{20},\frac{h}{20}\right), \quad \eps'=\frac{\eps}{C_1}, \quad \text{where} \quad  \lambda' :=  \frac{h}{20 \log \Cn},
\end{equation}
and $\Cn$ is given by \ref{B:C1}, and $C_1$ is given by Proposition \ref{p:derivatives-control}.
Then, almost surely there exists an $n_2$
such that for all $n>n_2$ the conclusions of Theorem~\ref{t:main} and of Proposition~\ref{p:derivatives-control} hold
for the chosen values of~$\eps, \eps'$. We will show that knowing $n_2$ suffices to give an explicit value for the constant~$C_{\xi}$. To do so, let us first establish the following lemma, analogous to the first steps the proof of Theorem~\ref{t.vector}:
\begin{lemma}\label{l:V-bounds}
For any $n\ge \max( \frac{2}{\lambda'} \mmax_\ki,n_2)$ the parameter interval $J_{k_n,n}$, containing the energy~$E_{\ki}$, is a jump interval in terms of Proposition~\ref{p:derivatives-control}. Moreover, denote by~$\mjump_{(n)}$ the corresponding jump moment, and let $\mV_{(n)}\in [\mjump_{(n)}, n]$ be
the moment  obtained by the application of ``curved-V'' Lemma~\ref{l:V-shape}. Then, for any such $n$, the following estimates hold:
\begin{itemize}
\item the jump moment satisfies $\mW_{(n)}'<\lambda' n$,
\item the lowest point of curved-V satisfies $\mV_{(n)}> 2\lambda' n$.
\end{itemize}
\end{lemma}
\begin{proof}
Indeed, if the corresponding interval was not a jump one, or if the upper estimate for the jump index $\mW_{(n)}'$ did not hold, we would have
(due to the log-growth estimates~\eqref{eq:derivative-u} and~\eqref{eq:derivative-upl} respectively) a lower bound for the norm of $v_{\lceil \lambda' n\rceil,\ki}$:
\begin{equation}\label{eq:norms-hv}
\log |v_{\lceil \lambda' n\rceil,\ki}| > \log |v_{\mmax_\ki,\ki}| - 2 C_1 n \eps' + L_{\lceil \lambda' n\rceil} (E_j)-L_{\mmax_\ki} (E_j).
\end{equation}
As $\lambda'n \ge 2\mmax_{\ki}$, we then would have
$$
L_{\lceil \lambda' n\rceil} (E_j)-L_{\mmax_\ki} (E_j)\ge h (\lceil \lambda' n\rceil-\mmax_\ki) \ge h \cdot \frac{ \lambda' n}{2};
$$
as due to the choices~\eqref{eq:eps-choices} one has
\[
2 C_1 n \eps' = 2 \eps n  \le \frac{h \lambda'}{10} \cdot n,
\]
the right hand side of~\eqref{eq:norms-hv} would thus be greater than~$\log |v_{\mmax_\ki,\ki}|$, and this would be in contradiction with the choice~\eqref{eq:choice-mj} of the index~$\mmax_{\ki}$.

The second part, the a lower bound for $\mV_{(n)}$, is obtained by an argument close to the one ensuring~\eqref{eq:bm-p-lower}. Namely, we get a lower estimate for $|v_{n,\ki}|$, joining a lower estimate
\[
\log |v_{\mV_{(n)},\ki}| - \log |v_{\mmax_{\ki},\ki}| \ge - |\mmax_{\ki} - \mV_{(n)}|\cdot  \log \Cn
\]
with
\[
\log |v_{n,\ki}| - \log |v_{\mV_{(n)},\ki}| \ge L_{n} - L_{\mV_{(n)}} - 2 \eps n \ge |n- \mV_{(n)}| \cdot h - 3 \eps n,
\]
we get
\begin{multline}\label{eq:v-n-lower}
\log |v_{n,\ki}| - \log |v_{\mmax_{\ki},\ki}| \ge  |n- \mV_{(n)}| \cdot h - |\mmax_{\ki} - \mV_{(n)}|\cdot  \log \Cn - 3 \eps n \ge
\\
\ge nh - 2 \mV_{(n)}\cdot  \log \Cn - h \frac{\lambda'}{2} n.
\end{multline}
Now, if we had $\mV_{(n)}\le 2\lambda' n$, that would imply $2 \mV_{(n)}\cdot  \log \Cn\le \frac{h}{10} n$ and hence
\[
\log |v_{n,\ki}| - \log |v_{\mmax_{\ki},\ki}| \ge hn - \frac{h}{10} n - h \frac{\lambda'}{2} n>0
\]
(where we have used that $\lambda'<\frac{1}{20}$). And this would be a contradiction with the choice of the index~$\mmax_{\ki}$.
\end{proof}

Lemma~\ref{l:V-bounds} already suffices to provide explicit exponential decrease bounds for all $m>\max(2\mmax_\ki,\lambda'n_2)=: m_{init}$. Namely, we have the inequalities
\[
\mW_{(n)} < \lambda n, \quad \mV_{(n)} > 2 \lambda n
\]
for every $n>\max( \frac{2}{\lambda'} \mmax_\ki,n_2)$. In the same way as in the proof of the first part of  Theorem~\ref{t.vector},
this implies that for every $m>m_{init}$, taking $n_{init}:=\lceil \frac{m}{\lambda'} \rceil$, from
\[
\mW_{(n)}< \left\lceil \frac{m}{2} \right\rceil < m < 2 \lambda n
\]
we get (compare with~\eqref{eq:m12})
\begin{equation}\label{eq:v-m-2}
(\log |v_{m,\ki}| - \log |v_{\lceil \frac{m}{2}\rceil,\ki}|) \in  U_{2n\eps} \left(-(L_{m}(E_j)- L_{\lceil \frac{m}{2}\rceil}(E_j))\right).
\end{equation}
Denoting $r'_m:=\log |v_{m,\ki}| + \alpha \cdot |m-\mmax_{\ki}|$ and recalling that we chose $\alpha=\frac{h}{2}$, we get from~\eqref{eq:v-m-2}
\begin{multline}\label{eq:rp-incr}
r'_m - r'_{\lceil \frac{m}{2}\rceil} = \log |v_{m,\ki}| - \log |v_{\lceil \frac{m}{2}\rceil, \ki  }| + \frac{m}{2} \alpha \le
\\
\le \left(-(L_{m}(E_j)- L_{\lceil \frac{m}{2}\rceil}(E_j))\right) + 2n\eps + \frac{m}{2} \alpha
\\
\le -\frac{m}{2} h + 4 n\eps + \frac{m}{2}\alpha
\end{multline}
Now, due to the choice of $\eps$ we have $4n\eps \le \frac{\lambda' h n}{5} \le \frac{mh}{5}$, and finally the right hand side of~\eqref{eq:rp-incr} is less or equal to
\[
-\frac{m}{2} h + \frac{mh}{5} + \frac{m}{2}\alpha = m h (-\frac{1}{2} + \frac{1}{5} + \frac{1}{4}) < 0.
\]
Hence, for every $m\ge m_{init}$ we have
\[
v_{m} + \alpha \cdot |m-\mmax_{\ki}| = r'_m \le r'_{\lceil \frac{m}{2} \rceil} = v_{\frac{m}{2}} + \alpha \cdot |\frac{m}{2}-\mmax_{\ki}|,
\]
and hence it suffices to establish~\eqref{eq:SULE-v} for $m\in [0,m_{init}]$; recall that $m_{init}=\max(2\mmax_{\ki},\lambda' n_2)$.

Now, note that to handle the case $2 \mmax_{\ki}\le \lambda' n_2$, it suffices to take $C_{\xi}>e^{\alpha n_2}$, as then for any $m= 0,\dots,m_{init}$ one has
$$
\log |v_{m}| < 1 < C_{\xi} e^{-\alpha |m-\mmax_{\ki}|}.
$$

Finally, let us consider the case $\lambda' n_2<2\mmax_{\ki}$. Take $n_{init}:=\left\lceil \frac{2\mmax_\ki}{\lambda'}\right\rceil$ and consider the corresponding jump index $\mW_{(n_{init})}$ and break point $\mV_{(n_{init})}$. Due to Lemma~\ref{l:V-bounds}, we have
\[
0\le \mW_{(n_{init})}, \mmax_\ki \le 2\lambda'n_{init} < \mV_{(n_{init})}
\]
%\missingfigure{Down-going wedge, two possibilities}
\begin{figure}[!h!]
\includegraphics[width=0.45\textwidth]{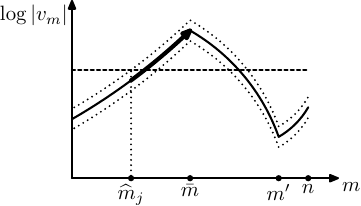} \quad \includegraphics[width=0.45\textwidth]{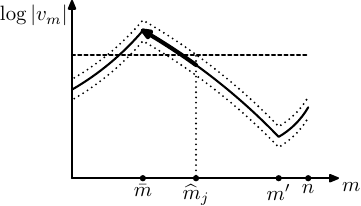}
\caption{Estimating $L_{\mjump} (E_j) - L_{\mmax_{\ki}} (E_j)$. The graph represents the behavior of $\log |v_m|$ predicted by curved-V Lemma, together with $\pm\eps n$ error. Dashed line shows maximal possible value of $\log |v_m|$ due to the choice of $\mmax_{\ki}$, and the arrow shows the ``forbidden'' growth between the compared values if the estimated difference was too large.}\label{f:contr}
\end{figure}

The conclusions of Lemma~\ref{l:V-shape} imply that, regardless of whether $\mmax_{\ki}<\mW_{(n_{init})}$ or $\mmax_{\ki}\ge \mW_{(n_{init})}$,
\[
\log |v_{\mjump_{(n_{init})},\ki}| - \log |v_{\mmax_\ki,\ki}| \ge |L_{\mjump_{(n_{init})}} (E_j) - L_{\mmax_{\ki}} (E_j)| - 3 \eps n_{init},
\]
%\todo[inline]{Comment on the order? Sub-lemma?}
and as the increment in the left hand side should be negative, we actually get (see Fig.~\ref{f:contr})
\begin{equation}\label{eq:L-delta}
|L_{\mjump_{(n_{init})}} (E_j) - L_{\mmax_{\ki}} (E_j)| \le  3 \eps n_{init}.
\end{equation}

Also, again from Lemma~\ref{l:V-shape} for every $m\le m_{init}$ we have
\[
\log |v_{m,\ki}|  - \log |v_{\mjump_{(n_{init})},\ki}| \le -| L_{m} (E_j)- L_{\mjump_{(n_{init})}} (E_j)| + 3 \eps n_{init},
\]
and joining it with~\eqref{eq:L-delta} we obtain
\begin{multline*}
\log |v_{m,\ki}|  - \log |v_{\mjump_{(n_{init})},\ki}| \le -|L_{m} (E_j)- L_{\mmax_{\ki}} (E_j) | + 6 \eps n_{init}
\\
\le  - h \cdot  | m-\mmax_{\ki}| + 8 \eps n_{init}.
\end{multline*}
Now, $\eps\le \frac{\lambda' \xi}{20}$, and hence
\[
8\eps n_{init}\le \frac{2}{5} \lambda' \xi n_{init} = \frac{2}{5} \lambda'  \lceil \frac{2\mmax_{\ki}}{\lambda'} \rceil \cdot  \xi < \xi \mmax_{\ki}.
\]
We finally obtain
\[
|v_{m,\ki}| \le e^{-h\cdot |m-\mmax_{\ki}| + \xi \mmax_{\ki}} |v_{\mmax_{\ki},\ki}|,
\]
so the estimate~\eqref{eq:SULE-v} holds for these values of $m$ for any $C_{\xi}\ge 1$.

Joining the two cases, we see that~\eqref{eq:SULE-v} (and hence the desired uniform estimate~\eqref{eq:SULE}) always holds for $C_{\xi}:= e^{\alpha n_2}$, where $n_2$ corresponds to the chosen~$\eps$ and~$\eps'$. This concludes the proof of the theorem for the case of~$\ell_2(\N)$.
\end{proof}

\begin{remark}\label{rk:V}
Actually, a slightly more accurate estimates (copying those of the proof of Theorem~\ref{t.vector}) allow to show that the eigenfunctions' localization rate is given by the corresponding function $L_{n}(E)$ up to an arbitrarily small correction: for any $\xi>0$ there exists $C_{\xi}$ such that for any eigenfunction $u$, satisfying $Hu=Eu$, there exists $\mmax$ such that
$$
\forall \ki \quad \forall m \quad |u_\ki(m)| \le C_{\xi} e^{\xi \mmax} \cdot e^{-|L_m(E_{\ki})-L_{\mmax}(E_{\ki})| + \xi |m-\mmax|}.
$$
\end{remark}

%\newpage

\subsection{Operator on $\ell_2(\Z)$}

\begin{proof}[Proof of the second part of Theorem~\ref{t.dl}]

Let us now pass to the case of the operator on~$\ell_2(\Z)$. Again, due to the spectral localization there exists an orthonormal base of eigenvectors
$$
H u_\ki = E_\ki u_\ki,
$$
and this equation becomes a recurrent relation $v_{n+1,j}=\Pi_{n,E_\ki} v_{n,j}$ on the vectors  $v_{n,\ki} =\left( \begin{smallmatrix} u_\ki(n+1) \\ u_\ki(n) \end{smallmatrix} \right)$. Again, we will take $\alpha=\frac{h}{4}$, where $h$ is chosen for the setting of the product of random matrices $\Pi_{n,E}$, where $E\in [-K,K]\supset \sigma(H)$.

Now, for any given $\xi>0$, as in the first part, we take $\eps, \eps',\lambda$ given by~\eqref{eq:eps-choices}
%\[
%\eps= \lambda' \cdot \min(\frac{\xi}{20},\frac{h}{20}), \quad \eps'=\frac{\eps}{C_1}, \quad \text{where} \quad  \lambda' :=  \frac{h}{20 \log \Cn}.
%\]
and consider $n_2$ such that for all $n>n_2$ for the products
$$
\Pi_{n,E}\dots \Pi_{-n,E}
$$
the conclusions of Theorem~\ref{t:main} hold, and Lemma~\ref{l:W-shape} hence can be applied.
%and of Proposition~\ref{p:derivatives-control} hold.
As before, we will construct $C_{\xi}$, depending only on $n_2$ (but not on the eigenvalue~$E_\ki$),
for which the estimate~\eqref{eq:SULE} holds; actually, we will again establish a stronger estimate~\eqref{eq:SULE-v}.

Also as before, let the eigenfunction $u_\ki$ (and the corresponding eigenvalue $E_\ki$) be fixed,
and let $\mmax_{\ki}$ be the index of the maximal norm for the corresponding vectors~$v_{m,\ki}$:
$$
|v_{\mmax_\ki,\ki}|=\max_{m\in \Z} |v_{m,\ki}|
$$

For any $n\ge n_2$ denote by $\mV_{-,(n)}<\mW_{(n)}<\mV_{+,(n)}$ be the indices given for this~$n$
by Lemma~\ref{l:W-shape} (these indices correspond to the breakpoints of the ``curved~W'' graph,
the central one being the upwards break point).

The following lemma is an analogue of Lemma~\ref{l:V-bounds}:
\begin{lemma}\label{l:W-bounds}
For any $n\ge \max( \frac{2}{\lambda'} |\mmax_\ki|,n_2)$
%the parameter interval $J_{k_n,n}$, containing the energy~$E_{\ki}$, is a jump interval in terms of Proposition~\ref{p:derivatives-control}. Moreover,
the following estimates hold:
%denote by~$\mW_{(n)}'$ the corresponding jump moment, and let $\mV_{(n)}\in [\mW_{(n)}', n]$ be
%the moment  obtained by the application of ``curved-V'' Lemma~\ref{l:V-shape}. Then, for any such $n$, the following estimates hold:
\begin{itemize}
\item the central index satisfies
%(corresponding to the jump) % jump moment satisfies
$|\mW_{(n)}|<\lambda' n$,
\item the left and right indices satisfiy
\[
\mV_{-,(n)} < -2\lambda' n, \quad \mV_{+,(n)}> 2\lambda' n.
\]
%\item the jump moment satisfies $\mW_{(n)}'<\lambda' n$,
%\item the lowest point of curved-V satisfies $\mV_{(n)}> 2\lambda' n$.
\end{itemize}
\end{lemma}
\begin{proof}
%\todo[inline] {Parallel to the previous one; difference $3$ times for the m-bar increment}
Let us start with the second conclusion, following the same lines as in Lemma~\ref{l:V-bounds}.
%\todo[inline]{Actually, shall we rewrite its proof? And shall we outline the corresponding parts of it as lemmas?}

Namely, we have
\begin{multline}\label{eq:W1}
\log |v_{\lfloor 2\lambda'n\rfloor ,\ki}| - \log |v_{\mmax_{\ki},\ki}| \ge - |\mmax_{\ki} - \lfloor 2\lambda'n\rfloor |\cdot  \log \Cn \ge
\\
\ge  -3\lambda'n \cdot  \log \Cn = - \frac{3}{20} hn .
\end{multline}
On the other hand, if we had $\mV_{+,(n)}\le 2\lambda' n$, this would imply
\begin{multline}\label{eq:W2}
\log |v_{n,\ki}| - \log |v_{\lfloor 2\lambda'n\rfloor ,\ki}| \ge L_{n}(E_j) - L_{\lfloor 2\lambda'n\rfloor}(E_j) - 2 \eps n \ge
\\
\ge (n- \lfloor 2\lambda'n\rfloor ) \cdot h - 3 \eps n \ge hn - 2 \lambda'h n - 3\frac{\lambda'h}{20} n.
\end{multline}
%\missingfigure{W is to the left of $m=\lfloor 2\lambda'n\rfloor$}

Adding~\eqref{eq:W1} and~\eqref{eq:W2}, and recalling that $\lambda'<\frac{1}{20}$, we would get
\[
\log |v_{n,\ki}| - \log |v_{\mmax_{\ki},\ki}| \ge nh ( 1- 2\lambda'-\frac{3}{20} \lambda' - \frac{3}{20}) > 0,
\]
thus obtaining a contradiction with the choice of the index~$\mmax_{\ki}$.

We have obtained the desired $\mV_{+,(n)}> 2\lambda' n$. The same arguments show that $\mV_{-,(n)}<- 2\lambda' n$; this time, a contradiction comes from the consideration of~$\log |v_{-n,\ki}|$.
%
%with
%\[
%\log |v_{n,\ki}| - \log |v_{\mV_{(n)},\ki}| \ge L_{n} - L_{\mV_{(n)}} - 2 \eps n \ge |n- \mV_{(n)}| \cdot h - 3 \eps n,
%\]
%we get
%\begin{multline}\label{eq:W-n-lower}
%\log |v_{n,\ki}| - \log |v_{\mmax_{\ki},\ki}| \ge  |n- \mV_{(n)}| \cdot h - |\mmax_{\ki} - \mV_{(n)}|\cdot  \log \Cn - 3 \eps n \ge
%\\
%\ge nh - 2 \mV_{(n)}\cdot  \log \Cn - h \frac{\lambda'}{2} n.
%\end{multline}
%Now, if we had $\mV_{(n)}\le 2\lambda' n$, that would imply $2 \mV_{(n)}\cdot  \log \Cn\le \frac{h}{10} n$ and hence
%\[
%\log |v_{n,\ki}| - \log |v_{\mmax_{\ki},\ki}| \ge hn - \frac{h}{10} n - h \frac{\lambda'}{2} n>0
%\]
%(where we have used that $\lambda'<\frac{1}{20}$). And this would be a contradiction with the choice of the index~$\mmax_{\ki}$.
%

%\begin{equation}\label{eq:norms-W}
%\log |v_{\lceil \lambda n\rceil,\ki}| > \log |v_{\mmax_\ki,\ki}| - 2 C_1 n \eps' + L_{\lceil \lambda' n\rceil} (E_j)-L_{\mmax_\ki} (E_j).
%\end{equation}

For the first conclusion, we have
\begin{multline}\label{eq:v-W-center}
\log |v_{\mW_{(n)},\ki}| > \log |v_{\mmax_\ki,\ki}| - 2 C_1 n \eps' + |L_{\mW_{(n)}} (E_j)-L_{\mmax_\ki} (E_j)| -n\eps \ge
\\
\ge \log |v_{\mmax_\ki,\ki}| - 2 n \eps + h \cdot |\mW_{(n)} - \mmax_{\ki}| - 2n\eps
\end{multline}
If we had $|\mW_{(n)}|\ge \lambda' n$, the inequality $\lambda'n >2|\mmax_{\ki}|$ would imply that
\begin{multline*}
\log |v_{\mW_{(n)},\ki}| - \log |v_{\mmax_\ki,\ki}| \ge  h \cdot \frac{\lambda' n}{2} - 4 n \eps \ge
\\
\ge h \cdot \frac{\lambda' n}{2} - 4 n \cdot \frac{\lambda'h}{20} = \lambda' \cdot hn \cdot (\frac{1}{2}-\frac{1}{5}) >0,
\end{multline*}
again providing a contradiction with the choice of the index~$\mmax_{\ki}$.
\end{proof}

Again, Lemma~\ref{l:W-bounds} suffices to provide explicit exponential decrease bounds for all $|m|>\max(2\mmax_\ki,\lambda'n_2)=: m_{init}$.
Consider the case $m>0$, for the case $m<0$ is completely analogous. We have the inequalities
\[
\mW_{(n)} < \lambda n, \quad \mV_{(n)} > 2 \lambda n
\]
for every $n>\max( \frac{2}{\lambda'} \mmax_\ki,n_2)$. In the same way as before,
taking $n:=\lceil \frac{m}{\lambda'} \rceil$, from
\[
\mW_{(n)}< \left\lceil \frac{m}{2} \right\rceil < m < 2 \lambda n
\]
we get
\begin{equation}\label{eq:v-m-2-w}
(\log |v_{m,\ki}| - \log |v_{\lceil \frac{m}{2}\rceil,\ki}|) \in  U_{2n\eps} \left(-(L_{m}(E_j)- L_{\lceil \frac{m}{2}\rceil}(E_j))\right).
\end{equation}
Denoting $r'_m:=\log |v_{m,\ki}| + \alpha \cdot |m-\mmax_{\ki}|$ and recalling that we chose $\alpha=\frac{h}{2}$, we get from~\eqref{eq:v-m-2-w}
\begin{multline}\label{eq:rp-incr-w}
r'_m - r'_{\lceil \frac{m}{2}\rceil} = \log |v_{m,\ki}| - \log |v_{\lceil \frac{m}{2}\rceil, \ki  }| + \frac{m}{2} \alpha \le
\\
\le \left(-(L_{m}(E_j)- L_{\lceil \frac{m}{2}\rceil}(E_j))\right) + 2n\eps + \frac{m}{2} \alpha
\\
\le -\frac{m}{2} h + 4 n\eps + \frac{m}{2}\alpha
\end{multline}
Now, due to the choice of $\eps$ we have $4n\eps \le \frac{\lambda' h n}{5} \le \frac{mh}{5}$, and finally the right hand side of~\eqref{eq:rp-incr-w} is less or equal to
\[
-\frac{m}{2} h + \frac{mh}{5} + \frac{m}{2}\alpha = m h (-\frac{1}{2} + \frac{1}{5} + \frac{1}{4}) < 0.
\]
Hence, for every $m\ge m_{init}$ we have
\[
v_{m} + \alpha \cdot |m-\mmax_{\ki}| = r'_m \le r'_{\lceil \frac{m}{2} \rceil} = v_{\frac{m}{2}} + \alpha \cdot |\frac{m}{2}-\mmax_{\ki}|,
\]
and hence it suffices to establish~\eqref{eq:SULE-v} for $m\in [-m_{init},m_{init}]$.%; recall that $m_{init}=\max(2\mmax_{\ki},\lambda' n_2)$.

Again as before, the case $2 |\mmax_{\ki}|\le \lambda' n_2$, is handled by requesting $C_{\xi}\ge e^{2\alpha n_2}$, as then for any $|m| \le m_{init}$ one has
$$
\log |v_{m}| < 1 < C_{\xi} e^{-\alpha |m-\mmax_{\ki}|}.
$$
Finally, let us consider the case $\lambda' n_2<2|\mmax_{\ki}|$. Take $n_{init}:=\left\lceil \frac{2|\mmax_\ki|}{\lambda'}\right\rceil$ and consider the corresponding jump index $\mW_{(n_{init})}$ and break points $\mV_{-,(n_{init})}, \mV_{+,(n_{init})}$. Due to Lemma~\ref{l:W-bounds}, we have
\[
\mV_{-(n_{init})} \le -2\lambda'n_{init} \le \mW_{(n_{init})}, \mmax_\ki \le 2\lambda'n_{init} < \mV_{+,(n_{init})}
\]
%\missingfigure{Down-going wedge, two possibilities? The same as for $V$-ones? A reference?}

Then, the conclusions of Lemma~\ref{l:W-shape} imply
\[
\log |v_{\mW_{(n_{init})},\ki}| - \log |v_{\mmax_\ki,\ki}| \ge |L_{\mW_{(n_{init})}} (E_j) - L_{\mmax_{\ki}} (E_j)| - 3 \eps n_{init},
\]
%\todo[inline]{Comment on the order? Sub-lemma?}
and as the increment in the left hand side should be negative, we actually get
\begin{equation}\label{eq:L-delta-2}
|L_{\mW_{(n_{init})}} (E_j) - L_{\mmax_{\ki}} (E_j)| \le  3 \eps n_{init}.
\end{equation}

Also, again from Lemma~\ref{l:W-shape} for every $m\le m_{init}$ we have
\[
\log |v_{m,\ki}|  - \log |v_{\mV_{(n_{init})},\ki}| \le -| L_{m} (E_j)- L_{\mV_{(n_{init})}} (E_j)| + 3 \eps n_{init},
\]
and joining it with~\eqref{eq:L-delta-2} we obtain
\begin{multline*}
\log |v_{m,\ki}|  - \log |v_{\mV_{(n_{init})},\ki}| \le -|L_{m} (E_j)- L_{\mmax_{\ki}} (E_j) | + 6 \eps n_{init}
\\
\le  - h \cdot  | m-\mmax_{\ki}| + 8 \eps n_{init}.
\end{multline*}
Now, $\eps\le \frac{\lambda' \xi}{20}$, and hence
\[
8\eps n_{init}\le \frac{2}{5} \lambda' \xi n_{init} = \frac{2}{5} \lambda'  \left\lceil \frac{2\mmax_{\ki}}{\lambda'} \right\rceil \cdot  \xi < \xi \mmax_{\ki}.
\]
We finally obtain
\[
|v_{m,\ki}| \ge e^{-h\cdot |m-\mmax_{\ki}| + \xi \mmax_{\ki}} |v_{\mmax_{\ki},\ki}|,
\]
so the estimate~\eqref{eq:SULE-v} holds for these values of $m$ for any $C_{\xi}\ge 1$.

Joining the cases studied, we again see that the uniform estimate~\eqref{eq:SULE} holds for the choice $C_{\xi}:=e^{2 \alpha n_2}$, where $n_2$ corresponds to the chosen $\eps$ and~$\eps'$. This completes the proof for the case of~$\ell_2(\Z)$.
\end{proof}

Again in the same way as before, we have the following remark.

\begin{remark}\label{rk:W}
Slightly more accurate estimates allow to show that the eigenfunctions' localization rate is given by the corresponding function $L_{n}(E)$ up to an arbitrarily small correction: for any $\xi>0$ there exists $C_{\xi}$ such that for any eigenfunction $u$, satisfying $Hu=Eu$, there exists $\mmax$ such that
$$
\forall \ki \quad \forall m \quad |u_\ki(m)| \le C_{\xi} e^{\xi \mmax} \cdot e^{-|L_m(E_{\ki})-L_{\mmax}(E_{\ki})| + \xi |m-\mmax|}.
$$
\end{remark}

We conclude this section by a reference to Theorem~7.5 from \cite{DJLS}: this theorem states that
SULE implies SUDL, and hence the dynamical localization for this operator is also established.

\begin{appendix}

\section{Locally Cantor essential spectrum}\label{a.1}
Here we give an example of a non-stationary Anderson-Bernoulli potential such that the almost sure essential spectrum of the corresponding discrete Schr\"odinger operator $H:l^2(\mathbb{Z})\to l^2(\mathbb{Z})$ intersects an open interval at a Cantor set of zero measure.
 Construction is very explicit. Namely, choose  any  sequence $\{n_k\}_{k\in \mathbb{N}}$ of integers such that
$$
\quad  n_k\to \infty\ \ \text{and} \quad n_{k+1}-n_k\to \infty \quad \text{as} \quad k\to \infty.
$$
We define the random potential in the following way:
\begin{equation}\label{e.example2}
V(n)=\left\{
  \begin{array}{ll}
    \text{$0$ or~$1$ with probability $1/2$}, & \hbox{if $n\not\in\{n_k\}$;} \\
    \text{$0$ or~$100$ with probability $1/2$}, & \hbox{if $n\in \{n_k\}$.}
  \end{array}
\right.
\end{equation}
%$V(n)$ takes values~$0$ and~$1$ with probability $1/2$, if $n$ does not belong to the sequence $\{n_k\}$, and takes values~$0$ and~$100$ with probability $1/2$ otherwise. %In other words, in  the set of indices~$\Z$ we have segments of fast-increasing lengths $n_k-n_{k-1}$, separated by single indices with energy way completely at a different scale.

\begin{prop}\label{p.ex2}
Almost sure essential spectrum of the operator $H$ with the potential (\ref{e.example2}) is a union of the interval $[-2, 3]$ and a Cantor set contained in the interval $[98, 102]$.
\end{prop}

To characterize the spectrum of an operator it will be convenient to use the following criterion (notice that we do not make any assumptions regarding the nature of the potential in Proposition \ref{p.crit}):

\begin{prop}\label{p.crit}
Let $\{V(n)\}_{n\in \mathbb{Z}}$ be a bounded potential of the discrete Schr\"odinger operator $H$ acting on $\ell^2(\Z)$ via
\begin{equation}
[H u](n) = u(n+1) + u(n-1) + V(n) u(n).
\end{equation}
Then we have the following:

\vspace{5pt}

1) Energy $E\in \mathbb{R}$ belongs to the spectrum of the operator $H$ if and only if there exists $K>0$ such that for any $N\in \mathbb{N}$ there is $m\in \mathbb{Z}$ and a unit vector $\bar u$, $|\bar u|=1$,  such that $|T_{[m, m+i], E}\,\bar u|\le K$ for all $|i|\le N$, where $T_{[m, m+i], E}$ is the product of transfer matrices given by
$$
T_{[m, m+i], E}=\left\{
               \begin{array}{lll}
                 \Pi_{m+i-1, E}\ldots \Pi_{m, E}, & \hbox{if $i> 0$;} \\
                 \text{\rm Id},  & \hbox{if $i= 0$;} \\
                 \Pi_{m+i, E}^{-1}\ldots \Pi_{m-1, E}^{-1}, & \hbox{if $i<0$,}
               \end{array}
             \right.
$$
and $\Pi_{n, E}=\left(
               \begin{array}{cc}
                 E-V(n) & -1 \\
                 1 & 0 \\
               \end{array}
             \right)$.

\vspace{5pt}

2) Energy $E\in \mathbb{R}$ belongs to the essential spectrum of the operator $H$ if and only if there exists $K>0$ such that for any $N\in \mathbb{N}$ there is  a sequence $\{m_j\}_{j\in \mathbb{N}}, m_j\in \mathbb{Z},$ with $|m_j-m_{j'}|>2N$ if $j\ne j'$, and unit vectors $\bar u_j$, $|\bar u_j|=1$,  such that $|T_{[m_j, m_j+i], E}\,\bar u_j|\le K$ for all $|i|\le N$ and all $j\in \mathbb{N}$.
\end{prop}
\begin{proof}
Proof of Proposition \ref{p.crit} can be extracted from the density of generalized eigenvalues (energies for which there are polynomially bounded solutions of the Schr\"odinger equation), e.g. see  \cite[Theorem~2.11]{D16}, and the classical Weyl's criterion. We leave the details to the reader.
\end{proof}

For each $\omega\in \{0,1\}^{\mathbb{Z}}$ consider an operator $H_{\omega}:l^2(\mathbb{Z})\to l^2(\mathbb{Z})$ given by the potential
$$
V_{\omega}(n)=\left\{
                \begin{array}{ll}
                  100, & \hbox{if $n=0$;} \\
                  \omega_n, & \hbox{if $n\ne 0$.}
                \end{array}
              \right.
$$
There are uncountably many operators of this form. Each of them has exactly one eigenvalue in the interval $[98, 102]$. Let us denote this eigenvalue by $E_{\omega}$.

\begin{lemma}\label{l.trick}
Intersection of the almost sure essential spectrum of the operator $H$ given by the potential $(\ref{e.example2})$ with the interval $[98, 102]$ is exactly $\cup_{\omega\in \{0,1\}^{\mathbb{Z}}}\,E_{\omega}$.
\end{lemma}
\begin{proof}[Proof of Lemma \ref{l.trick}]
If $E_\omega\in [98, 102]$ is an eigenvalue of $H_\omega$, consider the corresponding eigenvector $\{u_{n, \omega}\}_{n\in \mathbb{Z}}\in l^2(\mathbb{Z})$, and the vector $\bar u=\left(
                                           \begin{array}{c}
                                             u_{1, \omega} \\
                                             u_{2, \omega} \\
                                           \end{array}
                                         \right)$. Notice that for any finite sequence $\{V_\omega(-N), \ldots, V_\omega(N)\}$ with probability 1 the potential $\{V(n)\}$ given by (\ref{e.example2}) contains infinitely many intervals $\{V(m_j-N), \ldots, V(m_j+N)\}$ that coincide with that sequence. Due to Proposition \ref{p.crit} this implies that $E_\omega$ belongs to almost sure spectrum of $H$ with potential (\ref{e.example2}).

Suppose now that $E_0\in [98, 102]$ belongs to almost sure essential spectrum of $H$ with potential (\ref{e.example2}).  Then there exists $K>0$, a sequence of vectors $\{\bar u_j\}$, and a sequence of finite intervals $\{V(m_j-j), \ldots, V(m_j+j)\}$ such that $|T_{[m_j, m_j+i]}\bar u_j|\le K$ for all $|i|\le j$. Using Cantor diagonal process, we can find a subsequence $\{j_t\}_{t\in \mathbb{N}}$ such that $\bar u_{j_t}\to \bar u^*$ as $t\to \infty$, and $V(m_{j_t}+k)=V_k^*\in \{0,1\}$ for all $k\ne 0$ and all large enough $t\in \mathbb{N}$. Consider $\omega=\{\ldots \omega_{-k}\ldots \omega_0\omega_1\ldots \omega_k\ldots\}$ with $\omega_k=V_k^*$. Then $\bar u^*$ must decay exponentially under the application of transfer matrices generated by the potential $\{V_\omega(n)\}$, and hence generate an eigenvector of the operator $H_\omega$ with the eigenvalue $E_\omega$.
\end{proof}

%\vspace{5pt}

%There is an example of a nonstationary bounded Anderson-Bernoulli-model type potential, for which the following holds. There exists an interval $J_0$ of energies such that $\sigma\cap J_0$ is almost surely a Cantor set.
%
%Example. We consider the operator on~$l_2(\N)$ with the Dirichlet boundary condition at the origin.
%
%
%Then the behavior inside these intervals is hyperbolic, and defines attracting and repelling Cantor sets (depending on a parameter) of sufficiently small Hausdorff dimension (less than~$1/2$). The eigenvalues then will correspond to the energies sending the attracting Cantor set to the repelling one; this is a smooth image of the product of two such sets, that is hence of Hausdorff dimension less than~$1$ (and in particular, of zero Lebesgue measure).

Notice that if $A>2$, then the matrix of the form $\left(
                                                    \begin{array}{cc}
                                                      A & 1 \\
                                                      -1 & 0 \\
                                                    \end{array}
                                                  \right)$ has two eigenvalues, namely $\frac{A+\sqrt{A^2-4}}{2}>1$ and $\frac{A-\sqrt{A^2-4}}{2}=\left(\frac{A+\sqrt{A^2-4}}{2}\right)^{-1}<1$. The corresponding eigenvectors are
                                                  $\left(\begin{array}{c}
                                                  \frac{A+\sqrt{A^2-4}}{2} \\
                                                  1 \\
                                                  \end{array} \right)$ and $\left(\begin{array}{c}
                                                  1 \\
                                                  \frac{A+\sqrt{A^2-4}}{2} \\
                                                  \end{array} \right)$. Let us denote the proectivizations of those vectors by $x_1(A)$ and $x_2(A)$.

                                                  For an operator $H_\omega$ each transfer matrix $\Pi_{n, E}$, $n\ne 0$, must be either
                                                  $\left(
                                                    \begin{array}{cc}
                                                      E & 1 \\
                                                      -1 & 0 \\
                                                    \end{array}\right)$,
                                                     or $\left(
                                                    \begin{array}{cc}
                                                      E-1 & 1 \\
                                                      -1 & 0 \\
                                                    \end{array}\right)$, and we are interested in the regime where $E\in [98, 102]$. Let us denote by $I_1(E)$ the interval on $S^1$ between the points $x_1(E)$ and $x_1(E-1)$, and by $I_2(E)$ the interval between the points $x_2(E)$ and $x_2(E-1)$. Denote by $f_{n, E}$ the proectivization of the map $\Pi_{n, E}$. Then if $n\ne 0$, we have $f_{n ,E}(I_1(E))\subset I_1(E)$, and $f_{n, E}^{-1}(I_2(E))\subset I_2(E)$. Moreover,  $f_{n, E}|_{I_1(E)}$ and $f_{n, E}^{-1}|_{I_2(E)}$ are contractions for each $n\ne 0$. For a given $\omega\in \{0,1\}^{\mathbb{Z}}$ there exists exactly one point $z_\omega(E)\in I_1(E)$ such that
$$
z_\omega(E)=\cap_{n\in \mathbb{N}}f_{-n, E}\circ \ldots \circ f_{-1, E}(I_1(E)).
$$
Notice that if the vector $\bar w\in \mathbb{R}^2, |\bar w|=1$, correspond to the direction defined by $z_\omega(E)$, then
$$
\left(\Pi_{-n, E}\ldots \Pi_{-1, E}\right)^{-1}(\bar w)\to 0 \ \ \text{as}\ \ \ n\to \infty,
$$
and for any vector $\bar v \nparallel \bar w$
$$
\left|\left(\Pi_{-n, E}\ldots \Pi_{-1, E}\right)^{-1}(\bar v)\right|\to \infty
$$
exponentially fast as $n\to \infty$. The set $K(E)=\cup_{\omega\in \{0,1\}^{\mathbb{Z}}}z\omega(E)$ is a dynamically defined Cantor set inside of $I_1(E)$. Notice that $\left|f_{{n,E}}'|_{I_1(E)}\right|\sim \frac{1}{E^2}$, and in our regime $E\sim 100$. Hence Hausdorff dimension of $K(E)$ is small, $\text{dim}_H\,K(E)=\text{dim}_B\,K(E)\ll 1/2$.

Similarly, the set
$$
C(E)=\cup_{\omega\in \{0,1\}^{\mathbb{Z}}}\left(\cap_{n\in \mathbb{N}}f_{1, E}^{-1}\circ\ldots\circ f_{n, E}^{-1}(I_2(E))\right)
$$
is a dynamically defined Cantor set, and $\text{dim}_H\,C(E)=\text{dim}_B\,C(E)\ll 1/2$.

A given point $E\in [98, 102]$ is an eigenvalue of an operator $H_\omega$ for some $\omega\in \{0,1\}^{\mathbb{Z}}$ if $f_{0, E}(K(E))\cap C(E)\ne \emptyset$. Now Proposition \ref{p.ex2} follows from the following
statement:
\begin{lemma}\label{l.intersect}
Let $K(E)$ and $C(E)$ be two families of dynamically defined Cantor sets on $\mathbb{R}^1$, $E\in [0,1]$. Suppose that the following properties hold:
\begin{enumerate}
\item The Cantor set $K(E)$ is generated by two $C^1$-smooth (both in $x\in \mathbb{R}^1$ and $E\in [0,1]$) orientation preserving contractions $f_{1, E}, f_{2, E}:\mathbb{R}^1\to \mathbb{R}^1$;
\vspace{4pt}
\item The Cantor set $C(E)$ is generated by two $C^1$-smooth (both in $x\in \mathbb{R}^1$ and $E\in [0,1]$) orientation preserving contractions $g_{1, E}, g_{2, E}:\mathbb{R}^1\to \mathbb{R}^1$;
\vspace{4pt}
\item $\max(K(0))<\min(C(0))$ and $\min(K(1))>\max(C(1))$;
\vspace{4pt}
\item There exists $\delta>0$ such that
$$
\frac{\partial f_{i,E}(x)}{\partial E}>\delta, \ \ \frac{\partial g_{i,E}(x)}{\partial E}<-\delta
$$
for all $E\in [0,1]$, $i=1,2$, and $x\in \mathbb{R}^1$;
\vspace{4pt}
\item We have
$$
\max_{E\in[0,1]}\dim_B\,C(E)+\max_{E\in[0,1]}\dim_B\,K(E)<1.
$$
\end{enumerate}
Then
$$
\left\{E\in[0,1]\ |\ C(E)\cap K(E)\ne \emptyset\right\}
$$
is a Cantor set of box counting dimension not greater than
$$
\left(\max_{E\in[0,1]}\dim_B C(E)+\max_{E\in[0,1]}\dim_B K(E)\right).
$$
\end{lemma}
\begin{proof}[Proof of Lemma \ref{l.intersect}]
Denote
$$
d_K=\max_{E\in[0,1]}\dim_B K(E), \ \ d_C=\max_{E\in[0,1]}\dim_B C(E).
$$
The assumption $(5)$ therefore means that $d_K+d_C<1$.
For any small $\eps>0$ there is a cover of $K(E)$ by $\eps^{-d_K}$ open intervals of length $\eps$, and of $C(E)$ by $\eps^{-d_C}$ intervals. Due to assumptions $(10$, $(2)$, and $(4)$, one can choose those intervals in such a way that each interval of the form $(x(E), x(E)+\eps)$ depends smoothly on $E\in [0,1]$, and $\frac{dx(E)}{dE}>\delta>0$ for intervals covering $K(E)$, and $\frac{dx(E)}{dE}<-\delta<0$ for intervals covering $C(E)$. This implies that the length of an interval in the space of parameters that correspond to a non-empty intersection of a given interval from a cover of $K(E)$ and a given interval from a cover of $C(E)$ is bounded from above by $\frac{2\eps}{\delta}$. Hence the set $\left\{E\in[0,1]\ |\ C(E)\cap K(E)\ne \emptyset\right\}$ can be covered by $\eps^{-d_K}\cdot \eps^{-d_C}=\eps^{-(d_K+d_C)}$ intervals of length $\frac{2\eps}{\delta}=\text{\rm const}\cdot \eps$. Hence $\text{dim}_B\,\left\{E\in[0,1]\ |\ C(E)\cap K(E)\ne \emptyset\right\}\le d_K+d_C$.
\end{proof}

\begin{remark}
 Notice that the question on the structure of the set of translations of one Cantor set that have non-empty intersections with another is closely related to the questions about the structure of the difference of two Cantor sets. Sums (and differences) of dynamically defined Cantor sets were heavily studied, e.g. see \cite{DG1} and references therein. But in our case we needed to work with two Cantor sets that depend on a parameter, so the question about the set of parameters that correspond to a non-empty intersection of the sets cannot be directly reduced to considering the difference of the Cantor sets, and therefore we needed Lemma \ref{l.intersect}.
\end{remark}

\section{Discontinuous upper limit}\label{s:limsup}

Lemma \ref{l:equicontinuous} claims that a sequence of functions $\{\frac{1}{n}L_n(a)\}$ is equicontinuous. This implies that $\limsup_{n\to \infty}\frac{1}{n}L_n(a)$ is a continuous function of $a\in J$. At the same time, Theorem~\ref{t:product} claims, in particular, that $\limsup_{n\to \infty} \frac{1}{n}\left(\log\|T_{n, a, \omega}\|-L_n(a)\right)=0$ for all $a\in J$. So it is tempting to expect that $\limsup_{n\to \infty} \frac{1}{n}\log\|T_{n, a, \omega}\|$ is a continuous function of the parameter $a\in J$, which would be nicely aligned with the fact that Lyapunov exponent is a continuous function of the parameter in the stationary setting.
Nevertheless, here we present an example that shows that this is not always the case.

%\begin{example}\label{ex:residual}
Consider two diagonal matrices with very different norms:
$$
H':= \left( \begin{smallmatrix}
2 & 0
\\
0 & 1/2 \end{smallmatrix} \right),
\quad H'':=
\left( \begin{smallmatrix}
100 & 0
\\
0 & 1/100 \end{smallmatrix} \right).
$$

Take $J$ to be the interval $[0,2\pi]$, and let $\mgr',\mgr''$ be two measures on $C^1(J,\SL(2,\R)),$ obtained in the following way: a random parameter-dependent matrix $A(a)$ w.r.t. each of these measures is the corresponding diagonal matrix, $H'$ or $H''$, precomposed with a rotation by a random uniformly distributed angle $\alpha\in [0,2\pi]$, and postcomposed with the rotation by the parameter. That is,
\begin{equation}\label{eq:A-def}
A(a) = R_{a} \cdot H \cdot R_{\alpha},
\end{equation}
where $H=H'$ for $\mgr'$ and $H=H''$ for $\mgr''$.

It is not hard to see that this choice of the matrices $\{A_j(a)\}$ implies that the functions $L_n(a)$ are in fact independent of $a\in J$, i.e. are constant functions (but certainly depend on $n$). %constants (depending on~$n$). %Note that the distributions of the matrices $A_j(a)$ for all possible parameter values $a\in J$ coincide (with a law that depends only on whether we consider $\mgr_{\alpha}$ or~$\mgr_{\beta}$), and hence the functions $L_n(a)$ are in fact constants (depending on~$n$).

We will take a sequence of times $n_i$, defined recurrently by
$$
n_1=1000, \quad n_{i+1}=10^{10^{n_i}}.
$$
Let the laws for the matrices $A_j(a)$ to be chosen in the following way: we take
$$
\mgr_j=\begin{cases}
\mgr'', & \text{ if } n_i<j\le 2n_i\, \text{ for some } i,\\
\mgr' & \text{ otherwise.}
\end{cases}
$$
%\end{example}
Then, we have the following proposition.
\begin{prop}\label{p:cancellations}
For the random product $A_n(a)\dots A_1(a)$ defined above, %in Example~\ref{ex:residual},
 almost surely there exists a (random)
dense set $X'\subset J$ of parameters, such that one has a strict inequality
$$
\limsup_{n\to\infty} \frac{1}{n} \log \|T_{n,a,\omega}\| < \limsup_{n\to\infty} \frac{1}{n} L_n(a).
$$
In particular, $\limsup_{n\to\infty} \frac{1}{n} \log \|T_{n,a,\omega}\|$ is not a continuous function of $a\in J$, contrary to the continuous (in fact, constant) function $\limsup_{n\to\infty} \frac{1}{n} L_n(a)$.
\end{prop}
\begin{remark}
It is clear that $X'\subset S_e$, where $S_e\subset J$ is a (random) subset of exceptional parameters defined in Theorem~\ref{t:product}. At the same time, one can show that $J\backslash X'$ is a $G_\delta$ subset of $J$, so $(J\backslash X')\cap S_e\ne \emptyset$, and, therefore, $X'$ must be a proper subset of $S_e$.
\end{remark}
\begin{proof}[Proof of Proposition \ref{p:cancellations}]
Let us first calculate the average log-norms $L_n$. Namely, following~\cite{AB}, for any matrix $B\in \SL(2,\R)$ let us consider the ``averaged expansion rate''
$$
N(B)=\frac{1}{2\pi} \int_0^{2\pi} \log \left| B \left( {\cos \gamma \atop \sin \gamma} \right) \right| \, d\gamma.
$$
An easy calculation (see~\cite[Proposition~3]{AB}) shows that
$$
N(B)=N(\|B\|):=\log \frac{\|B\|+\|B\|^{-1}}{2}.
$$

Now, let $H_j$ be the non-random sequence of diagonal matrices defining~$A_j$ in~\eqref{eq:A-def},
and let $q_j$ be the sequence of (non-random) values of $N(A_j)$, that is,
$$
q_j = N(A_j)= N(H_j)= \begin{cases}
N(2), & \text{if } \mgr_j=\mgr',
\\
N(100), & \text{if } \mgr_j=\mgr''.
\end{cases}
$$

%For any initial vector $v$ and a sequence of randomness $\omega=(\alpha_1,\alpha_2,\dots)$ we have
%$$
%T_{\omega,n,a}(v)=
%$$
%and hence (due to~\eqref{eq:phi-sum} and as the rotations do not change lengths)
%$$
%\log |T_{n,a,\omega} v| = \sum_{m=1}^n \varphi_{H_j}(v_{m-1}),
%$$
%where $\varphi_A(v)=\log \frac{|Av|}{|v|}$.

Now, let $v_0$ be a unit vector, and set $v_m=T_{m,a,\omega}v$.
 Then (see~\eqref{eq:phi-sum}) we have
\begin{equation}\label{eq:log-change}
\log |T_{n,a,\omega} v_0| = \sum_{m=1}^n \varphi_{A_m}(v_{m-1}),
\end{equation}
where $\varphi_A(v)=\log \frac{|Av|}{|v|}$. Now,
$$
\varphi_{A_m}(v_{m-1}) = \varphi_{H_m} (R_{\alpha_m} v_{m-1}),
$$
as the rotations do not change the lengths.

%and let $\gamma_m$ be the direction of this image vector (in particular, the initial vector $v_0$ corresponds to the direction $\gamma_0$).

However, for any initial vector $v$ the directions of the vector $R_\alpha(v)$, $\alpha\in [0,2\pi]$, are uniformly distributed. Therefore, taking the expectation of~\eqref{eq:log-change}, we get:
%Hence,
$$
\E \log |T_{n,a,\omega} v_0| = \sum_{m=1}^n \E  \varphi_{H_m}(R_{\alpha_m} v_{m-1}) =
\sum_{m=1}^n  N(H_m) = \sum_{m=1}^n  q_m.
$$
The right hand side does not depend on~$v_0$, and averaging this equality w.r.t. $v_0$ gives us
$$
\E N(T_{n,a,\omega}) = \sum_{m=1}^n  q_m.
$$

Note that for any matrix $A\in \SL(2,\R)$ we have
$$
\log \|A\| \ge N(A) \ge \log \frac{\|A\|}{2},
$$
and thus $|N(A) - \log \|A\| | \le \log 2$. Hence,
$$
|L_n - \sum_{m=1}^n q_m| \le \log 2,
$$
and thus
$$
\limsup_n \frac{1}{n} L_n = \limsup_n \frac{1}{n} \sum_{m=1}^n q_m =  \frac{N(2)+N(100)}{2},
$$
with the values close to the upper limit that are attained for $n=2n_i\cdot (1+o(1))$.

Now, fix a sufficiently small $\eps>0$ (for instance, $\eps=0.001$ will do); then, for all sufficiently large $n$ the conclusions of Theorem~\ref{t:main} for this $\eps$~hold. The mechanism leading to the appearance of the random set $X'$ is the following.

Take any interval $I\subset J$. Due to Lemma \ref{l:large} below and Conclusion~\ref{i:m5} of Theorem~\ref{t:main}, for a sufficiently large $n_i$ from the fast growing sequence $\{n_i\}$ defined above, with very large probability one can find an exceptional interval (in terms of Theorem~\ref{t:main}) in $I$ with some special property. Namely, if we denote that special exceptional (or ``jump'') interval by $J_i$, the corresponding cancelation parameter by $a_i$, and the corresponding jump moment by $m_{i}$, we can assume that
 $$
 (1.5-3\eps) n_i<m_{i}<1.5 n_i.
 $$
Thus the sequence of norms of products $T_{n,a,\omega}$ for $a=a_{i}$ will start decreasing after $n=m_i$, and hence we have
\begin{equation}\label{eq:gr-int}
\max_{n=n_i,\dots, 2n_i} \frac{1}{n} \log \|T_{n,a_{i},\omega}\| \le \frac{2}{3} N(2) + \frac{1}{3} N(100) + \eps.
\end{equation}
Moreover, the same estimate (upon replacing $\eps$ with $2\eps$) holds in a neighborhood $I_i$ of $a_{i}$ of size $10^{-8 n_i}$. Indeed, in such a neighborhood, the directions of any initial vectors stay $\eps$-close during $2n_i$ iterations, and one application of a matrix $A$ cannot increase any angle by more than $\|A\|^2\le 100^2$ times.

Notice that  between the moments $2n_i$ and $n_{i+1}$ we apply only norm 2 matrices, so one can easily see that  %have $q_m=N(2)$, thus (due to the upper estimate for the growth)
\begin{equation}\label{eq:gr-ext}
\forall a\in I_i \quad \max_{n=2n_i,\dots, n_{i+1}} \frac{1}{n} \log \|T_{n,a,\omega}\| \le \max (N(2) + 6\eps N(100) + 3\eps, \log 2) =\log 2;
\end{equation}
the term $6\eps N(100)$ here has to be added,  since the cancelation moment $m_{i}$ is not perfectly at the center $1.5 n_i$ of the interval $[n_i,2n_i]$, so we need to include the ``worst case scenario'' of the application of a growing sequence of norm 100 matrices along the interval $[(1-6\eps)n_i, 2n_i]$. % In fact, moreover, as we have $q_m=N(2)$ for all such indices, due to the upper estimate for the growth we can take $N(2) + 6\eps N(100) + 3\eps$ as an upper bound, but we will not need it.

While the interval $I_i$ of size $10^{-8n_i}$ was ``small'' for $2n_i$ iterations (that is, the corresponding
norms of the matrix products behaved similarly), it becomes ``large'' for $n_{i+1}=10^{10^{n_i}}$ iterations. Namely, due to Lemma~\ref{l:large} below, on the interval $I_i$  one can find (with the probability extremely close to~$1$) a new jump subinterval $J_{i+1}\subset I_i$ with the
corresponding cancelation point $a_{i+1}\in J_{i+1}$, such the corresponding jump moment $m_{i+1}$ satisfies
$(1.5-3\eps) n_{i+1}<m_{i+1}<1.5 n_{i+1}$.

Again, we find an interval $I_{i+1}$ around the point $a_{i+1}$ of size $10^{-8n_{i+1}}$ where the cancelation estimates~\eqref{eq:gr-int},~\eqref{eq:gr-ext} hold (with $i$ replaced by $i+1$), etc. Continuing this procedure, we find a sequence of decreasing intervals
$$
I \supset I_i \supset I_{i+1} \supset I_{i+2} \supset \dots,
$$
such that
$$
\forall j \quad \forall a \in I_j \quad \max_{n=n_j,\dots, n_{j+1}} \frac{1}{n} \log \|T_{n,a,\omega}\| \le \frac{2}{3} N(2) + \frac{1}{3} N(100) + \eps.
$$
Hence, taking $a$ to be the intersection point of all the $I_j$, we get
$$
\limsup_{n\to\infty} \frac{1}{n}\log \|T_{n,a,\omega}\| \le \frac{2}{3} N(2) + \frac{1}{3} N(100) + \eps.
$$
As we have started with an arbitrary initial interval $I$, the constructed points $a$ form a dense set $X'$ in the interval of parameters.

\begin{figure}
\includegraphics{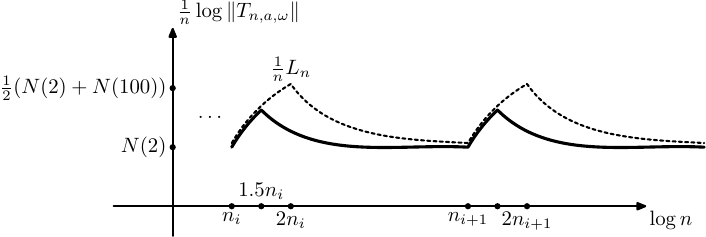}
\caption{Behavior of the sequence $\frac{1}{n}\log\|T_{n, a, \omega}\|$ for a parameter $a\in X'$ (bold line), compared with the one of $\frac{1}{n} L_n$ (dashed line).}\label{f:Lnocont}
\end{figure}

Thus, the proof of Proposition~\ref{p:cancellations} will be concluded, once we show that the probability of making each new step in the construction of the sequence of intervals $I_j$ is sufficiently high so that Borel-Cantelli Lemma can be applied. Finding a jump interval,  due to Conclusion~\ref{i:m5} of Theorem~\ref{t:main}, can be stated in terms of the corresponding lift iterations and their differences~$R_{m,\omega}(I)$. Namely, for finding a jump interval inside $I$ (assuming, that Conclusion~\ref{i:m5} holds)  with a jump index $m_i\in [(1.5-3\eps) n_i, 1.5 n_i]$ it suffices to establish the inequality
$$
R_{(1.5-3\eps)n_i,\omega}(I) + 3 <  R_{(1.5-2\eps)n_i,\omega}(I).
$$
Indeed, due to Conclusion~\ref{i:m5} of Theorem~\ref{t:main} (applied with $n=2n_i$) we know that number of exceptional intervals in $I$ with jump moment not greater than $1.5n_i$ is bounder from below by $R_{(1.5-2\eps)n_i, \omega}-1$. On the other hand, the number of exceptional intervals in $I$ with jump moment less than $(1.5-3\eps)n_i$ is at most $R_{(1.5-3\eps)n_i, \omega}$.
%allows to look for a jump interval in terms of the corresponding

Therefore, it is enough to establish that following statement.

\begin{lemma}\label{l:large}
For any $i$, let $\alpha_1,\dots,\alpha_{n_i}$ and an interval $I\subset J$ of length $|I|$
be given. Then, with the probability at least $1-\frac{12}{\eps n_i | I |}$ one has
$$
R_{(1.5-3\eps)n_i,\omega}(I) + 10 <  R_{(1.5-2\eps)n_i,\omega}(I).
$$
\end{lemma}
\begin{proof}
Denote $n'_1:=(1.5-3\eps)n_i$, $n'_2:=(1.5-2\eps)n_i$, and assume, additionally to the assumptions of the lemma, that $\alpha_n$ is given for all $n\le n'_1$. This defines the intermediate images $y_-=\tf_{n_1', b'_-,\omega}(\tx_0)$ and $y_+=\tf_{n_1',b'_+, \omega}(\tx_0)$, where $I=[b'_{-}, b'_{+}]$.
To prove Lemma \ref{l:large} it is enough to show that
\begin{equation}\label{eq:ypm}
\tf_{[n'_1, n'_2],b'_+, \omega}(y_+)- \tf_{[n'_1, n'_2],b'_-,\omega}(y_-) \ge (y_+ - y_-) +10
\end{equation}
with the claimed probability. This is exactly what we are going prove.

Note that as all the lifts $\tf$ commute with the shift by~$1$, inequality (\ref{eq:ypm}) is preserved if one shifts $y_-$ by any integer. Hence, without loss of generality we can assume that $y_-<y_+<y_-+1$.

Notice that due to the choice of the matrices in (\ref{eq:A-def}), for any fixed parameter $a\in J$ (in particular, for $a=b_-$) the Lebesgue measure on the circle of directions is stationary with respect to the inverse maps $f_{a,\omega}^{-1}$. Indeed, we defined $A(a)=R_a\circ H\circ R_{\alpha}$ (where $H$ is either $H'$ or $H''$), so $A(a)^{-1}=R_{-\alpha}\circ H^{-1}\circ R_{-a}$.  The rotation $R_{-\alpha}$ is a rotation by a random angle $-\alpha$ uniformly distributed in $[0,2\pi]$, hence the image of any initial point is uniformly distributed on the circle. A standard argument in random dynamics (see \cite{A}, \cite{DKN1},  or \cite{KN}), based on the ideas going back to Furstenberg's work \cite{Fur3}, implies that
$$
\forall\  a\in J, \quad \forall\  [y', y'']\subset \R  \quad \E |\tf_{a, \omega}([y', y''])| = |[y', y'']|.
$$
As lifts $\tf_{b'_-, \omega}$ and $\tf_{b'_+,\omega}$ of $1$-step maps differ by a translation by $|I|=b'_+ -b'_-$, this implies that
%\begin{equation}\label{}
\begin{multline}\label{eq:submartingale}
    \forall \ [y', y'']\subset \R  \quad \E \left( \tf_{b'_+, \omega}(y'')-  \tf_{b'_-, \omega}(y') \right) =\\
    \E\left((\tf_{b'_+, \omega}(y'')-\tf_{b'_-, \omega}(y'')) +(\tf_{b'_-, \omega}(y'')- \tf_{b'_-, \omega}(y'))\right) = |I|+(y''-y'). %( b'_+ - b'_-).
\end{multline}
%\end{equation}
Consider now the random process given by the lengths of the corresponding images %after $s$
under iterations following the initial moment $n'_1$:
$$
\eta_{m}(\omega):= \tf_{[n'_1, m],b'_+, \omega}(y_+)- \tf_{[n'_1, m], b'_-, \omega}(y_-),
$$
where the sequence of random parameterized matrices $\omega$ is defined by the sequence of angles $\alpha_1, \alpha_2, \ldots$.
Then~\eqref{eq:submartingale} then becomes a submartingale relation
$$
\E (\eta_{m+1} \mid \alpha_1,\dots, \alpha_m) = \eta_{m-1} + |I|.
$$

It suffices now to apply standard submartingales technique. Namely, consider a stopping time $T=T(\omega)$, defined as
%the first iteration at which the interval between the respective images of $y_-$ and of $y_+$ separate by more than~$11$ if it happens somewhere on $[n'_1;n'_2]$, and equal to the last allowed iteration $n'_2$ otherwise:
$$
T(\omega) = \min \left( \{ m \in [n'_1, n'_2] \mid  \eta_m \ge 11 \}  \cup \{n'_2\} \right).
$$
Then, one has
$$
\E \eta_{T(\omega)} = \eta_{n'_1} + |I| \cdot \E ( T(\omega) - n'_1).
$$
However, $\E\eta_{T(\omega)}< 12$ (due to the choice of $T(\omega)$ as the first moment at which the length exceeds~$11$), and hence
$$
\E ( T(\omega) - n'_1) \le \frac{12}{| I | }.
$$
Applying Chebyshev inequality, we see that
$$
\P(T(\omega)<n'_2) \le \frac{12}{(n'_2-n'_1) \cdot | I | } =\frac{12}{\eps n_i \cdot | I |},
$$
and the event $T(\omega)<n'_2$ exactly means that for some intermediate iteration $m\in [n'_1, n'_2]$ the length of the interval
$\left( \tf_{[n'_1, m],b'_+,\omega}(y'')-  \tf_{[n'_1, m],b'_-,\omega}(y') \right)$ exceeds $11$, and thus for all the consecutive iterations (in particular, $n'_2$) is at least $10+(y''-y')$. We have established~\eqref{eq:ypm}, thus completing the proof of Lemma~\ref{l:large}.
\end{proof}

Lemma~\ref{l:large} is proven, and this concludes the proof of Proposition~\ref{p:cancellations}.
%\begin{itemize}
%\item Norms of compositions: explicit calculation of $L_n$. Namely, increment by almost $c_1\sim \log 2$ for $\mgr_{\alpha}$ and almost by $c_2\sim \log 100$ for $\mgr_{\beta}$.
%
%Ref. to~\cite[Proposition~3]{AB}.
%
%\item Limsup of $\frac{1}{n}L_n$: exactly at the moments $2n_k$.
%\item Jump intervals between $1.4 n_k$ and $1.5 n_k$. Lemma: if yes, then $\limsup$ is small.
%\item Lemma: for any interval $I\subset J$ the probability of finding a jump interval inside it in $k$ iterations is $1-\dots$
%\end{itemize}
\end{proof}

\end{appendix}

%\section{Sparse random potentials}
%
%\section{Non-linear case: random dynamics on $S^1$}

\section*{Acknowledgments}

We are grateful to David Damanik, Lana Jitomirskaya, and Abel Klein for providing numerous references in spectral theory and many useful remarks.

\end{document}